\documentclass[10pt]{article}
\usepackage{hyperref}
\usepackage{amsmath, graphicx, latexsym, amssymb, amsthm, amsfonts}
\usepackage[mathscr]{eucal}
\usepackage{graphics}
\usepackage{color}
\usepackage{epsfig}
\usepackage{enumerate}
\usepackage{mathtools}
\usepackage{subcaption}
\usepackage{comment}
\usepackage{tikz}
\usetikzlibrary{arrows.meta,positioning,calc}

\mathtoolsset{showonlyrefs}

\newtheorem{lemma}{Lemma}[section]
\newtheorem{theorem}[lemma]{Theorem}

\newtheorem{proposition}[lemma]{Proposition}

\newtheorem{assumption}[lemma]{Assumption}

\newtheorem{remark}[lemma]{Remark}

\allowdisplaybreaks

\newcommand{\eps}{{\varepsilon}}

\newcommand{\NN}{{\mathbb{N}}}

\newcommand{\RR}{{\mathbb{R}}}
\newcommand{\EE}{{\mathbb{E}}}
\newcommand{\PP}{{\mathbb{P}}}

\newcommand{\cO}{{\mathcal{O}}}

\newcommand{\clb}{{\mathcal{B}}}

\newcommand{\clp}{{\mathcal{P}}}
\newcommand{\clc}{{\mathcal{C}}}
\newcommand{\cld}{{\mathcal{D}}}
\newcommand{\cle}{{\mathcal{E}}}
\newcommand{\clf}{{\mathcal{F}}}

\newcommand{\cls}{{\mathcal{S}}}

\newcommand{\one}{{\boldsymbol{1}}}

\newcommand{\xbd}{{\boldsymbol{x}}}
\newcommand{\ybd}{{\boldsymbol{y}}}

\newcommand{\Om}{\Omega}
\newcommand{\om}{\omega}

\newcommand{\bfX}{\mathbf{X}}
\newcommand{\bfY}{\mathbf{Y}}
\newcommand{\bfZ}{\mathbf{Z}}

\newcommand{\bfU}{\mathbf{U}}

\newcommand{\bfz}{\mathbf{z}}
\newcommand{\bfx}{\mathbf{x}}
\newcommand{\bfy}{\mathbf{y}}

\newcommand{\cll}{\mathcal{L}}
\newcommand{\clg}{\mathcal{G}}

\newcommand{\clm}{\mathcal{M}}

\newcommand{\MM}{\mathbb{M}}
\newcommand{\mn}{\varsigma}
\newcommand{\vt}{\vartheta}
\newcommand{\be}{\begin{equation}}
\newcommand{\ee}{\end{equation}}
\newcommand{\beq}{\begin{align*}}
\newcommand{\eeq}{\end{align*}}

\newcommand{\di}{\color{black}}

\numberwithin{equation}{section}
\begin{document}

\title{Flocking under Fast and Large Jumps: Stability, Chaos, and Traveling Waves}
\author{Sayan Banerjee, Amarjit Budhiraja and Dilshad Imon}
\date{}
\maketitle

\begin{abstract}
\noindent 
We study a pure-jump $n$-particle system with attractive mean field interactions under which each particle jumps forward by a random amount, independently sampled from a given distribution $\theta$, at exponentially distributed times with rate given by a function $w$ of its signed distance from the system center of mass. The function $w$ is taken to be non-increasing which leads to a `flocking' behavior: the particles below the center of mass jump forward at a higher rate than those above it. This model was introduced in  \cite{Balzs2011ModelingFA} and some of its properties were studied for the case when $w$ is bounded. In the current work we are interested in the setting where $w$ is unbounded, and this feature, together with the mild integrability we impose on the jump sizes, results in a stochastic dynamical system for interacting particles with fast and large jumps for which little is available in the literature. We identify natural conditions under which the system is well-posed and study the large $n$ limit (the so-called `fluid limit') of the empirical measure process associated with the system. By establishing well-posedness of the associated McKean-Vlasov equation we characterize the fluid limit of the particle system and prove a propagation of chaos result. Next, for the centered $n$-particle system, by constructing suitable Lyapunov functions, we establish existence and uniqueness of stationary distributions and study their tail properties. In the special case where $w$ is an exponential function and $\theta$ is an exponential distribution, by establishing that all stationary solutions of the McKean-Vlasov equation must be the unique fixed point of the equation, we prove a propagation of chaos result at $t=\infty$ and establish convergence of the particle system, starting from stationarity, in the large $n$ limit, to a traveling wave solution of the McKean-Vlasov equation. The proof of this result may be of interest for other interacting particle systems where convexity properties or functional inequalities generally used for establishing such a result are not available. Our work answers several open problems posed in \cite{Balzs2011ModelingFA}.
\noindent\newline


\end{abstract}

\section{Introduction}
Consider $n$ particles moving forward on the real line, at independent exponentially distributed times with rates given by a function $w$ of its signed distance from the center of mass of the system, and with jump sizes that are iid distributed as $\theta$. The jump rate function $w: \RR \to (0,\infty)$ is assumed to be non-increasing which produces an attractive `flocking' behavior in the system under which the particles below the center of mass jump at a higher rate than those above it. Such models of attractive interacting particle systems arise in many contexts and some of these are discussed later in the introduction. This specific model was introduced in \cite{Balzs2011ModelingFA}, as a model in mathematical finance and for different types of biological flocking phenomena, who studied some basic properties of the system and formulated several interesting open problems. One of the main results of \cite{Balzs2011ModelingFA} shows that when $w$ is bounded, differentiable, with bounded derivative, under some natural mild conditions on the jump sizes, the empirical measure process associated with the particle system converges in probability, as $n\to \infty$, in a suitable path space, to the unique solution of a nonlinear evolution equation that is the natural McKean-Vlasov equation associated with the interacting particle system. For the proof of this `fluid limit' convergence result, as the authors note in their work, all three key steps -- tightness, identifying the limit, uniqueness of the limit-- crucially rely on the boundedness of $w$ and the authors formulate the interesting open problem of treating the unbounded case, e.g. when $w(x) = e^{-x}$, $x\in \RR$. 

One of the first objectives of the current work is to answer this question, with a particular focus on the case $w(x)=e^{-x}$. We remark that when $w$ is unbounded, the situation is very different as particles lagging far behind the mean may jump at  very high rates resulting in a very fast increase of the mean process, which in turn may rapidly increase particle jump rates. Understanding and analyzing this complex feedback phenomenon is one of the key technical challenges of this model. Indeed, even the well-posedness (e.g. non-explosiveness) of the model is not immediate without additional conditions. In Assumption \ref{Z} we formulate a natural condition that suitably balances the jump sizes and jump rates and which ensures the well-posedness of the Markov process describing the evolution of the particle system. Our first goal is to study the large $n$ limit of the empirical measure process (i.e. the fluid limit). A key to proving such a limit result is establishing suitable tightness properties of the (empirical) mean process $\{m_n(t), t \ge 0\}$ for the particle system which requires a detailed probabilistic understanding of different types of jumps of the particles relative to their center of mass. The second main ingredient here is the uniqueness of solutions of the associated McKean-Vlasov equation (see \eqref{eq:mfeq1}) in a suitable class that is needed in order to characterize the weak limit points, once tightness is established. The standard arguments rely on the boundedness and Lipschitz property of $w$ (see e.g.  \cite[Theorem 6.14]{Balzs2011ModelingFA}) whereas for the basic example that we are interested in, namely $w(x) =e^{-x}$, neither of these properties are satisfied. Our proof of uniqueness relies on establishing two important facts: (a) To any solution of the McKean-Vlasov equation, with a given initial condition,  one can associate a unique in law `nonlinear Markov process' ({\di see  Section \ref{sec:modres}, immediately above Theorem \ref{mv}, for a precise definition}); (b) Any two nonlinear Markov processes starting from the same initial condition, under a suitable coupling, are identical. This approach is inspired by \cite{Oelschlager1984AMA}, however,  the proof details that require carefully handling the poor available control on $w$ are substantially more demanding. Using this uniqueness result and suitable tightness estimates we obtain convergence of the empirical measure process to the deterministic limit given by the solution of the McKean-Vlasov equation. By standard results this also immediately yields a propagation of chaos result for the $n$-particle system. (See Theorem \ref{thm.mainFL} for a precise statement of these results).

A second topic of interest in this work is the long-time behavior of the particle system.
Note that as the jumps only move the particles forward, in order to study long-time behavior, it is natural to consider the particle system centered by its center of mass. The paper \cite{Balzs2011ModelingFA} (under assumptions on $w$ noted previously) studied the stability properties of the (centered) particle system when $n=2$ and identified the explicit form of the stationary distribution when the jumps are deterministic or their distribution has a density (see Section 4 therein). Behavior for $n\ge 3$ was left as an open problem. In the current work we establish under suitable conditions-- that cover the setting of  \cite{Balzs2011ModelingFA} with a general $n$ but also include the more demanding case where $w$ is unbounded (and not globally Lipschitz) -- existence and uniqueness of stationary distributions $\pi_n$ of the $n$-dimensional Markov process describing the evolution of the (centered) particle system. Proofs are based on constructing suitable Lyapunov functions that provide quantitative estimates capturing the  attractive behavior suggested by the description of the particle system. These Lyapunov functions also give useful uniform in $n$ tail estimates for the stationary distributions that are needed to understand the large $n$ behavior of these stationary distributions.

Our final goal in this work is to rigorously relate traveling wave solutions of the McKean-Vlasov equation \eqref{eq:mfeq1} with the large $n$ and large time behavior of the underlying particle system. By a traveling wave solution of \eqref{eq:mfeq1} we mean a flow of deterministic measures $\{\mu(t)\}_{t\ge 0}$ where $\mu(t,dx) = \rho(x-ct) dx$ for a suitable function $\rho$ and a constant $c$. The function $\rho$ describes the shape of the traveling wave whereas the constant $c$ gives its speed. The paper \cite{Balzs2011ModelingFA} considers the case where the jump distribution is exponential and 
shows that, for any non constant $w$, there is a unique continuously differentiable traveling wave solution of the McKean-Vlasov equation and in the specific case $w(x) = e^{-\beta x}$, where $\beta >0$, identifies the exact form of the traveling wave function $\rho$ in terms of a Gumbel density function and gives a formula for $c$. However \cite{Balzs2011ModelingFA} leaves open the issue of a rigorous connection between the finite particle system and the traveling wave solution. Indeed their finite time convergence analysis for the fluid limit, which requires $w$ to be bounded, does not cover the case when $w(x) = e^{-\beta x}$. Moreover, to connect the long-time behavior of the fluid limit with the large $n$ behavior of the stationary particle system, one needs  to establish a certain `interchange of limits' property that says that the limits $n \to \infty$ and $t \to \infty$ commute to give the same limiting object. In the current work, we show that, under conditions, the sequence of probability laws of the $n$-particle empirical measures under stationary distribution $\pi_n$ is tight and any weak limit point is a stationary solution of the centered McKean-Vlasov equation. In addition, if $\theta$ is an exponential distribution with rate $\gamma$ and $w(x) = e^{-\beta x}$ for some $0< \beta \le \gamma <\infty$, then the centered version of the McKean-Vlasov equation (see \eqref{FL}) has a unique fixed point $\nu^*$ and furthermore, any random stationary solution of this equation must in fact be non-random and {\di it must be equal to} $\nu^*$ at all times. This fixed point can be identified with the traveling wave solution of \cite{Balzs2011ModelingFA}. 
Together these results say that the empirical measures under stationary distributions $\pi_n$ must converge to the unique fixed point $\nu^*$ in probability making a precise connection with the traveling wave solution of \cite{Balzs2011ModelingFA}. One immediate consequence of this result is the propagation of chaos at $t=\infty$ which says that, for any $k\ge 1$, the first $k$-dimensional marginal of $\pi_n$ converges to $(\nu^*)^{\otimes k}$ as $n \to \infty$. As another consequence of this result we show that the asymptotic speed of the $n$-particle flock, as measured by $\lim_{t\to \infty} m_n(t)/t$, converges to the speed of the traveling wave identified in \cite{Balzs2011ModelingFA}, as $n\to \infty$. This gives a second rigorous connection between the finite particle system and the associated traveling wave solution. A crucial ingredient in proving these results is establishing the fact that any stationary solution of the centered McKean-Vlasov equation must equal the unique fixed point of the equation. Establishing such a property for more general $w$ and $\theta$ remains a challenging open problem.

We remark that in general the analysis of long-time behavior and steady states of 
McKean-Vlasov equations is challenging as even with strong stability properties of the underlying particle system this nonlinear evolution equation may have multiple fixed points and a complex local stability structure.  
Even in cases where there is a unique fixed point for this evolution equation, rigorously relating this fixed point  to the stationary distribution of the $n$-particle system is in general not straightforward. 
One basic approach for this is by establishing \emph{uniform-in-time} propagation of chaos (POC), which says that, for any $k\ge 1$,  a suitable distance between the law of the $k$-marginal of the $n$-dimensional particle system at time instant $t$ and the $k$-fold product measure constructed from the time $t$ solution of the    McKean-Vlasov equation converges to $0$, as $n\to \infty$, uniformly in $t$. However uniform-in-time POC has only been established in a handful of models, mostly diffusion processes, which satisfy strong convexity assumptions, or in the presence of functional inequalities like Poincar\'e and Log-Sobolev inequalities \cite{jourdain2008,jourdain2013propagation,BuFan,durmus2020elementary,lacker2023sharp}.
In contrast, the approach taken in the current paper bypasses such uniform-in-time POC estimates and it would be interesting to see if the approach can be implemented for other particle systems lacking strong convexity properties or where functional inequalities are not available.  On the other hand, a limitation of the approach taken here is that it does not readily yield  any convergence rates.

This work concerns a class of interacting stochastic particle systems in which the source of the interaction is the relative positions  of the particles in the system. The phrase {\em particle systems with topological interaction} has been coined by \cite{CDGP} to describe these types of models which arise in many application areas including ecology, evolutionary biology, engineering, and mathematical finance. One of the most well studied such system is the Atlas model (a special case of `rank-based' diffusions), that describes a system of competing Brownian particles, which arises from problems in stochastic portfolio theory \cite{ferkar}. Mathematically, the $n$-dimensional system is modeled by a $\RR^n$-valued diffusion for which each coordinate performs an independent Brownian motion on the real line except for the lowest ranked particle which receives, in addition,  a drift of $\gamma >0$. This positive drift creates an attractive interaction which leads to a very rich long-time theory for this system, both for the case when $n<\infty$ and $n=\infty$ (cf. \cite{palpit, sar, sartsa, DJO, demtsa, banerjee2022domains, IPS, jourdain2008,shkolnikov2012large,jourdain2013propagation, andreis2019ergodicity} and references therein for  a small sampling of results on this topic and related themes). 
From problems motivated by evolutionary biology, several similar models of interacting particle systems have been considered, cf. \cite{bebrnope, bebrnope2, berzha, bruder, durrem, atar2020hydrodynamics}. For example, \cite{durrem} studies a collection of $n$-particles on the real line in which each of the particles gives birth at rate $1$ to a new 
 particle whose position $y$, relative to its 
 parent’s position $x$ is given through a probability distribution $\rho(x-y)dy$ where 
$\rho$ is a probability density on $\RR$. Whenever a new particle is born, the   current leftmost particle is removed. Among other results, authors study traveling wave solutions  and identify the asymptotic speed (as $t$ and $n$ become large) of the particle system. Analogous attractive interacting particle systems  arise from load-balancing problems in queuing systems in which the incoming jobs are routed to shorter queues in some manner to keep the overall queue-lengths appropriately balanced (cf. \cite{mitzenmacher2001power,lu2011join,der2022scalable,banerjee2023load}). There is also an extensive literature on lattice-based attractive interacting particle systems, e.g. the voter model, contact process, exclusion and zero-range processes (cf. \cite{kipnis2013scaling}). A Gaussian pursuit model where attractive behavior results from particles following their immediate leader is analyzed in \cite{banerjee2016brownian}. Several other models have been studied in the physics literature \cite{ben2007toy,hongler2013local} where particles undergo local interactions (eg. follow the immediate leader)  as well as global interactions (eg. follow the global leader and barycenter-modulated dynamics) to produce emergent attractive behavior. 

A closely related particle system, motivated by  applications in 
distributed parallel simulation, has been studied in \cite{greenberg1996asynchronous,stolyar2023particle,stolyar2023large}.
In this model, each particle in the $n$-particle system jumps forward by a random amount, independently sampled from a common distribution $\theta$, with rate given by a non-increasing function $\eta_n: [0,1] \to \RR_+$ of its \emph{quantile} in the empirical distribution of the system. The compact domain of $\eta_n$ allows for minimal regularity assumptions and \cite{stolyar2023particle} obtains a fluid limit  for the empirical measure process  under the assumption that $\eta_n$ converges uniformly to a strictly decreasing continuous function $\eta$ on $[0,1]$ and additional assumptions on the jump distribution $\theta$. The bounded rate of jumps and special quantile-based form of the associated McKean-Vlasov equation make it amenable to long-time analysis and \cite{greenberg1996asynchronous,stolyar2023large} prove a $L^1$-contraction property in time of the distance between solutions started from different initial conditions. This strong form of asymptotic stability gives convergence to a unique steady state of the centered McKean-Vlasov equation, given by a traveling wave. Furthermore, exploiting this asymptotic stability property of the McKean-Vlasov equation, it is shown in \cite{stolyar2023particle} that the empirical distribution for the stationary  $n$-particle system converges to this traveling wave, under suitable assumptions, as $n \to \infty$. 
 We note that there are some key differences between this model and the one considered in the current work. First, in our model, the unboundedness of the jump-rates is a key technical challenge and the behavior of $w(-x)$ for large $x$ (which grows exponentially in our basic example of interest) plays a crucial role in the system evolution. Second, the form of the dynamics  in \cite{greenberg1996asynchronous,stolyar2023particle,stolyar2023large} yields that the mean of the solution of the McKean-Vlasov equation evolves linearly in time  which enables an analysis in `moving frames'  facilitating study of long-time behavior. In our model, the evolution of the mean function in the McKean-Vlasov equation is not linear and does not take a simple explicit form in general and therefore these methods cannot be used. Finally, the asymptotic stability property is key to the analysis in these papers which readily gives that a stationary solution must be the same as the unique fixed point. In our case we are unable to obtain this type of a strong contractivity property. Indeed proving that all stationary solutions must equal the unique fixed point is one of the main challenges in extending our results to more general $w$ and $\theta$.

 \subsection{Organization of the Main Results}

The following are the main results of this work.

(i) \textbf{Fluid limit and POC: } In Theorem \ref{thm.mainFL}, we establish the fluid limit for the empirical measure process under suitable conditions on $w$ and jump distributions. We also give a pointwise propagation of chaos result. The proof requires several new ideas, noted below the theorem. 
These include, in particular, suitably controlling the empirical mean process $m_n(\cdot)$ uniformly in $n$ and establishing uniqueness of solutions to the McKean-Vlasov equation. 

(ii) \textbf{Ergodicity of the finite particle system: }By constructing suitable Lyapunov functions, we deduce in Theorem \ref{stat.ex} the existence of a stationary distribution $\pi_n$ for the centered $n$-particle system (viewed from the centre of mass $m_n(\cdot)$) for each $n \in \NN$, along with uniform-in-$n$ bounds on the stationary expectation of certain functions which provide uniform estimates on the tail behavior of $\pi_n$. The stationary distribution is shown to be unique under mild natural `spread-out' assumptions on the jump distribution. It is also shown that $m_n(t)/t$ converges as $t \to \infty$ for fixed $n$ and the limit is identified in terms of $w$ and $\pi_n$. A useful uniform integrability property under $\pi_n$ is obtained in Theorem \ref{stat.pr}.

(iii) \textbf{POC at time $t=\infty$ and traveling waves: } Theorems \ref{stat.ex} and \ref{stat.pr} establish the integrability required for $\pi_n$ for the particle system started from law $\pi_n$ to satisfy the hypotheses of Theorem \ref{thm.mainFL} (a) and (b) which give tightness of the empirical measure process and characterize the subsequential weak limit points of this process. Taking a subsequential fluid limit then produces a solution to the McKean-Vlasov equation started from a (possibly random) measure which is obtained as the (subsequential) limit of the empirical measures under $\pi_n$. After centering these solutions $\mu(t)$ by the limiting center of mass $m(t)$ at time $t$, we obtain a flow of probability measures $\{\nu(t)\}$ which is necessarily stationary, namely  the law of $\nu(t)$ is the same for all $t$. Hence, if one can prove in addition that the fluid limit equation has a fixed point $\nu^*$ such that any stationary solution $\nu(\cdot)$ is, in fact, equal to $\nu^*$ (i.e.  $\nu(t)= \nu^*$ for all $t$), it follows that the empirical measures under $\pi_n$ converge weakly to $\nu^*$ as $n \to \infty$ and POC is established at time $t=\infty$ (see Theorem \ref{thm.stat.sol}). The last condition is explicitly verified in Theorem \ref{thm.ss.exp} for the case when the jump distribution is Exponential with rate $\gamma$ and $w(x) = e^{-\beta x}$ for some $0 < \beta \le \gamma < \infty$. $\nu^*$ in this case corresponds to a Gumbel-type distribution \eqref{eq:838} that arose as a traveling wave solution in \cite{Balzs2011ModelingFA}. These techniques also allow us to identify the asymptotic velocity  $\lim_{n \to \infty} \lim_{t \to \infty} \frac{m_n(t)}{t}$. These results are stated in Theorems \ref{thm.stat.sol} and \ref{thm.ss.exp}.

\subsection{Notation}\label{sec:notat}
We denote a $k$-dimensional  vector $(x_1, \ldots x_k) \in \RR^k$ as $\bfx$. 
The set $\{1, \ldots , n\}$ will be denoted as $[n]$.
For a Polish space $\cls$,
$C_b(\cls)$ (resp. $\MM_b(\cls))$ denotes the class of all real bounded continuous (resp. measurable) functions on $\cls$.  Borel $\sigma$-field on a Polish space $\cls$ will be denoted as $\clb(\cls)$. Space of probability measures on $(\cls, \clb(\cls))$ will be denoted as $\clp(\cls)$ and equipped with the topology of weak convergence.
For a random variable $X$ with values in a Polish space $\cls$, 
$\mathcal{L}(X)$ denotes the probability law of $X$, and given a sub $\sigma$-field $\clg$, $\mathcal{L}(X \mid \clg)$ will denote the conditional law of $X$ given $\clg$ (which is a $\clg$-measurable $\clp(\cls)$ valued random variable). For probability measures $\mu, \nu \in \clp(\RR)$ we use the notation $\mu \le_d \nu$ to denote that $\mu$ is stochastically dominated by $\nu$.
$\text{Lip}_1$ denotes the class of all Lipschitz functions from $\RR$ to $\RR$ with Lipschitz constant bounded by $1$. 
$ \text{Id}$ denotes the identity function on $\RR$, i.e. $\text{Id}(x)=x$, for every $x \in \RR$. 
For the sake of convenience, we will denote the square of the positive part of $x \in \RR$, $[(x)^+]^2$ as $(x)^{+^2}$. 
Given measurable $\psi: \cls_1 \times \cls_2 \to \RR$,  a $\cls_1$ random variable $X$, and  $\mu \in \clp(\cls_2)$, we will write $\int \psi(X, y) \mu({\di dy})$ (when the integral makes sense) as $\EE_Z(\psi(X, Z))$ where $Z$ is a $\cls_2$ valued random variable independent of $X$ with $\cll(Z)= \mu$.
For a Polish space $\cls$, we denote by $\mathcal{D}([0,\infty):\cls)$ the space of functions $f: [0,\infty) \rightarrow \cls$ which are right-continuous and have finite left-limits (RCLL) endowed with the usual Skorokhod topology.
The space of continuous functions $f: [0,\infty) \rightarrow \cls$ equipped with the topology of local uniform convergence will be denoted as $\mathcal{C}([0,\infty):\cls)$. For a given $T>0$, the spaces $\mathcal{D}([0,T]:\cls)$ and $\mathcal{C}([0,T]:\cls)$ are defined similarly.
A sequence $\{X_n\}$ of random variables with values in $\mathcal{D}([0,\infty):\cls)$ is said to be $\clc$-tight if the corresponding sequence of probability laws is relative compact and any limit point of the latter sequence is supported on $\mathcal{C}([0,\infty):\cls)$.
  $\mathcal{P}_1(\RR)$ will denote the subset of $\mathcal{P}(\RR)$ consisting of all $\mu$ for which $\int_\RR |x| \mu(dx) < \infty $. 
This space is equipped with the \textit{Wasserstein} $1$ metric, denoted by $\mathcal{W}_1$, and  defined as follows. For measures, $\mu$, $\nu \in \mathcal{P}_1(\RR)$, 
$$\mathcal{W}_1(\mu, \nu)\doteq \inf_{\pi\in \Pi(\mu,\nu)} \int_{\RR^2} |x-y|\,\pi(dx,\, dy), $$ where $\Pi(\mu,\nu)$ denotes the set of all probability measures on $\RR^2$ with first marginal as $\mu$ and second marginal as $\nu$ (i.e. the set of all couplings of $\mu$ and $\nu$). Equivalently, 
$$\mathcal{W}_1(\mu, \nu)= \sup_{\substack{ f\in\text{Lip}_1}} \int_\RR f(x) (\mu-\nu)(dx).$$

For $\cls$ valued random variables $X_n, X$, we denote the convergence in distribution of $X_n$ to $X$ by $X_n \Rightarrow X$. 

For $\theta_0 \in \clp(\RR_+)$,  a $\RR_+$ valued stochastic process $\{V(t), t\ge 0\}$, with $V(0)=0$ will be called a $\mathcal{J}(\theta_0)$-jump process if it is a Markov process with generator
$$\cll^V f(v) = \int f(v+y) \theta_0(dy) - f(v), \; v \in \RR_+, f \in \MM_b(\RR_+).$$

We use $\lceil \cdot \rceil: \RR \to \RR$ to denote the smallest integer greater than a given real number. The $n$-dimensional vector $(1, 1, \ldots 1)$ will be denoted as $\one$.
The summation $\sum_{1\le j \le n: j\neq k}$ will be written as $\sum_{j: j\neq k}$.
Bin$(n,p)$ denotes a Binomial random variable with $n$ trials and the probability of success $p$. $Exp(\lambda)$ denotes an Exponential random variable with mean $\lambda^{-1}$. 
For a measurable function, $f: \RR \to \RR$ and a measure $\mu$ on $\RR$, we denote $\langle f, \mu \rangle \doteq \int_\RR f\,d\mu$ when the latter is well defined. When $f = \mbox{Id}$ we write the left side above as $\langle x, \mu \rangle$.
\section{Main Results}
\label{sec:modres}
As noted in the introduction, our primary focus in this work is when $w$ is unbounded. In this case, one needs additional conditions on the jump rates and jump distribution to  ensure that the stochastic process associated with the $n$-particle system is well defined.
 
 Throughout $Z$ will denote a random variable distributed as $\theta$ and we will denote $\EE Z = \mn$ and $\EE Z^2 = \vt$ with $\mn, \vt \in (0, \infty)$.

The infinitesimal generator of the Markov process that we are interested in studying can be given as follows. For a bounded measurable map $f: \RR^n \to \RR$,
\begin{equation}\label{eq:gene}
\cll_nf(\bfx) \doteq  \sum_{j=1}^n \EE_Z[f(\bfx + Z e_j) - f(\bfx)] w(x_j - \bar\bfx),\; \bar\bfx \doteq \frac{1}{n}\sum_{i=1}^n x_i,\; \bfx = (x_1, \ldots , x_n) \in \RR^n.
\end{equation}
{\di We will impose the following basic assumption that suitably balances high jump rates with big jump sizes.
The final part of the assumption will be used to control the `overshoot' (in a single jump) of a particle lagging behind the mean in terms of its distance from the mean, the jump function $w$, and moments of $Z$.

\begin{assumption}\label{Z}
The jump rates and jump sizes satisfy: $\EE Z^3 < \infty$ and
    \begin{equation}
            \sup_{a \in \RR_+} w(-a) \mathbb{E} \left[(Z-a)^+ \vee (Z-a)^{+^2} \right]   \doteq c_{w}   < \infty.
        \end{equation}
\end{assumption}
Note that this latter assumption is satisfied if $w$ is bounded (and if the first part of the assumption holds). It also holds if the tails of the distribution of $Z$ are decaying sufficiently fast in comparison to the growth of $w$ at $-\infty$, e.g. when for $0<\beta \le \gamma$, $w(x)= e^{-\beta x}$, $x \in \RR_+$, and $\theta$ is $Exp(\gamma)$.}

The following result gives the existence of a Markov process with the above generator.
\begin{proposition}\label{prop:exis}
Suppose that Assumption \ref{Z} holds.
Then, for any $\bfx \in \RR^n$, there is a unique in law Markov process 
$\{\bfX^n(t), t \ge 0\}$ with values in $\RR^n$ such that $\bfX^n(0) = \bfx$ and such that, for every bounded and measurable $f: \RR^n \to \RR$,
$$M^n_f(t) \doteq f(\bfX^n(t)) - f(\bfx) - \int_0^t \cll_nf(\bfX^n(s)) ds
$$
is a local martingale with respect to the filtration $\clf^n_t \doteq \sigma \{\bfX^n(s): s\le t\}$.
\end{proposition}

 Proof is given in Section \ref{sec:exis}.

{\di \begin{remark}
 The basic idea in the proof is to argue that under Assumption \ref{Z} one cannot have `explosion', that is, infinitely many jumps on any compact time interval so that the process can be constructed recursively from one jump to the next. When Assumption \ref{Z} fails, such explosion can in fact occur. We illustrate this with a simple example.

Consider the case $n=2$, let $w(x)=e^{-\beta x}$ for some $\beta>0$, and let the jump distribution $\theta$ be supported on $\{1\}\cup\{a_m:m\ge 1\}$, where
\[
a_m \doteq 2^{2^m}, \qquad \PP(Z=1)=\frac12, \qquad \PP(Z=a_m)=c_0 m^{-2}, \quad m\ge 1,
\]
with $c_0>0$ chosen so that $\sum_{m\ge 1} c_0m^{-2}=\frac12$. Set $d_m \doteq a_m/2$ and suppose that initially $|X_2(0)-X_1(0)|=d_{m_0}$ for some large $m_0$.

Let $D(t)\doteq |X_2(t)-X_1(t)|$ be the gap process. When $D(t)=d$, the lagging particle jumps at rate $e^{\beta d/2}$ while the leading particle jumps at rate $e^{-\beta d/2}$. Starting from a level $d_m$, consider the first $N_m \doteq m^3$ jumps. With very high probability, none of these jumps is by the leading particle, since the probability of a leading-particle jump at any jump epoch is exponentially small in $d_m$. Conditional on this event, the jump sizes are i.i.d.\ with law $\theta$.

Now note that if the lagging particle jumps by an amount $z$, then the new gap is $|D-z|$. Since $a_m=2d_m$, a jump of size $a_m$ does not force the gap to grow; rather, it reflects the gap across $d_m$:
\[
d_m-r \mapsto d_m+r, \qquad d_m+r \mapsto d_m-r.
\]
On the other hand, jumps of size at most $a_{m-1}$ can change the gap by at most $a_{m-1}$ each. Thus, before the first jump of size at least $a_{m+1}$, the gap remains trapped in the interval
\[
[d_m-N_ma_{m-1},\, d_m+N_ma_{m-1}].
\]
Since $N_ma_{m-1}=m^3 2^{2^{m-1}}=o(d_m)$, for all large $m$ the gap stays between, say, $\frac34 d_m$ and $\frac54 d_m$ throughout this stage.

Also, with probability tending to $1$ as $m\to\infty$, among these first $N_m$ lagging-particle jumps there is at least one jump of size at least $a_{m+1}$, because
\[
\PP(Z\ge a_{m+1})=\sum_{k\ge m+1} c_0k^{-2}\asymp m^{-1},
\]
and hence $N_m\PP(Z\ge a_{m+1})\asymp m^2\to\infty$, implying that the probability of no such jump in the $N_m$ steps is $(1 - \PP(Z\ge a_{m+1}))^{N_m} \le \exp\{-N_m\PP(Z\ge a_{m+1})\} \to 0$. If such a jump occurs while $D\le \frac54 d_m$, then the new gap satisfies
\[
D' = z-D \ge a_{m+1}-\frac54 d_m > d_{m+1},
\]
for all large $m$, since $a_{m+1}=a_m^2=4d_m^2$ and $d_{m+1}=a_{m+1}/2=2d_m^2$.

Thus, by an application of the Borel-Cantelli Lemma, almost surely, the gap reaches successively the levels $d_m,d_{m+1},d_{m+2},\dots$. On this event, at level $d_m$ the total jump rate is at least $e^{3\beta d_m/8}$, so the waiting time to move from level $d_m$ to level $d_{m+1}$ is at most of order $N_m e^{-3\beta d_m/8}$. Since $\sum_m N_m e^{-3\beta d_m/8}<\infty$, these transition times are summable. It follows that, with positive probability, infinitely many jumps occur in finite time. This shows that when Assumption \ref{Z} fails, explosion may occur.
\end{remark}
}

 For the Markov process $\bfX^n$ as in Proposition \ref{prop:exis}, define
 processes $\{m_n(t), t\ge 0\}$ and $\{\mu_n(t), t\ge 0\}$ with sample paths in $\cld([0,\infty): \RR)$ and $\cld([0,\infty): \clp_1(\RR))$, respectively, as
\begin{equation}
m_n(t) \doteq \frac{1}{n}\sum_{i=1}^n X^n_i(t), \;\; \mu_n(t) \doteq \frac{1}{n}\sum_{i=1}^n \delta_{X^n_i(t)}, \; t \ge 0, \label{eq:133}
\end{equation}
 where $\bfX^n(t) = (X^n_1(t), \ldots , X^n_n(t))'$.
 From the Markov property in Proposition \ref{prop:exis}, it is easy to verify that, for $f\in \text{Lip}_1 $ and $t\geq 0$, with
 $$L\langle f, \mu_n(t) \rangle \doteq  \langle g_f(\cdot)w(\cdot-m_n(t)), \, \mu_n(t)\rangle,$$
 where \begin{equation}\label{eq:gf} g_f(x) \doteq \EE_Z[f(x+Z)] - f(x), x \in \RR,\end{equation}
  $$A_{t,f}(\mu_n(\cdot))\doteq \langle f, \mu_n(t) \rangle - \langle f, \mu_n(0) \rangle - \int_0^t L \langle f, \mu_n(s) \rangle\, ds,$$ is a $\clf^n_t$ - local martingale starting from $0$.

  Our first result will establish the convergence of $\mu_n$, in probability, in $\cld([0,\infty): \clp_1(\RR))$ to a suitable deterministic limit. From the above (local) martingale property one expects that the limit, denoted as $\mu$, should satisfy the equation 
  \begin{equation}\label{eq:mfeq0}
  A_{t,f}(\mu(\cdot))=0, \mbox{ for all } t\ge 0 \mbox{ and } f \in \text{Lip}_1,
  \end{equation}
  where 
  \begin{equation}\label{mfeq}
    A_{t,f}(\mu(\cdot)) \doteq \langle f,\mu(t) \rangle - \langle f,\mu(0) \rangle - \int_0^t \left\langle g_f(\cdot)w(\cdot-m(s) ), \, \mu(s)\right\rangle\, ds,
\end{equation}
and $m(s)\doteq \langle x, \mu (s) \rangle$.
Equation \eqref{eq:mfeq0} above is the natural McKean-Vlasov equation associated with the interacting particle system and
in \cite{Balzs2011ModelingFA},  it is referred to as the mean field equation. We will use these two terms interchangeably.  It is shown in \cite{Balzs2011ModelingFA}, under the conditions of that paper (in particular the condition that $w$ is bounded with bounded derivative), that this equation has a unique solution. Note that, when $w$ is not bounded, one needs additional conditions on $\mu$ to make sense of the right side of the last display.
Specifically, this expression is meaningful if $\mu \in \clc([0,\infty): \clp_1(\RR))$ satisfies $\int_0^t\int_{\RR} w(x-m(s)) \mu(s,dx) dt <\infty$ for every $t \ge 0$.
Thus, throughout, we will refer to the following as the McKean-Vlasov or the mean field equation for our system
\begin{equation}\label{eq:mfeq1}
 \int_0^t \langle w(\cdot-m(s)), \mu(s)\rangle ds <\infty, \;  A_{t,f}(\mu(\cdot))=0, \mbox{ for all } t\ge 0 \mbox{ and } f \in \text{Lip}_1.
  \end{equation}
Our next result shows the unique solvability of this equation under suitable conditions on $w$ (that allow for unbounded $w$). This result will be used to  characterize the limit in probability of $\mu_n$ by the unique solution of \eqref{eq:mfeq1}. 
We now introduce the key assumptions that will be needed for this convergence result.
Throughout we will assume, without loss of generality, that $\bfX^n$ for all $n$ are given on a common probability space $(\Om, \clf, \PP)$.

The first  condition imposes a suitable smoothness requirement on $w$.
\begin{assumption}\label{ass.w}
    For each $x \in \RR$, $w$ is Lipschitz on $[x, \infty)$ with Lipschitz constant, $L(x) \in (0,\infty)$.
\end{assumption}
Note that this assumption is satisfied, e.g. when for some $\beta>0$, $w(x) = e^{-\beta x}$, $x \in \RR_+$.
The next set of conditions are on the initial data.
The first one below imposes a certain uniform integrability condition on the initial data. In what follows, for $\gamma \in \clp(\RR)$, $A \in \RR$, and suitable $f: \RR \to \RR$, we will occasionally write  
$\int_{(-\infty, A)} f(x) \gamma(dx)$ as $\int_{-\infty}^A f(x) \gamma(dx)$. Similarly,  $\int_{\RR} f(x) \gamma(dx)$ will occasionally be written as $\int_{-\infty}^{\infty} f(x) \gamma(dx)$, or simply as $\int f(x) \gamma(dx)$.
\begin{assumption}\label{ass.t1}
   For any $\eta>0$, 
        \begin{equation}
            \lim_{B\to \infty} \limsup_{n \to \infty} \mathbb{P}\left(\int |x| \mathbf{1}(|x|\geq B) \mu_n(0,dx) \geq \eta \right) =0 .
        \end{equation}
\end{assumption}
The next condition ensures that the initial configuration has a suitably small left tail. 
\begin{assumption}\label{ass.char}
    For any $\epsilon, c >0$, 
    \begin{equation}
        \lim_{A\to \infty} \limsup_{n\to \infty} \mathbb{P} \left[ \int_{-\infty}^{-A} w(x-c)\mu_n(0,dx)\geq \epsilon \right]=0.
    \end{equation}
\end{assumption}
The above two conditions are clearly satisfied if for some compact $K \subset  \RR$, $X^n_i(0) \in K$, for all $i\in [n]$ and $n \in \NN$, a.s. However, we will need to allow for more general types of initial conditions in our study of the long-time behavior of the particle system, where such a uniform compact containment does not hold but Assumptions \ref{ass.t1}  and \ref{ass.char} are satisfied.

\begin{remark}\label{rem:conds}
    We remark that Assumption \ref{ass.char} holds if
    the following two properties hold. 
    \begin{enumerate}[(a)]
        \item\label{remc1} For any $c>0$, $\limsup_{x\to -\infty} \frac{w(x-c)}{w(x)} < \infty$.
        \item\label{remc2} $$ \limsup_{n\to \infty} \mathbb{E} \left[ \int |x|(1+w(x)) \mu_n(0,dx)
    \right] < \infty.$$  
    \end{enumerate}
 To see this, fix $c>0$ and let $c_0 \doteq \sup_{x<0} \frac{w(x-c)}{w(x)}$. Note that $c_0<\infty$. 
 Then using  (\ref{remc2}), and the non-increasing property of $w$, we have
 $$\limsup_{n\to \infty} \mathbb{E} \left[\int |x| w(x-c)\mu_n(0,dx) \right] \leq \limsup_{n\to \infty} \mathbb{E} \left[ \int |x| (c_0w(x)+w(-c)) \mu_n(0,dx) \right]  < \infty.$$ Hence, for any $\epsilon>0$, 
    \begin{align*}
        \lim_{A\to \infty} \limsup_{n\to \infty} \mathbb{P} \left[\int_{-\infty}^{-A} w(x-c) \mu_n(0,dx) \geq \epsilon \right] &\leq \lim_{A\to \infty} \limsup_{n\to \infty} \epsilon^{-1}\mathbb{E} \left[ \int_{-\infty}^{-A} w(x-c) \mu_n(0,dx) \right]\\
        &\leq \lim_{A\to \infty} \limsup_{n\to \infty} (A\epsilon)^{-1}\mathbb{E} \left[ \int |x| w(x-c) \mu_n(0,dx) \right] =0.
    \end{align*}
    \end{remark}

The final condition on the initial data gives suitable convergence of the initial empirical measure.
For a probability measure $\gamma \in \clp_1(\RR)$, consider the following two properties:
\begin{equation}\label{cM1}
\begin{aligned} \mbox{ For every } a\in \RR, \mbox{ there exists } \delta>0 & \mbox{ such that, with } L(\cdot) \mbox{ as in Assumption \ref{ass.w}, }\\
        &\int_\RR e^{\delta L(x-a)} \gamma(dx) < \infty.
        \end{aligned}
\end{equation}
\begin{equation}\label{cM2}
\mbox{ For some } \beta \in (0,1], \, \int_\RR |x|^{1+\beta} \gamma(dx)<\infty.
\end{equation}
\begin{assumption}\label{ass.init}
    As $n \to \infty$,  $\mu_n(0) \to \gamma$ in probability, in $\mathcal{P}_1(\RR)$, where $\gamma \in \mathcal{P}_1(\RR)$ satisfies
    \eqref{cM1} and \eqref{cM2}.
     \end{assumption}
\subsection{Fluid limit and propagation of chaos} \label{sec:flpoc}
We now present our first main result concerning the convergence on compact time intervals of the empirical measure process encoding particle locations, as $n \to \infty$.  Proof of the result is given in Section \ref{sec:pf2.8}. Define
\begin{align*}
\clm &\doteq \{\mu \in  \mathcal{C}([0,\infty):\mathcal{P}_1(\RR)): \mu(0) \mbox{ satisfies conditions \eqref{cM1} and \eqref{cM2}} \}.
\end{align*}
\begin{theorem}\label{thm.mainFL}
We have the following.
\begin{enumerate}[(a)]
\item Suppose that Assumptions \ref{Z}, \ref{ass.t1}, and \ref{ass.char} hold. Then 
$\{\mu_n\}_{n\in \NN}$ is $\mathcal{C}$-tight in $\mathcal{D}([0,\infty):\mathcal{P}_1(\mathbb{R}))$.
\item Suppose in addition that Assumption \ref{ass.w} holds. Then the subsequential distributional limit points  $\mu$ of $\mu_n$, in  
    $\mathcal{D}([0,\infty):\mathcal{P}_1(\mathbb{R}))$, satisfy the mean field equation \eqref{eq:mfeq1}.
\item Suppose that in addition Assumption \ref{ass.init} holds. Then there is a unique $\mu \in \clm$ that solves the McKean-Vlasov equation \eqref{eq:mfeq1} and satisfies
    $\mu(0) = \gamma$.
     Furthermore, $\mu_n$ converges in probability, in $\mathcal{D}([0,\infty):\mathcal{P}_1(\mathbb{R}))$, to  $\mu$. 
\item (Propagation of chaos.) Suppose that Assumptions \ref{Z} and \ref{ass.w} hold and $X^n_1(0), \dots, X^n_n(0)$ are iid with distribution $\gamma$, where $\gamma$ satisfies \eqref{cM1} and \eqref{cM2}. Then $\mu_n$ converges in probability, in $\mathcal{D}([0,\infty):\mathcal{P}_1(\mathbb{R}))$, to  $\mu$ as given in part 
(c). Moreover, for any $t\ge 0$, the law of $\bfX^n(t)$ is $\mu(t)$-chaotic, i.e.    for any $k \in \NN$, a bounded continuous $f: \RR^k \rightarrow \RR$, and $t \ge 0$,
$$
\EE\left[f\left(X^n_1(t),\dots,X^n_k(t)\right)\right] \to \int_{\RR^k}f(x_1,\dots,x_k)\mu(t,dx_1)\dots\mu(t,dx_k) \ \text{ as } n \to \infty.
$$
\end{enumerate}
\end{theorem}

\noindent As noted in the introduction, for the case when $w$ is a bounded function that is differentiable and has a bounded derivative,  parts (a)-(c) of the above theorem were shown in \cite{Balzs2011ModelingFA}. The setting considered above is substantially more involved due to the possible confluence of fast and large jumps under our conditions and requires several new ideas. We highlight below the key steps in the proof of Theorem \ref{thm.mainFL}.
\begin{enumerate}[(i)]
\item \textit{A priori bound on the mean process $m_n$: }A key estimate required in the proof of Theorem \ref{thm.mainFL} is a probability bound on the empirical mean process $\{m_n(t), \, t > 0\}$, defined in \eqref{eq:133}, which is obtained in Theorem \ref{thm.mbd} below. This is a crucial technical step. The difficulty in estimating the empirical mean process is in controlling  the effect of the particles far behind the mean. When $w$ is unbounded, these particles can have very high jump rates. Each such jump increases the mean which, in turn, increases the jump rate of all the particles. Moreover, as the jump distribution could have heavy tails, some of these jumps could be very large. This `feedback mechanism' could potentially result in the mean process (and hence the particle locations) increasing very fast. As we will see, controlling the growth rate of $m_n$ \emph{uniformly in $n$} on compact time intervals is at the heart of the proof of Theorem \ref{thm.mainFL}. For bounded $w$, treated in \cite{Balzs2011ModelingFA}, this control becomes significantly simpler via stochastically upper bounding the particle system by one comprising \emph{independent} particles moving forward at rate $\sup_{x \in \mathbb{R}}w(x)$.
\item \textit{Proving $\mathcal{C}$-tightness of $\mu_n(\cdot)$ in the Skorohod space $\mathcal{D}([0,\infty):\mathcal{P}_1(\RR))$: } This requires proving pointwise tightness (Lemma \ref{lem.tght}) and an oscillation estimate (Lemma \ref{lem.tght.b}) for the measure-valued processes $\{\mu_n\}_{n \in \NN}$. These are established by  exploiting monotonicity properties of the process combined with martingale arguments on a `good set' where one has good a-priori control on the empirical mean. The probability of the `bad set' is shown to be small using the mean bound obtained in Theorem \ref{thm.mbd}.  
\item \textit{Characterizing subsequential limits of $\{\mu_n\}_{n \in \NN}$: } These are shown to solve the McKean-Vlasov equation \eqref{eq:mfeq1}, again by using monotonicity, martingale techniques and the a priori mean bound.
\item \textit{Uniqueness of solutions to \eqref{eq:mfeq1} in $\mathcal{M}$: } Unlike the class of $w$ treated in \cite{Balzs2011ModelingFA}, our conditions allow for cases when $w$ is not globally Lipschitz (eg. $w(x)=e^{-x}, \, x \in \RR$). Hence, we cannot directly apply standard Gr\"onwall type arguments to deduce uniqueness. We instead use Assumption \ref{ass.init} which says that the left tails of the (limiting) initial measure $\mu(0)$ are light enough to compensate for the growing Lipschitz constant $L(x)$ of $w$ on $[x,\infty)$ as $x \rightarrow -\infty$. To show uniqueness, we work with certain `coupled nonlinear Markov processes', whose marginal laws are given by the solutions of the two candidate solutions of the McKean-Vlasov equations and then argue that these processes are in fact identical.
The last step is accomplished by controlling the $L^1$ distance between the two solutions by appealing to Assumption \ref{ass.init}, monotonicity and the mean bound, in conjunction with Gr\"onwall's inequality.
This idea is inspired by  \cite{Oelschlager1984AMA} (see Theorem \ref{mv}), however, the proof details are substantially more involved due to the poor available control on $w$. 
\end{enumerate}

\begin{remark}
    Condition \eqref{cM1} clearly holds for any Lipschitz $w$. On the other hand, in the case where $w(x) = e^{-\beta x}$ for some $\beta>0$, $w$ is far from Lipschitz and this assumption roughly requires $\nu((-\infty, -x]) \le C e^{-\delta e^{\beta x}},\, x \ge 0,$ for some constants $C, \delta>0$. Although this seems like a strong requirement, it is natural for the problem. Specifically, the `steady state' traveling wave solution to the McKean-Vlasov equation \eqref{eq:mfeq1} in this case when $Z$ has an exponential distribution, described in Theorem \ref{thm.ss.exp} below, exhibits double exponentially decaying left tails which says that \eqref{cM1} holds for this initial distribution and also for those  that are `not too far' from this traveling wave solution.
\end{remark}

For any $\mu \in \mathcal{C}([0,\infty):\mathcal{P}_1(\RR))$, the function $m(t) \doteq \int x \mu(t, dx)$, $t\ge 0$, is a continuous function from $\RR_+$ to $\RR$. This, in particular, says that for any $T\in (0, \infty)$, and a random variable $Y(0)$ distributed as $\mu(0)$, the function $w(Y(0)-m(t))$ is a.s. a bounded function on $[0,T]$, {\di since
$\sup_{t \in [0,T]} w(Y(0)-m(t)) \le w(Y(0)- \sup_{t \in [0,T]} m(t))<\infty$, a.s.} And since $w$ is a non-increasing function, this allows us to construct a pure jump process $\{Y(t), t\ge 0\}$ with sample paths in $\cld([0, \infty):\RR)$, recursively from one jump to next, for which $\cll(Y(0)) = \mu(0)$ and the jump at any instant $t$, distributed as $\theta$, occurs at rate $w(Y(t)-m(t))$.
Furthermore, this process is unique in law characterized by the properties that $\cll(Y(0))= \mu(0)$ and, for every $f \in \MM_b(\RR)$,
\be \label{eq:338}
f(Y(t)) - f(Y(0)) - \int_0^t \cll_{\mu} f(s,Y(s)) ds,\; t\ge 0,
\ee
is a local martingale  with respect to the filtration $\clf_t^Y \doteq \sigma \{Y(s): s\le t\}$, where
$$\cll_{\mu} f(s,y) = g_f(y) w(y- m(s)), \; s \ge 0.$$

Now suppose $\mu \in \mathcal{C}([0,\infty):\mathcal{P}_1(\RR))$ satisfies  the McKean-Vlasov equation \eqref{eq:mfeq1}. In the course of proving Theorem \ref{thm.mainFL}, we will show that, under assumptions, the process $\{Y(t)\}$ constructed as above using such a $\mu$ in fact satisfies $\mu(t)= \cll(Y(t))$ for every $t \ge 0$. 
In particular, the jump rate of $Y$ at time $t$ depends on $\cll(Y(t))$. Following standard terminology (cf. \cite{sznit}), we will call this process the \emph{nonlinear Markov process} associated with $\mu$. The next theorem records this observation. Proof is given in Section \ref{s.FL}.

\begin{theorem}\label{mv}
  Suppose $w$ satisfies Assumption \ref{ass.w} and $\mu \in \mathcal{C}([0,\infty):\mathcal{P}_1(\RR))$ is a solution to \eqref{eq:mfeq1} satisfying $\int_{-\infty}^{\infty}L(x-a)\mu(0,dx) < \infty$ for any $a \in \RR$. Then the process $Y$ constructed as above from $\mu$ with generator $\cll_{\mu}$ is the nonlinear Markov process associated with $\mu$, that is, $\cll(Y(t)) = \mu(t)$ for every $t \ge 0$.
\end{theorem}

\subsection{Long-time behavior}

Our next results concern the long-time behavior of a suitably centered version of the $n$-particle system described by $\bfX^n$. With $\bfX^n$ as in Proposition \ref{prop:exis}, let $Y^n_i(t) \doteq X^n_i(t) - m_n(t)$, $i = 1, \ldots, n$, $t \ge 0$. We let $\bfY^n(t) \doteq (Y^n_1(t), \ldots , Y^n_n(t))'$.
It is easy to verify using Proposition \ref{prop:exis} that, under Assumption \ref{Z},
$\bfY^n$ is a $\cls_0 \doteq \{\bfy \in \RR^n: \bfy\cdot \one=0\}$ valued Markov process with infinitesimal generator 
\begin{equation}\label{eq:geny}
\cll^{\bfY}_n f(\bfy) \doteq  \sum_{j=1}^n \EE_Z[f(\bfy + Z e_j - n^{-1} Z\one) - f(\bfy)] w(y_j), \; \bfy = (y_1, \ldots , y_n)' \in \cls_0,
\end{equation}
where $f\in \MM_b(\cls_0)$.

The following theorem gives the existence of a stationary distribution $\pi_n$ for $\bfY^n$ for each $n \ge 2$, which is unique under suitable irreducibility assumptions. It also establishes certain integrability properties for $\{\pi_n, n \ge 2\}$ and obtains an `asymptotic velocity' for the mean process for any fixed $n \ge 2$. These results are new even for the setting where $w$ is a bounded globally Lipschitz function. Proof of the result is given in Section \ref{subsec6.1}. 

We recall that a distribution $F$ is called \emph{spread-out} if it has an absolutely continuous component: there exists $f:\RR \rightarrow [0,\infty)$, with $\int_{\RR}f(x)dx \in (0,1]$, such that
$$
F(A) \ge \int_A f(x)dx, \ A \in \mathcal{B}(\RR).
$$
\begin{theorem}\label{stat.ex}
Suppose that Assumptions \ref{Z} and \ref{ass.w} hold and that $w$ is not a constant function.
    Then, for each $n\geq 2$, there exists a stationary distribution $\pi_n$, for the Markov process $\bfY^n = (Y^n_1, Y^n_2, ..., Y^n_n)'$. Furthermore, for any $t\geq 0$,
    \begin{equation}\label{st.thm1}
        \sup_{n\geq 2} \mathbb{E}_{\pi_n} \left[ \frac{1}{n} \sum_{i=1}^n |Y_i^n(t)|(1+w(Y_i^n(t)))\right] < \infty,
    \end{equation}
    where $\mathbb{E}_{\pi_n}$ denotes the expectation under the probability measure under which $\cll(\bfY^n(0)) = \pi_n$. Moreover, if $\theta$ is spread-out, then the stationary distribution is unique. In this case, for any $n \ge 2$, the mean process $\{m_n(\cdot)\}$ of the particle system $\bfX^n$, with $\bfX^n(0) = \bfx$, satisfies
    \begin{equation}\label{velm}
       \frac{m_n(t)}{t} \to \mn\EE_{\pi_n}\left[\frac{1}{n}\sum_{i=1}^n w(Y^n_i(0))\right] = \mn\EE_{\pi_n}\left[w(Y^n_1(0))\right] 
    \end{equation}
   in probability as $t \to \infty$, for $\pi_n$-a.e. $\bfx$.
\end{theorem}
\begin{remark}
{\di (i) The estimate in \eqref{st.thm1} gives a uniform in $n$ integrability property under the stationary distribution $\pi_n$. In particular, this relates the tail behavior of $\pi_n$ to the growth rate of $w$. This will play a key role in the proof of the law of large numbers in \eqref{velm} (see \eqref{mn2}).

It will also be used to show that Assumptions \ref{ass.t1} and \ref{ass.char} hold when the particles are started with stationary (re-centered) configuration given by distribution $\pi_n$. This, in turn, will help establish the `fluid limit' Theorem \ref{thm.mainFL}, for the re-centered particle system, which plays a pivotal role in the proof of Theorem \ref{thm.stat.sol}.
}

\noindent (ii) One basic setting in which Theorems \ref{thm.mainFL}, \ref{mv} and  \ref{stat.ex} are applicable is when the following hold: (i) $w$ is a non-constant bounded Lipschitz function which is non-increasing; (ii) $\int_{0}^{\infty} z^3 \theta(dz) <\infty$; (iii) $\sup_n  \mathbb{E}[\int |x|^{1+\delta} \mu_n(0,dx)] < \infty$ for some $\delta \in (0,1)$; (iv) $\mu_n(0) \to \nu$, in the weak topology, in probability for some $\nu \in \clp(\RR)$. In this case it is easily seen that Assumptions  \ref{Z}, \ref{ass.w}, \ref{ass.t1}, \ref{ass.char}, \ref{ass.init} are all satisfied and therefore the conclusions of Theorems \ref{thm.mainFL}, \ref{mv} and \ref{stat.ex} hold. In fact, Theorem \ref{thm.mainFL} was previously proved under similar but slightly different assumptions in \cite{Balzs2011ModelingFA}. However, our primary interest is in a setting where $w$ is unbounded and these assumptions will also hold in interesting cases where $w$ is unbounded, cf.  Theorem \ref{thm.ss.exp}.
\end{remark}
In studying the asymptotic behavior of $\pi_n$ as $n \to \infty$, we will need a uniform integrability property under $\pi_n$ in addition to \eqref{st.thm1} for which we introduce the following two assumptions.
\begin{assumption}\label{Z2}
    The jump distribution $\theta$ and the rate function $w$ satisfy
    $$\lim_{A\to \infty} \sup_{x\geq 0} w(A-x)\mathbb{E} [(Z-x)^{+^2}] =0.$$
\end{assumption}
 The above assumption gives a uniform control on overshoots of the particle jumps beyond the mean, which helps in establishing the uniform integrability property in Theorem \ref{stat.pr}.

\begin{assumption}\label{ass.w.LT}
The function $w$ satisfies:
     $\lim_{x\to \infty} w(x) =0$  and  $\lim_{x\to -\infty} w(x) =\infty$.   
\end{assumption}
One basic setting where the last two assumptions hold is when for $0<\beta \le \gamma$, $w(x)= e^{-\beta x}$, $x \in \RR_+$, and $\theta$ is $Exp(\gamma)$.

Proof of the following theorem is in Section \ref{subsec6.2}.
\begin{theorem}\label{stat.pr}
    Suppose Assumptions \ref{Z}, \ref{ass.w}, \ref{Z2} and \ref{ass.w.LT} hold. Then, for each $n\geq 2$, the stationary distribution, $\pi_n$, obtained in Theorem \ref{stat.ex} satisfies for any $\eta>0$,
    \begin{equation}\label{st.thm2}
        \lim_{A\to \infty} \limsup_{n\to \infty} \mathbb{P}_{\pi_n} \left[ \frac{1}{n}\sum_{i=1}^n |Y_i^n(0)| \mathbf{1}[|Y_i^n(0)|\geq A] \geq \eta \right]=0,
    \end{equation}
    where $\mathbb{P}_{\pi_n}$ denotes the probability measure under which $\cll(\bfY^n(0)) = \pi_n$.
\end{theorem}

\noindent {\di The uniform integrability property in the above theorem will be used in the proof of Theorem \ref{thm.stat.sol} in order to verify Assumption \ref{ass.t1} with $\PP$ replaced by
$\PP_{\pi_n}$, the law of the stationary re-centered particle system.}

The proofs of Theorems \ref{stat.ex} and \ref{stat.pr} involve constructing suitable Lyapunov functions.

Theorem \ref{thm.mainFL} gave convergence of the empirical measure process to the solution of the McKean-Vlasov equations under suitable conditions on the initial distributions of the particles. By verifying that these conditions are satisfied when the initial condition is given by the stationary distributions $\pi_n$, our next two results give the convergence of the empirical measure under $\pi_n$ as $n\to \infty$ and characterizes the limit as the fixed point of the centered version of the McKean-Vlasov equation \eqref{eq:mfeq1}. 
The latter is the equation satisfied by the measures $\{\nu(t), t \ge 0\}$ that are defined by the relation
$$\int f(x) \nu(t) (dx) \doteq \int f(x-m(t)) \mu(t,dx), \;  f \in \MM_b(\RR), \  t \ge 0,$$
where $\{\mu(t), t \ge 0\}$ solves \eqref{eq:mfeq1}.
It is easy to verify that this equation is given as
\begin{equation}\label{FL}
\begin{aligned}
        \langle f, \nu(t) \rangle &= \langle f, \nu(0) \rangle + \int_0^t \langle g_f\cdot w, \nu(s) \rangle ds - \int_0^t \mn\langle w, \nu(s) \rangle \langle f', \nu(s) \rangle ds, \; t \ge 0, f \in \mbox{Lip}_1,\\
        &\int_0^t \langle w, \nu(s)\rangle ds < \infty, \mbox{ for all } t \ge 0.
        \end{aligned}
    \end{equation}
    where $g_f$ is as in \eqref{eq:gf}.
    
We will call a stochastic process $\{\nu(t), t \ge 0\}$ with sample paths in $\clc([0, \infty): \clp_1(\RR))$ a {\em stationary solution} of \eqref{FL} if it solves \eqref{FL} a.s., and $\cll(\nu(t+\cdot)) = \cll(\nu(\cdot))$ for all $t\ge 0$ as probability measures on $\clc([0, \infty): \clp_1(\RR))$.
We will call a $\nu^* \in \clp_1(\RR)$ a {\em fixed point} of \eqref{FL}, if the constant path $\nu(t) \doteq \nu^*$ for all $t\ge 0$ is a solution of \eqref{FL}. Note that a fixed point is obviously a stationary solution.
Define the map $\Theta^n: \RR^n \to \clp_1(\RR)$ as $\Theta^n(\bfy) \doteq \frac{1}{n} \sum_{i=1}^n \delta_{y_i}$, for $\bfy = (y_1, \ldots , y_n)' \in \RR^n$.
Let $\Pi_n \doteq \pi_n \circ (\Theta^n)^{-1}$, where $\pi_n$ is the stationary distribution obtained in Theorem \ref{stat.ex}.

We now introduce one additional condition on $w$.
\begin{assumption}\label{eq:wtails}
Suppose that, for any $c \in (0,\infty)$, $\limsup_{x\to -\infty} \frac{w(x-c)}{w(x)} < \infty$. 
\end{assumption}
\begin{theorem}\label{thm.stat.sol}
    Let Assumptions \ref{Z}, \ref{ass.w}, \ref{Z2},  \ref{ass.w.LT}, and \ref{eq:wtails} hold. Suppose that $\nu^* \in \clp_1(\RR)$ is a fixed point  for \eqref{FL}.
 Further suppose that whenever $\{\nu(t), t\ge 0\}$ is a stationary solution of \eqref{FL} that satisfies $\EE \langle w, \nu(0)\rangle <\infty$, then we  must have that $\nu(0) =\nu^*$ a.s.  Then, as $n\to \infty$, $\Pi_n \to \delta_{\nu^*}$ in probability, in $\clp(\clp_1(\RR))$. Furthermore, we have propagation of chaos at $t=\infty$, namely, for $k \ge 1$, denoting the first $k$-marginal distribution of $\pi_n$ by $\pi_n^{(k)}$, we have $\pi_n^{(k)} \to (\nu^*)^{\otimes k}$ as $n\to \infty$.

Finally, if $\{m_n(\cdot)\}$ denotes the mean process of the particle system $\bfX^n$ with $\bfX^n(0) \sim \pi_n$, then
 \begin{equation}\label{meanlt}
\lim_{n \to \infty}\lim_{t \to \infty} \frac{m_n(t)}{t} = \mn\lim_{n \to \infty}\EE_{\pi_n}\left[\frac{1}{n}\sum_{i=1}^n w(Y^n_i(0))\right] = \mn\int_{-\infty}^{\infty}w(y) \nu^*(dy),
 \end{equation}
 where the inner limit is taken in probability for the first equality.
\end{theorem}
The proof of the above theorem is in Section \ref{subsec6.3}.
The second requirement in the statement of the theorem above is in general hard to verify, but may be analyzed on a case-by-case basis. We do this for one important setting in the result below. This setting corresponds to the case 
 where the distribution for the jump length, $Z$, is Exponential with rate $\gamma>0$ and  the jump rate function, $w$, is exponential as well, i.e.   $w(x)=e^{-\beta x}$, $x \in \RR$, where 
 $\beta \in (0,\gamma]$.
%
\newcommand\bnut{\Bar{\nu}_t}
\begin{theorem}\label{thm.ss.exp}
Suppose that, for some $0<\beta\le \gamma <\infty$, $w(x)=e^{-\beta x}$, $x \in \RR$. Also suppose that $\theta$ is the law of an Exponential random variable with rate $\gamma$. The following hold.
\begin{enumerate}[(a)]
 \item    Let $\{\bar\nu(t), t \ge 0\}$  be a stationary solution of \eqref{FL} that satisfies $\EE\langle w, \bar \nu(0)\rangle <\infty$. Then, almost surely, $\Bar{\nu}(0) = \nu^*$, where $\nu^* \in \clp_1(\RR)$ is a fixed point for \eqref{FL}, given as
    \begin{equation}
    \nu^*(dx)=\frac{\gamma}{\Gamma(1+\gamma\beta^{-1})}\exp\left[ -\left(\gamma (x-\beta^{-1}\Psi(\gamma\beta^{-1}))\right) - e^{-\beta\left(x-\beta^{-1}\Psi(\gamma\beta^{-1})\right)}\right]\;dx,
    \label{eq:838}
    \end{equation}
    where $\Psi(a)\doteq \frac{\Gamma'(a)}{\Gamma(a)}$, $a> 0$. 
\item As $n\to \infty$, $\Pi_n \to \delta_{\nu^*}$ in probability, in $\clp(\clp_1(\RR))$.
\item There is a unique  $\mu \in \clm$ that solves  the McKean-Vlasov equation \eqref{eq:mfeq1} and satisfies
    $\mu(0) = \nu^*$. This $\mu$ is given as 
    $$\mu(t,dx)=\frac{\gamma}{\Gamma(1+\gamma\beta^{-1})}\exp\left[ -\gamma\left(x-\bar m(t)\right) - e^{-\beta\left(x-\bar m(t)\right)}\right]\;dx,$$
    and $$\bar m(t) = \beta^{-1} e^{-\Psi(\gamma\beta^{-1})}t + \beta^{-1}\Psi(\gamma\beta^{-1}).$$
    In particular, $m(t) = \langle x, \mu(t)\rangle = \beta^{-1} e^{-\Psi(\gamma\beta^{-1})}t$. 
\item When the initial distribution of the particle system is $\pi_n$, we have that  $\mu_n$ converges in probability to $\mu$ in $\mathcal{D}([0,\infty):\mathcal{P}_1(\mathbb{R}))$. Furthermore, the limit in \eqref{meanlt} equals $\beta^{-1} e^{- \Psi(\gamma \beta^{-1})}$.   
\end{enumerate}
\end{theorem}

For $\beta = \gamma =1$, \eqref{eq:838} is the well-known Gumbel distribution. As noted in \cite[Section 5.2]{Balzs2011ModelingFA}, the above setting has interesting connections with extreme value theory.

Proof of this result is in Section \ref{s.Exp}. The main idea is to consider a solution $\nu^*(\cdot)$ of an equation related to \eqref{FL}, in which the `non-linearity' $\langle w, \nu(\cdot)\rangle$ is replaced by a given input $\langle w, \bar{\nu}(\cdot)\rangle$ obtained from the stationary solution.
 By coupling arguments, it is then shown that the sup-norm distance between the stationary {\di cumulative distribution function (cdf)} $\bar{F}(t)$ of $\bar{\nu}(t)$ and the cdf $F^*(t)$ of $\nu^*(t)$ approaches zero, a.s., as $t \to \infty$. The result is then shown by exploiting this fact, the special form of $\nu^*$, and the stationarity of $\bar{\nu}(\cdot)$.

\begin{remark}\label{POCcomp}
    Note that $\mu(t, dx)$ above is a traveling wave solution, namely it is of the form $\rho(x-ct) dx$. Such traveling wave solutions to \eqref{eq:mfeq1} were obtained in certain cases (including the one studied in Theorem \ref{thm.ss.exp}) by \cite{Balzs2011ModelingFA} {\di(see \cite[Theorem 3.1 and Corollary 3.2]{Balzs2011ModelingFA})}. The corresponding measures $\nu^*(dx) = \rho(x) dx$ can be verified to be fixed points of \eqref{FL}. In view of this, a natural approach to prove part (b) (and consequently also part (d)) of the above theorem
    would be the following: (i) Establish `uniform-in-time propagation of chaos': e.g., show that, with $\nu_n(t, dx) = \mu_n(t, dx-m_n(t))$, $\sup_{t \ge 0}\EE\mathcal{W}_1(\nu_n(t),\nu(t)) \le C \EE\mathcal{W}_1(\nu_n(0),\nu(0)) + \eps_n$ for some constant $C>0,$ where $\nu$ is the fluid limit of $\nu_n$ and $\eps_n \to 0$ as $n \to \infty$;  (ii) Show that $\mathcal{W}_1(\nu(t), \nu^*) \to 0$ as $t \to \infty$. Estimates similar to (i) have  been obtained for several interacting particle systems (mostly diffusion processes) satisfying certain strong stability properties \cite{BuFan,durmus2020elementary,lacker2023sharp}. Results similar to (ii) have been obtained in \cite{greenberg1996asynchronous,stolyar2023particle,stolyar2023large} in the context of a particle system with quantile based interactions. However, the stability properties of our model are not sufficiently strong  to provide uniform estimates as in (i) or global stability of McKean-Vlasov equation as in (ii).  Moreover, unlike in \cite{greenberg1996asynchronous,stolyar2023particle,stolyar2023large} where the mean function moves with constant velocity in the fluid limit, in our case the  evolution of $m(\cdot)$ is nonlinear and has a complex form through the rate function $w$ (cf. \eqref{eq:mfeq0} with $f=\mbox{Id}$). This makes   the tools developed in these papers hard to implement for the current setting. 

    In view of this  we  develop different techniques for proving the above results. Instead of attempting the uniform estimate in (i), we construct Lyapunov functions to directly quantify integrability properties of $\pi_n$ which allows us to  use our fluid limit result (Theorem \ref{thm.mainFL}) to show that subsequential limits of empirical measures of the stationary particle systems correspond to stationary solutions of \eqref{FL}. A separate argument is then used to characterize stationary solutions as fixed points of \eqref{FL} and thereby prove convergence of the whole sequence of empirical measures under $\pi_n$ to the fixed point. 
\end{remark}

{\di Figure \ref{fig:partsim} plots the joint motion of a system of $n=20$ particles for different choices of $w$ and $\theta$. The dotted black line represents the center of mass. The first two figures respectively have $w(x) = 1 + x^-, \theta =Exp(1)$ and $w(x) = e^{-x}, \theta =Exp(1)$. The clustering effect is much quicker in the second figure after which the cluster moves at a slower speed than in the first figure.

The last one has $w(x) = 1 + \sqrt{x^-}$ and $\theta = Pareto(2)$. The effect of frequent large jumps due to heavy-tailed $\theta$ is manifested via particles switching orders and fluctuating more, resulting in less synchronized flocking.}
\begin{figure}[htbp]
    \centering
    \begin{subfigure}[b]{0.85\textwidth}
        \centering
        \includegraphics[width=\textwidth]{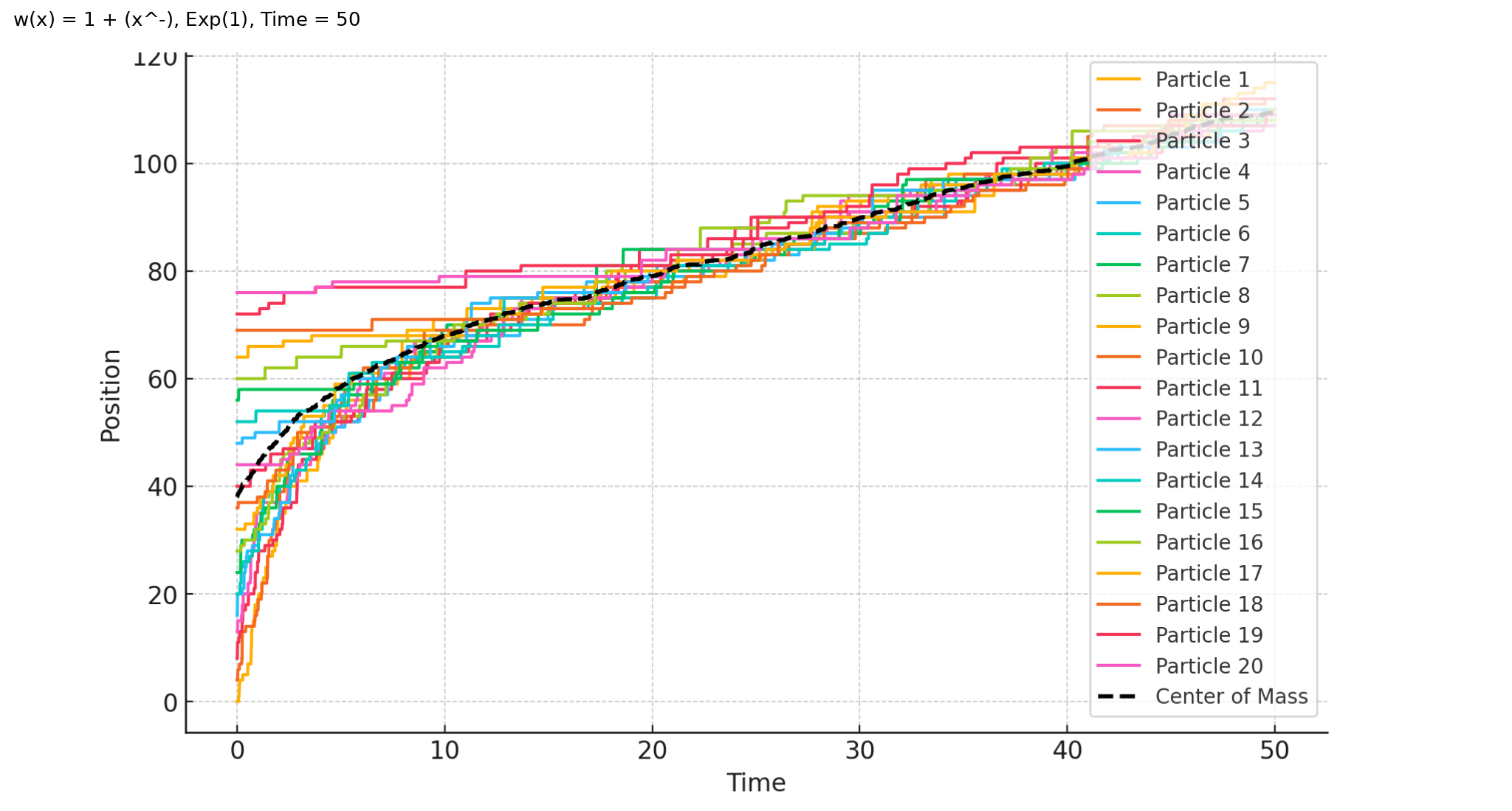}
    \end{subfigure}
    \\
    \begin{subfigure}[b]{0.85\textwidth}
        \centering
        \includegraphics[width=\textwidth]{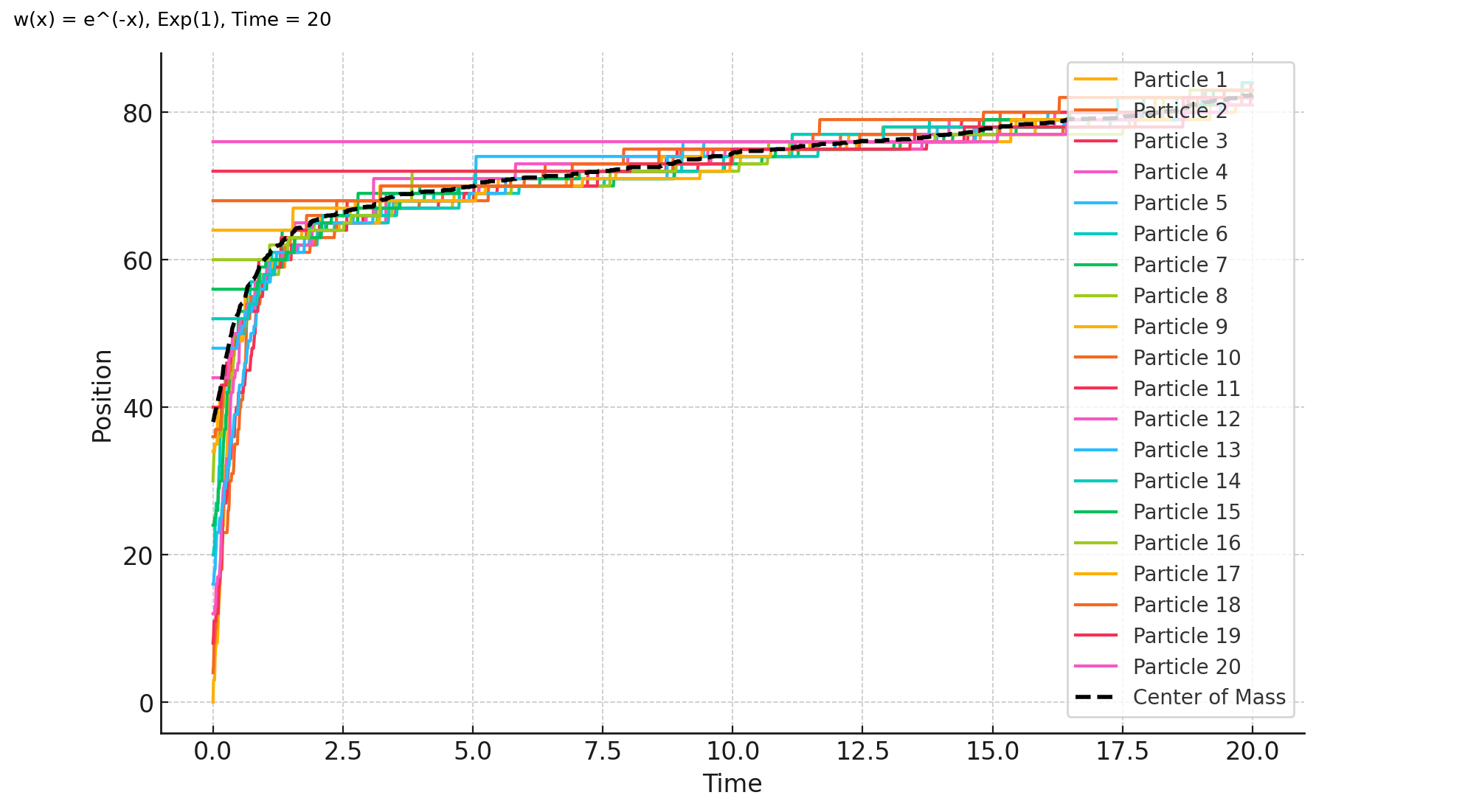}
    \end{subfigure}
    \\
    \begin{subfigure}[b]{0.85\textwidth}
        \centering
        \includegraphics[width=\textwidth]{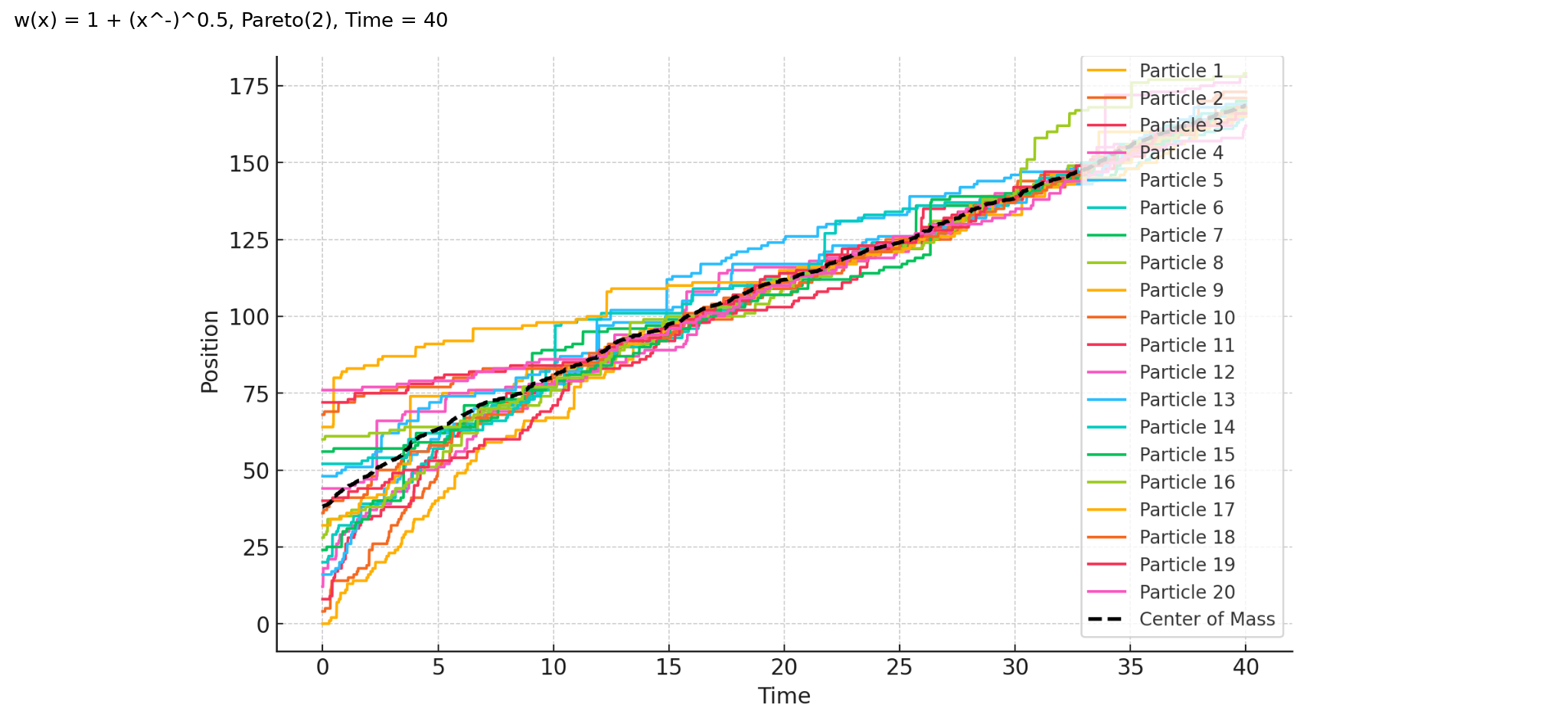}
    \end{subfigure}
    \caption{Visualization of particle trajectories for different choices of $w$ and $\theta$.}
    \label{fig:partsim}
\end{figure}


\section{Proof of Proposition \ref{prop:exis}.}
\label{sec:exis}
For $R\in (0,\infty)$, consider the modification of the $n$-particle dynamics in which the jump rates $w(y_i - \bar\ybd)$ are replaced by 
$$w^{n,R}_i(\ybd) \doteq w(y_i- \bar\ybd) \one_{\{\ybd_{\max}  \le R\}} + \one_{\{\ybd_{\max}  > R\}}, \; i \in [n],\; \ybd = (y_1, \ldots , y_n)' \in \RR^n,$$
where $\ybd_{\max} \doteq \max\{y_1, \ldots y_n\}$.
From the monotonicity property of $w$ it follows that the jump-Markov process $\bfX^{n,R}$, starting from any initial condition $\xbd \in \RR^n$, associated with the jump rates  $\{w_i^{n,R}, i \in [n]\}$ is well defined. The generator $\cll_n^{R}$ of this Markov process is given by \eqref{eq:gene} with $w(x_j - \bar \bfx)$ there replaced by $w_i^{n,R}(\xbd)$. 
Furthermore, using a standard construction (based on a common collection of Poisson random measures), we can ensure that the Markov processes $\{\bfX^{n,R}, R \in \NN\}$ are defined on a common probability space and satisfy the  property that for $R_0 < R_1 \le R_2$,  \begin{equation}\label{eq:consis}
\mbox{ on the set }
\{A^{n,R_1}(T) \le R_0\},\,
\bfX^{n,R_1}(t) = \bfX^{n,R_2}(t) \mbox{ for all } t\in [0,T],
\end{equation}
where $A^{n,R}(t) \doteq \max\{X^{n,R}_i(t), 1 \le i \le n\}$, $t\in [0,T]$, $R\in \NN$.
From the form of the generator $\cll_n^{R}$ it follows that, 
$$\clm^{n,R}(t) \doteq A^{n,R}(t) - \xbd_{\max} - \sum_{i=1}^n \int_0^t \EE_Z([Z-A^{n,R}(s) + X^{n,R}_i(s)]^+) w_i^{n,R}(\bfX^{n,R}(s)) ds, \; t \ge 0, $$ 
is a local martingale with predictable quadratic variation given as
$$\langle \clm^{n,R} \rangle_t = \sum_{i=1}^n \int_0^t \EE_Z([Z-A^{n,R}(s) + X^{n,R}_i(s)]^+)^2
w_i^{n,R}(\bfX^{n,R}(s)) ds.$$
Using Assumption \ref{Z} we see that
$$
\EE\langle \clm^{n,R} \rangle_T \le \sum_{i=1}^n\EE \int_0^T \EE_Z([Z-A^{n,R}(s) + X^{n,R}_i(s)]^+)^2
(w(X^{n,R}_i(s)-A^{n,R}(s))+1) ds\le
nT ( c_w + \vt).
$$
This shows that $\clm^{n,R}$ is in fact a martingale and therefore, using Assumption \ref{Z} again,
\begin{align}\label{eq:maxbd}
\EE A^{n,R}(T) &\le \xbd_{\max} + \sum_{i=1}^n \int_0^T \EE_Z([Z-A^{n,R}(s) + X^{n,R}_i(s)]^+)
(w(X^{n,R}_i(s)-A^{n,R}(s))+1) ds\notag\\
&\le
\xbd_{\max} + nT ( c_w + \mn) \doteq c_{T,n}.
\end{align}
Since $c_{T,n}$ does not depend on $R$, it follows that $\PP(\cup_{R\ge 1} \{A^{n,R}(T) < R/2\}) = 1$. The result now follows upon combining this with the observation in \eqref{eq:consis}, on noting that
$\bfX^n(t) \doteq \lim_{R\to \infty} \bfX^{n,R}(t)$, $t\ge 0$, defines a Markov process with generator $\cll_n$.
\hfill \qed

\section{A Moment Estimate on Overshoots}\label{s.lR}
The following lemma provides bounds for the first and second moments of the overshoot (over a certain level) of a renewal process.   This result is used in multiple instances in the subsequent proofs. A  similar argument for establishing a related estimate was used in \cite{budhiraja2024simple}.
\begin{lemma}\label{R}
Let $\{Z_i, i \in \NN\}$ be an iid sequence of $\RR_+$ valued random variables with $\mathbb{E}Z_1^{m+1} < \infty$ for some $m\ge 1$. Let $S_k \doteq \sum_{j=1}^k Z_j$, $k \in \NN$ and let $\{\xi(t), t \ge 0\}$ be the corresponding renewal process, namely $\xi(t) \doteq \max\{k \in \NN: S_k \le t\}$. For $l \in \RR_+$, let $\cO_l \doteq S_{\xi(l)+1}-l$ be the 
overshoot of level $l$.
Then,  $\sup_{l\in \RR_+} \mathbb{E}(\cO_l{\di ^m}) < \infty$.
\end{lemma}
\begin{proof}
By replacing $Z_i$ with $Z_i \one_{\{Z_i>0\}}$ we can assume without loss of generality that $Z_i>0$ for all $i$. Let for $x \in \RR_+$, $\tau(x) \doteq \xi(x)+1$. Then $\tau(x)$ is a $\clf_n \doteq \sigma\{Z_i: 1 \le i \le n\}$ stopping time for each $x\in \RR_+$. Note that, for $l\in \RR_+$, ${\di \cO}_l \le Z_{\tau(l)}$. Also note that for $n>1$
$$\{\tau(l) =n\} \subset \{[l-Z_n]^+ < S_{n-1} \le l\}
= \{\tau(l) > n-1\} \cap \{\tau([l-Z_n]^+) \le n-1\}.$$
For each $x \in \RR_+$, $\{\tau(l) > n-1\} \cap \{\tau([l-x]^+) \le n-1\}$
is $\clf_{n-1}$ measurable and is therefore independent of $Z_n$.
Let ${\di\EE}(Z_1^m) \doteq a$ and denote the distribution of $Z_1$ by $\theta_1$.
Then,
\begin{multline*}
\mathbb{E}(\cO_l^m) \le \mathbb{E}(Z_{\tau(l)}^m) \le a + \sum_{n=2}^{\infty} E(Z_n^m 
\one_{\{\tau(l) = n\}})
\le a+ \sum_{n=2}^{\infty} \mathbb{E}\left(Z_n^m \one_{\{\tau([l-Z_n]^+) \le n-1 < \tau(l)\}}\right)\\
= a+ \sum_{n=2}^{\infty} \int_0^{\infty} x^m 
\mathbb{P}(\tau([l-x]^+) \le n-1 < \tau(l)) \theta_1(dx) \le
a+ \int_0^{\infty} x^m \mathbb{E}(\tau(l) - \tau([l-x]^+)) \theta_1(dx).
\end{multline*}
Using the stopping time property of $\tau(\cdot)$ we see that for $l,x\in \RR_+$, $\tau(l) - \tau([l-x]^+)$ is stochastically dominated by $\tau(x)$. Combining this with the above display, we have that
\begin{equation}\label{eq:1115}
\mathbb{E}(\cO_l^m) \le a+ \int_0^{\infty} x^m \mathbb{E}\tau(x) \theta_1(dx).
\end{equation}
Let $m_1 = \mathbb{E}(Z_1)$ and choose $K \in (0,\infty)$ such that 
$\mathbb{E}(Z_1\one_{\{Z_1 >K\}}) \le m_1/4$. Then
$$\mathbb{E}(Z_1\one_{\{Z_1 \le K\}}) = m_1 - \mathbb{E}(Z_1\one_{\{Z_1 >K\}}) \ge 3m_1/4.$$
Also,
$$\mathbb{E}(Z_1\one_{\{Z_1 \le K\}}) \le \frac{m_1}{4} + K \mathbb{P}(Z_1 \ge m_1/4)
$$
and therefore
$$
\mathbb{P}(Z_1 \ge m_1/4) \ge \frac{1}{K} \mathbb{E}(Z_1\one_{\{Z_1 \le K\}}) - \frac{m_1}{4K}
\ge \frac{3m_1}{4K} - \frac{m_1}{4K} = \frac{m_1}{2K}.
$$
Define for $n \in \NN$ and $x \in \RR_+$
$$
C(n) = \sum_{j=1}^n \one_{\{Z_j \ge m_1/4\}}, \;
\zeta(x) \doteq \min\{n\ge 1: C(n) = \lceil 4x/m_1 \rceil\}.
$$
Note that, for $x \in \RR_+$,
$$\sum_{j=1}^{\zeta(x)} Z_j \ge m_1C(\zeta(x))/4 \ge 4m_1x/(4m_1) = x$$
and therefore $\tau(x) \le \zeta(x)+1$.
Thus, noting that $\zeta(x)$ is a negative Binomial random variable with chance of success bounded below by $m_1/(2K)$,
$$\mathbb{E}\tau(x) \le \mathbb{E}\zeta(x)+1
\le (1+ 4x/m_1) 2K/m_1 + 1 = 2K/m_1 + 8Kx/m_1^2 +1.$$
Using this estimate in \eqref{eq:1115} we have
$$
\mathbb{E}(\cO_l^m) \le a+ \int_0^{\infty} x^m [2K/m_1 + 8Kx/m_1^2 +1] \theta_1(dx)
\le \left(\frac{2K}{m_1} + {\di 2}\right) \mathbb{E}Z_1^m + \frac{8K}{m_1^2} \mathbb{E}Z_1^{m+1}.
$$
The result follows.
\end{proof}

\section{Proof of Theorems \ref{thm.mainFL} and \ref{mv}}\label{s.FL}
The proof of Theorem \ref{thm.mainFL} is broken up into the following steps.
\subsection{Mean bound}
As noted below Theorem \ref{thm.mainFL}, establishing an a priori control on the empirical mean process is a key ingredient in the proof of Theorem \ref{thm.mainFL}. The current subsection is dedicated to this.

For estimating $m_n(\cdot)$ we will use the following weaker form of 
Assumption \ref{ass.t1}.
\begin{assumption}\label{mun0}
    The initial configuration of the $n$-particle system satisfies
        \begin{equation}
        \lim_{B\to \infty} \limsup_{n\to \infty} \mathbb{P} \left[\frac{1}{n} \sum_{i=1}^n |X_i(0)|\geq B \right] =0. 
        \end{equation}
\end{assumption} 
To see the claim that Assumption \ref{mun0}  follows from 
 Assumption \ref{ass.t1}, note that 
 for any $B>2$,
\begin{multline*}
     \mathbb{P} \left[\frac{1}{n} \sum_{i=1}^n |X^n_i(0)|\geq B \right] \leq \mathbb{P} \left[ \frac{1}{n}\sum_{i=1}^n |X^n_i(0)|.\mathbf{1}(|X^n_i(0)|\geq B/2) \geq 1\right]\\ + \mathbb{P} \left[ \frac{1}{n}\sum_{i=1}^n |X^n_i(0)|.\mathbf{1}(|X^n_i(0)|< B/2) \geq B- 1\right]
     =\mathbb{P} \left[ \frac{1}{n}\sum_{i=1}^n |X^n_i(0)|.\mathbf{1}(|X^n_i(0)|\geq B/2) \geq 1\right].
\end{multline*} 
The claim is now immediate.


%
In what follows, frequently we will suppress $n$ in our notation and write $X^n_i$ as $X_i$.
\begin{theorem}\label{thm.mbd}
     Suppose that Assumptions \ref{Z} and \ref{mun0} hold. Then, for every $T \in (0,\infty)$, 
     \begin{equation}\label{mbd}
        \lim_{R\to \infty}\limsup_{n\to \infty} \mathbb{P}(m_n(T) \geq R) =0.
    \end{equation}
\end{theorem}
%
A key quantity in the proof of the above theorem is 
$$
\hat{m}_n(t) \doteq n^{-1} \sum_{k=1}^n (X_k(t)-m_n(t))^+, \ t\ge 0.
$$
This quantity gives the average overshoot of the particles from the mean value and the main idea in the proof is to first give an estimate on $\sup_{0\le t \le T} \hat m_n(t)$ and then leverage this estimate to control the growth of $m_n(T)$ in probability. Particles which are not too far behind the mean jump with `bounded' rates and this leads to better control on the increase in $\hat m_n$ due to these particles. As for particles far behind the mean, although they can jump with very high rates (thereby making the direct estimation of the mean process $m_n$ difficult), such a jump increases $\hat m_n$ only if the corresponding particle jumps over the mean in this single jump. This is where Assumption \ref{Z} plays a crucial role by roughly keeping the product of the size and rate of such `influential' jumps bounded. These facts enable control on the process $\hat m_n$. This control is then transferred to the  mean process $m_n$ through a detailed analysis of the contributions of the particles above and below the mean to the growth of $m_n$. A schematic outline of the mean control technique is given in Figure \ref{fig:mean-process-control}.


\begin{figure}[h]
\centering
\begin{tikzpicture}[
    node distance=1.25cm,
    box/.style={draw, rounded corners, align=center, text width=3.8cm, minimum height=1.0cm},
    arr/.style={->, thick, >=stealth}
]
\node[box] (far) {Particles far below the mean\\large rate $w(X_k-m_n)$};
\node[box, right=of far] (over) {Only jumps crossing the mean affect\\$(X_k-m_n)^+$};
\node[box, right=of over] (ass) {Overshoot control\\Assumption~\ref{Z}};
\node[box, below=of over] (hatm) {Bound on\\$\displaystyle \hat{m}_n(t)=n^{-1}\sum_k (X_k(t)-m_n(t))^+$};
\node[box, below=of ass] (mean) {Stopping time argument\\controls $m_n(T)$};

\draw[arr] (far) -- (over);
\draw[arr] (over) -- (ass);
\draw[arr] (ass) -- (hatm);
\draw[arr] (hatm) -- (mean);
\end{tikzpicture}
\caption{
Conceptual structure of the a priori mean process estimate. 
Particles far below the empirical mean may jump at very high rates, but such 
jumps contribute to $\hat m_n(t)$
only when they cross the mean. 
The corresponding overshoot is controlled through Assumption~\ref{Z}. 
This gives a bound on the  process
\(
\hat{m}_n(t) 
\)
which is then transferred, by a stopping time argument, to a high-probability 
bound on the empirical mean process $m_n(T)$.
}
\label{fig:mean-process-control}
\end{figure}

Recall the constant $c_w$ introduced in Assumption \ref{Z}.

\begin{lemma}\label{1}
Suppose Assumption \ref{Z} holds. For any $T>0$, there exist positive constants $c_1,c_2$, such that whenever $\bfX^n(0) = \bfx^n$ a.s., for some (non-random) $\bfx^n \in \RR^n$, then, for all $R>0$,
\begin{equation*}
 \sup_n \mathbb{P}\left[ \sup_{t \in [0,T]} \hat{m}_n(t) \geq \hat{m}_n(0) + (c_w+\mn w(0))T + R\right]  
        \leq \frac{4(c_1 + c_2 \hat m_n(0))}{R^2}. 
\end{equation*}
\end{lemma}
\begin{proof}
From the local martingale property in Proposition \ref{prop:exis} it follows that, with
\begin{multline}\label{L}
            L\hat{m}_n(t) \doteq \frac{1}{n} \sum_{k=1}^n \mathbb{E}_Z\left[\left(X_k(t)+Z-m_n(t)-\frac{Z}{n}\right)^+ -  (X_k(t)-m_n(t))^+\right] w(X_k(t)-m_n(t)) \\ + \frac{1}{n}\sum_{k=1}^n \sum_{j:j \neq k}^n \mathbb{E}_Z\left[ \left(X_j(t)-m_n(t)-\frac{Z}{n}\right)^+ - (X_j(t)-m_n(t))^+\right] w(X_k(t)-m_n(t)),
\end{multline}
 \begin{equation}\label{eq:533}
 \hat{M}_n(t) \doteq \hat{m}_n(t)-\hat{m}_n(0)-\int_0^t L\hat{m}_n(s)\; ds
 \end{equation}
 is a $\clf^n_t$-local martingale.
The (predictable) quadratic variation of $\hat{M}_n(t)$ can be bounded as
\begin{align}\label{eq:qvbd}
        \langle \hat{M}_n\rangle(t) &\leq \frac{2}{n^2}\sum_{k=1}^n
        \int_0^t   [A_k^n(s) + B_k^n(s)]w(X_k(s)-m_n(s)) \; ds, 
\end{align}
where
\begin{align*}
A_k^n(s) &\doteq 
        \mathbb{E}_Z \left[  \left(X_k(s)+Z-m_n(s)-\frac{Z}{n} \right)^+-(X_k(s)-m_n(s))^+  \right]^2 , \\ 
B_k^n(s) &\doteq         \mathbb{E}_Z \left[\sum_{j:j\neq k} \left\{ \left(X_j(s)-m_n(s)-\frac{Z}{n}\right)^+ - (X_j(s)-m_n(s))^+ \right\}\right]^2.
    \end{align*}
Recalling 
$\vt = \EE Z^2$ and using the Lipschitz property of the function $x \mapsto x^+$, the quadratic variation of $\hat{M}_n$ can now be further bounded as
\begin{align}
        \langle \hat{M}_n\rangle(t) &\leq 
        \frac{4}{n^2} \vt  \sum_{k=1}^n \int_0^t w(X_k(s)-m_n(s)) \; ds. \label{m2}
\end{align}
For $R \in (0,\infty)$, let (suppressing $n$ in the notation) $\tau_R \doteq \inf\{t\ge 0: m_n(t) \ge R\}$. Note that
$$\EE \sum_{k=1}^n \int_0^{T \wedge \tau_R} w(X_k(s)- m_n(s)) ds <\infty,$$
and hence from \eqref{m2} it follows that $\hat M_n(\cdot \wedge \tau_R)$ is  a martingale.


 Now we give a bound on $L\hat m_n(t)$.
Note that, for any $j$, 
\begin{equation}\label{eq:108}
    \left[ \left(X_j(t)-m_n(t)-n^{-1}Z\right)^+ - (X_j(t)-m_n(t))^+\right]  \leq 0.
\end{equation}

Also, if $X_k(t) \geq m_n(t)$, recalling $\mn \doteq \EE Z$,
$$\mathbb{E}_Z\left[ \left(X_k(t)+Z-m_n(t)-n^{-1}Z \right)^+ -(X_k(t)-m_n(t))^+\right] w(X_k(t)-m_n(t)) \leq \mn w(0),$$
and if $X_k(t) < m_n(t)$,
\begin{align*}
            &\mathbb{E}_Z \left[ \left(X_k(t)+Z-m_n(t)-n^{-1}Z\right)^+ -(X_k(t)-m_n(t))^+ \right] w(X_k(t)-m_n(t)) \\ 
            &\leq \mathbb{E}_Z[(Z-(m_n(t)-X_k(t)))^+]w(-(m_n(t)-X_k(t))) \leq c_w.
\end{align*}
 Combining the last two statements we get that, for all $t \geq 0$,
\begin{equation}\label{eq:109}
\frac{1}{n} \sum_{k=1}^n \mathbb{E}_Z\left[\left(X_k(t)+Z-m_n(t)-n^{-1}Z\right)^+ -  (X_k(t)-m_n(t))^+\right] w(X_k(t)-m_n(t)) \le (c_w+ \mn w(0)), 
\end{equation}
and combining this with \eqref{eq:108} and \eqref{L}, we get 
\begin{equation}\label{m1}
    L\hat{m}_n(t) \leq (c_w+ \mn w(0)), \quad \mbox{ for all } t \geq 0.
\end{equation}
%
The next step in the proof is to estimate $\EE\langle \hat M_n\rangle(T)$ and then  use it to estimate $\sup_{0\le t \le T} \hat m_n(t)$.

Using the martingale property of $\hat M_n(\cdot \wedge \tau_R)$ noted above along with \eqref{eq:533} and \eqref{eq:109} we now see that
\begin{multline}
\EE(\hat m_n(T \wedge \tau_R)) = \hat m_n(0) + \EE \int_0^{T\wedge \tau_R} L \hat m_n(s) ds
\le \hat m_n(0) + (c_w + \mn w(0))T\\
+  \mathbb{E} \left[\frac{1}{n}\int_0^{T\wedge \tau_R}\sum_{k=1}^n \sum_{j: j\neq k} \mathbb{E}_Z \left\{\left( X_j(s)-m_n(s)- n^{-1}Z\right)^+-(X_j(s)-m_n(s))^+ \right\} w(X_k(s)-m_n(s))\; ds \right].
\label{eq:948}
\end{multline}
As noted previously, the last term in the above display can be bounded above by $0$, however we seek a more useful bound.
For that, fix $A>0$ such that $\mathbb{P}(Z<A)>0$. Since the term inside the inner expectation in the last display is non-positive, we have that, for all $j,k$
\begin{align*}
&\mathbb{E} \left[\frac{1}{n}\int_0^{T\wedge \tau_R} \mathbb{E}_Z \left\{\left( X_j(s)-m_n(s)- n^{-1}Z\right)^+-(X_j(s)-m_n(s))^+ \right\} w(X_k(s)-m_n(s))\; ds \right] \\
&\le  \mathbb{E}\left[\frac{1}{n} \int_0^{T\wedge \tau_R}  \mathbb{E}_Z \left\{ \left( X_j(s)-m_n(s)- n^{-1}Z\right)^+-(X_j(s)-m_n(s))^+ \right\} \mathbf{1}(B_{j,s}) w(X_k(s)-m_n(s))\; ds \right],
\end{align*}
where $B_{j,s} \doteq \left\{X_j(s)\geq m_n(s)+ \frac{A}{n} \right\}$.
Next, on the set $B_{j,s}$,
\begin{equation*}
    \mathbb{E}_Z\left[\left(X_j(s)-m_n(s)-n^{-1}Z\right)^+ - (X_j(s)-m_n(s))^+ \right] \leq \mathbb{E}_Z\left[ -n^{-1}Z \mathbf{1}\{Z<A\}\right].
\end{equation*}
Thus, letting $p\doteq \mathbb{E}_Z\left[ Z \cdot \mathbf{1}\{Z<A\}\right]$,
on the set, $B_{j,s}$, 
$$\mathbb{E}_Z\left[\left(X_j(s)-m_n(s)-n^{-1}Z\right)^+ -(X_j(s)-m_n(s))^+ \right] \leq - \frac{p}{n}.$$
Therefore, the  term in the last line in \eqref{eq:948} is bounded above by 
\begin{multline*}
- \frac{p}{n^2}\mathbb{E}\left[ \int_0^{T\wedge \tau_R} \sum_{k=1}^n \sum_{j:j\neq k} w(X_k(s)-m_n(s)) \mathbf{1}(B_{j,s}) ds \right]\\
=
- \frac{p}{n^2} \mathbb{E} \left[ \int_0^{T\wedge \tau_R} \left[
\sum_{k=1}^n \sum_{j=1}^n w(X_k(s)-m_n(s))\mathbf{1}(B_{j,s}) - 
 \sum_{k=1}^n w(X_k(s)-m_n(s))\mathbf{1}(B_{k,s})\right] \; ds\right] \\
 \le - \frac{p}{n^2} \mathbb{E} \left[ \int_0^{T\wedge \tau_R} \left[\sum_{k=1}^n \sum_{j=1}^n w(X_k(s)-m_n(s)) \mathbf{1}(B_{j,s})\right] \; ds\right] + 
 n^{-1}p w\left(n^{-1}A\right)T.
\end{multline*}
Since $\hat{m}_n(T\wedge \tau_R)\geq 0$, from \eqref{eq:948} we then have that 
%
\begin{equation*}
    \frac{1}{n^2}\mathbb{E}\left[ \int_0^{T\wedge \tau_R} \sum_{k=1}^n \sum_{j=1}^n w(X_k(s)-m_n(s))\mathbf{1}(B_{j,s}) \right] \leq \frac{1}{p} \left[(c_w + \mn w(0))T + n^{-1}p w\left(n^{-1}A\right)T + \hat{m}_n(0)\right].
\end{equation*}
Using the fact that $\tau_R \to \infty$ as $R\to \infty$ and monotone convergence theorem, we have,  for every $n\in \NN$,
\begin{equation}\label{b}
     \frac{1}{n^2}\mathbb{E} \left[ \int_0^T \sum_{j=1}^n \sum_{k=1}^n w(X_k(s)-m_n(s)) \mathbf{1}(B_{j,s})\; ds\right] \leq C_T + \frac{1}{p}\hat{m}_n(0), 
\end{equation}
where $C_T \doteq p^{-1} ((c_w + \mn w(0))T +   p w(0)T)$.
Using this estimate we now have,
\begin{align}
        &\frac{1}{n^2}\mathbb{E}\left[\int_0^T \sum_{k=1}^nw(X_k(s)-m_n(s))\; ds \right] \\
        &\leq \frac{1}{n^2} \mathbb{E} \left[\int_0^T \sum_{k=1}^n w(X_k(s)-m_n(s)) \mathbf{1}[X_j(s)>m_n(s)-A, \; \mbox{ for all } j \in [n]]\; ds \right] \\
        & \quad + \frac{1}{n^2} \mathbb{E} \left[\int_0^T \sum_{k=1}^n w(X_k(s)-m_n(s))\mathbf{1}[X_j(s)\leq m_n(s)-A, \; \text{for some }j \in [n]]\; ds \right] \\
        &\leq n^{-1}w(-A)T
        + \frac{1}{n^2}\mathbb{E}\left[ \int_0^T \sum_{k=1}^n w(X_k(s)-m_n(s)) \mathbf{1}\left[X_j(s)\geq m_n(s)+ n^{-1}A, \text{ for some } j \in [n] \right]\; ds \right] \\
        &\leq n^{-1}w(-A)T
        + \frac{1}{n^2}\mathbb{E}\left[ \int_0^T \sum_{j,k=1}^n w(X_k(s)-m_n(s))   \mathbf{1}(B_{j,s})\; ds \right] \le n^{-1}w(-A)T + C_T + \frac{1}{p}\hat{m}_n(0).
    \end{align}\label{QV}
    In the second inequality above, we used the fact that $\sum_{j}(X_j(s) - m_n(s))=0$ and hence, $X_j(s)\leq m_n(s)-A$ for some $j$ implies $X_j(s)\geq m_n(s)+ n^{-1}A$ for some (different) $j$.
Using this in \eqref{m2} we have
$$\mathbb{E} [\langle \hat{M}_n\rangle (T)] \le c_1 + c_2 \hat m_n(0),$$
where $c_1 = 4\vartheta(w(-A)T+ C_T)$ and $c_2 = 4\vt p^{-1}$.
Applying Doob's maximal inequality and the estimate in \eqref{m1}, we now get, for $R \in (0,\infty)$,
\begin{multline*}
 \sup_n \mathbb{P}\left[ \sup_{t \in [0,T]} \hat{m}_n(t) \geq \hat{m}_n(0) + (c_w+\mn w(0))T + R\right]  \leq \sup_n \mathbb{P} \left[ \sup_{t \in [0,T]} \hat{M}_n(t)  \geq R \right] \\
        \leq \frac{4}{R^2} \sup_n \mathbb{E} [\langle \hat{M}_n\rangle (T)] 
        \leq \frac{4(c_1 + c_2 \hat m_n(0))}{R^2}. 
\end{multline*}
\end{proof}
\begin{proof}[{\bf Proof of Theorem \ref{thm.mbd}}]
Assume for now that $\bfX^n(0) = \bfx^n$ a.s., for some non-random $\bfx^n \in \RR^n$. We leverage the  estimate on $ \hat m_n(t)$ in Lemma \ref{1} to bound $m_n(T)$ in probability.

Denote, for fixed $R>0$,
$$\beta(n,T,R) \doteq \hat{m}_n(0) + (c_w+\mn w(0))T + R$$ 
 and consider the event, 
$$\cle_R^n \doteq \{ \hat{m}_n(0) \leq R,\; \hat{m}_n(t)< \beta(n,T,R) \; \mbox{ for all } t\in [0,T]\}.$$
Also, define $$\beta(T,R) \doteq (c_w+\mn w(0))T + 2R.$$ 
Note that,
\begin{equation}\label{eq:540}
\mbox{ on } \cle_R^n, \ \hat m_n(t) < \beta(T,R), \mbox{ for all } t \in [0,T].
\end{equation}
For $l \in \RR_+$, define $\cO_l$ as in Lemma \ref{R} where $Z_i$ are iid distributed as $\theta$. From this lemma and Assumption \ref{Z} we have that
$\lambda \doteq \sup_{l \in \RR_+} \EE \cO_l <\infty$.

%

Finally, define a sequence of stopping times (once again suppressing dependence on $n$ in the notation), $\{\sigma_k \}_{k\geq 0}$, as follows. Let $\sigma_0 =0 $. For $i \geq 1$, define
$$\sigma_i \doteq \inf \left\{ t \geq \sigma_{i-1} : m_n(t) \geq m_n(\sigma_{i-1}) + 2\lambda + (1+\mn)\beta(T,R)\right\} \wedge T.$$
Let $S_i\doteq \left\lbrace k: X_k(\sigma_{i-1})<m_n(\sigma_{i-1})\right\rbrace$ and
 define $$\underline{m}_n^i(t) \doteq \frac{1}{n}\sum_{S_i} X_k(t+\sigma_{i-1}) , \;\;\; \overline{m}_n^i(t) \doteq \frac{1}{n}\sum_{S_i^c} X_k(t+\sigma_{i-1}), \; t\geq 0.$$
 Recall the filtration $\clf_t^n$ from Proposition \ref{prop:exis} and note that $\sigma_i$ are $\clf_t^n$-stopping times.
 Also observe that,  for $\sigma_{i-1} \le t < \sigma_{i}$, the jump rate of the particle, $X_k(.)$, at time $t$, is bounded above by $w\left(X_k(t)-m_n(\sigma_{i-1})\right.$ $\left.- 2\lambda -(1+\mn)\beta(T, R)\right)$. 

 Let $\{Z_j, j \in \NN\}$ be iid random variables distributed as $\theta$ and let  $\{V^0_k(\cdot)\}$ be an iid collection of $\mathcal{J}(\theta)$-jump processes (cf. Section \ref{sec:notat}), and suppose that these collections and the underlying particle system are mutually independent. 
 Let $V_k(t) \doteq V^0_k(w(-2\lambda- (1+\mn)\beta(T,R))t))$. Then, {\di recalling that 
 $\cO_{x}$  denotes the overshoot} of level $x$ defined as in Lemma \ref{R} associated with the sequence $\{Z_i\}$,
\begin{align*}
&\cll\left(X_k\left((\sigma_{i-1}+t) \wedge \sigma_{i} \right) \one\{X_k(\sigma_{i-1})<m_n(\sigma_{i-1})\}\mid \mathcal{F}_{\sigma_{i-1}}\right)\\
&\quad\leq_d \cll\left((\cO_{k,i} + m_n(\sigma_{i-1}) + V_k(t))\one\{X_k(\sigma_{i-1})<m_n(\sigma_{i-1})\} \mid \mathcal{F}_{\sigma_{i-1}}\right), \;\; t \geq 0,\end{align*} 
where {\di$\cO_{k,i} = \cO_{m_n(\sigma_{i-1})-X_k(\sigma_{i-1})}$}. This follows upon bounding the process $X_k(\sigma_{i-1} + \cdot)$ by $m_n(\sigma_{i-1})$ until the first time $X_k$ jumps over the latter threshold. After this jump, the overshoot of $X_k$ of level $m_n(\sigma_{i-1})$ is given by $\cO_{k,i}$ and the jump rate of $X_k$ is subsequently bounded by $w(-2\lambda- (1+\mn)\beta(T,R))$ on the event $\{X_k(\sigma_{i-1})<m_n(\sigma_{i-1})\}$ until time $\sigma_i$.

Similarly, we can also observe that 
\begin{align*}
&\cll\left(X_k\left((\sigma_{i-1}+t) \wedge \sigma_{i} \right) \one\{X_k(\sigma_{i-1})\ge m_n(\sigma_{i-1})\}\mid \mathcal{F}_{\sigma_{i-1}}\right)\\
&\quad \leq_d \cll\left((X_k(\sigma_{i-1})  + V_k(t))\one\{X_k(\sigma_{i-1})\ge m_n(\sigma_{i-1})\}
\mid \mathcal{F}_{\sigma_{i-1}}\right), \;\; t \geq 0.
\end{align*}


%
In what follows, probabilities conditioned w.r.t. $\clf_{\sigma_i}$ will be denoted as $\PP_i$.
With this notation, using the first stochastic domination relation above,
\begin{multline}\label{eq:713}
            \mathbb{P}_{i-1}\left[ \underline{m}_n^i(t  \wedge (\sigma_{i}-\sigma_{i-1})) \geq n^{-1} \sum_{S_i}\left( \cO_{k,i}+m_n(\sigma_{i-1})\right) + \frac{\mn}{2}\beta(T, R) \right] \\ 
            \leq \mathbb{P}\left[ n^{-1}\sum_{k=1}^n V_k(t) \geq \frac{\mn\beta(T, R)}{2}\right].
\end{multline}
Now, denoting the cardinality of a finite set $A$ by $\#A$,  define, for $t\ge 0$, 
$$\underline{\alpha}_n(t) \doteq \frac{1}{n}\#\{k \in [n]: X_k(t)<m_n(t)\}, \;  \overline{\alpha}_n(t)\doteq 1- \underline{\alpha}_n(t).$$
Then, denoting $$t_0(T, R) \doteq \frac{\beta(T, R)}{4w(-2\lambda- (1+\mn)\beta(T, R))},$$
we have from \eqref{eq:713}, a.s.,
\begin{multline}\label{2}
    \limsup_{n\to \infty}\mathbb{P}_{i-1} \left[ \underline{m}_n^i(t_0(T, R)\wedge (\sigma_{i}-\sigma_{i-1}) ) \geq 2\lambda + \underline{\alpha}_n(\sigma_{i-1})    m_n(\sigma_{i-1}) + \frac{\mn}{2}\beta(T, R)\right] \\
    \leq \limsup_{n\to \infty}\mathbb{P}\left[n^{-1}\sum_{k=1}^n V^0_k(4^{-1}\beta(T, R)) \geq \frac{\mn}{2}\beta(T, R)\right] + 
    \limsup_{n\to \infty}\mathbb{P}_{i-1}\left[\frac{1}{n}\sum_{S_i} \cO_{k,i} >2\lambda\right] =0.
\end{multline}
For the last equality we have used law of large numbers for iid random variables
and Lemma \ref{R} together with our moment assumptions on $Z$ in Assumption \ref{Z} and the estimate
\begin{equation*}
    \mathbb{P}_i \left[ n^{-1} \sum_{S_i} \cO_{k,i} > 2\lambda \right] \leq \frac{\sup_{l\geq0} \mathbb{E} \cO_l^2}{4n\lambda^2},
\end{equation*}
where $\cO_l$ is as in Lemma \ref{R}. 

A similar estimate,  as in \eqref{eq:713}- \eqref{2}, using the second stochastic domination relation above, shows that, 
\begin{equation}\label{3}
    \limsup_{n\to \infty}\mathbb{P}_{i-1} \left[ \overline{m}_n^i(t_0(T, R)\wedge (\sigma_{i}-\sigma_{i-1})) ) \geq \overline{m}_n^i(0) + \frac{\mn}{2}\beta(T, R)\right] = 0\; \mbox{ a.s.}
\end{equation}
Note that $\overline{m}_n^i(0) = \overline{\alpha}_n(\sigma_{i-1})m_n(\sigma_{i-1}) + \hat{m}_n(\sigma_{i-1})$. 
Thus, recalling \eqref{eq:540},
$$\mathbb{P}\left[ \overline{m}_n^i(0) \geq \overline{\alpha}_n(\sigma_{i-1}) m_n(\sigma_{i-1}) + \beta(T, R), \, \cle_R^n \right] = 0.$$
Combining with \eqref{3}, a.s.,
\begin{equation}\label{4}
   \limsup_{n\to \infty} \mathbb{P}_{i-1} \left[ \overline{m}_n^i \left(t_0(T, R)\wedge (\sigma_{i}-\sigma_{i-1})\right) \geq \overline{\alpha}_n(\sigma_{i-1})m_n(\sigma_{i-1}) + (1+ 2^{-1}\mn)\beta(T, R), \, \cle_R^n \right] = 0.
\end{equation}
 Finally, using \eqref{2} and \eqref{4}, we obtain
$$\limsup_{n\to \infty}\mathbb{P}_{i-1}\left[ m_n\left( (\sigma_{i-1} + t_0(T, R))\wedge \sigma_{i} \right) \geq 2\lambda + m_n(\sigma_{i-1})+(1+\mn)\beta(T, R), \, \cle_R^n \right] = 0.$$
{\di From the definition of $\sigma_i$, on the set where $\sigma_{i} \leq \sigma_{i-1} +t_0(T, R)$, we must have that
$$m_n\left( (\sigma_{i-1} + t_0(T, R))\wedge \sigma_{i} \right) \geq 2\lambda + m_n(\sigma_{i-1})+(1+\mn)\beta(T, R).$$
Thus
$$
\mathbb{P}_{i-1}\left[\sigma_{i} \leq \sigma_{i-1} +t_0(T, R), \, \cle_R^n \right]
\le \mathbb{P}_{i-1}\left[ m_n\left( (\sigma_{i-1} + t_0(T, R))\wedge \sigma_{i} \right) \geq 2\lambda + m_n(\sigma_{i-1})+(1+\mn)\beta(T, R), \, \cle_R^n \right].
$$
Combining the above displays
\begin{multline} \label{5}
\limsup_{n\to \infty}\mathbb{P} \left[ \sigma_{i} \leq \sigma_{i-1} +t_0(T, R), \, \cle_R^n \right]
 \le \mathbb{E} \limsup_{n\to \infty} \mathbb{P}_{i-1}\left[\sigma_{i} \leq \sigma_{i-1} +t_0(T, R), \, \cle_R^n \right]\\
\le \mathbb{E} \limsup_{n\to \infty} \mathbb{P}_{i-1}\left[m_n\left( (\sigma_{i-1} + t_0(T, R))\wedge \sigma_{i} \right) \geq 2\lambda + m_n(\sigma_{i-1})+(1+\mn)\beta(T, R), \, \cle_R^n\right] =0,
\end{multline}
where the first inequality follows from Fatou's lemma.
}

%

The above estimate gives a control on the number of $\sigma_i$ that are less than $T$. The increase in the mean on each of the intervals $[\sigma_{i-1}, \sigma_i)$ is upper bounded  by $2\lambda + (1+\mn) \beta(T,R)$. We now estimate the `jump' of the mean at each instant $\sigma_i$.

For $i\geq 1$,  define $$J_i^n \doteq m_n(\sigma_i)-m_n(\sigma_i-).$$
Consider the point process $\{\xi_n^i(t), t\ge 0\}$ with arrival epochs comprising the collective jumps of all the $\{ X_j(\cdot):1\leq j\leq n\}$, arranged in chronological order, starting with the first jump after $\sigma_{i-1}$. More precisely, conditionally on $\mathcal{F}_{\sigma_{i-1}}$, the successive jump locations of the process $[\sum_{i=1}^n X_j(\cdot +\sigma_{i-1})-\sum_{j=1}^n X_j(\sigma_{i-1})]$ (ignoring the times when the jumps happened) form the arrival epochs of $\xi_n^i(\cdot)$. Thus, $\{\xi_n^i(t), t\ge 0\}$ is a renewal process for which the interarrival times have probability law  $\theta$. 
In particular, conditionally on $\mathcal{F}_{\sigma_{i-1}}$, $J_i^n$ has the same law as $n^{-1} \cO_{n(2\lambda +(1+\mn)\beta(T,R))}$, {\di where $\cO_l$ is as in Lemma \ref{R}}. Hence, using Lemma \ref{R}, we obtain 
\begin{equation}\label{6}
    \sup_{i\geq 1} \mathbb{E}J_i^n \leq \frac{1}{n}\sup_{l\geq 0} \mathbb{E}\cO_l < \infty.
\end{equation}
We can now give an estimate on $m_n(T).$ Using \eqref{5} and \eqref{6}, and letting $K\doteq \lceil T/t_0(T, R) \rceil$,
\begin{align*}
        &\mathbb{P} \left[ m_n(T) > 1 + \left( \frac{4Tw\left(-2\lambda-(1+\mn)\beta(T, R)\right)}{\beta(T, R)} +1\right)((1+\mn)\beta(T, R)+2\lambda) + m_n(0), \; \cle_R^n\right] \\
        &\leq \mathbb{P} \left[  \sigma_{i} - \sigma_{i-1} \leq t_0(T, R) \; \text{, for some} \; 1\leq i \leq K, \; \cle_R^n \right] \\ & \quad + \mathbb{P} \left[ J_i^n > K^{-1} \text{, for some } 1\leq i\leq K, \; \cle_R^n \right] \\
        &\leq \sum_{i=1}^{K} \mathbb{P} [\sigma_{i} - \sigma_{i-1} \leq t_0(T, R), \; \cle_R^n]  + \sum_{i=1}^{K}     \left[ J_i^n > K^{-1}, \; \cle_R^n \right] \\
        &\leq \sum_{i=1}^{K} \mathbb{P} [\sigma_{i} - \sigma_{i-1} \leq t_0(T, R), \; \cle_R^n]  + \frac{K^2}{n} \cdot \sup_{l\geq 0}\mathbb{E}\cO_l,
      \end{align*}
      where in the last line we have used \eqref{6}. Combining the above estimate with \eqref{5} 
      and letting $$\Gamma(T,R) \doteq 1 + \left(\frac{4Tw \left(-2\lambda- (1+\mn)\beta(T, R)\right)}{\beta(T, R)} +1\right)((1+\mn)\beta(T, R)+2\lambda),$$
      we now have that
      \begin{equation}\label{eq:550}
      \limsup_{n\to \infty}\mathbb{P} \left[ m_n(T) >\Gamma(T,R) + m_n(0), \; \cle_R^n\right]=0.
      \end{equation}

 Letting
 $$\hat{\cle}_R^n \doteq \{ \hat{m}_n(t) < \beta(n,T,R), \, \forall t\in [0,T]\},$$ 
 we get
\begin{align*}
    &\limsup_{n\to \infty} \mathbb{P}[m_n(T) \geq \Gamma(T,R) + m_n(0)]\one\{ \hat{m}_n(0) \leq R \} \\
    & \le {\di  \limsup_{n\to \infty} \mathbb{P}[m_n(T) \geq \Gamma(T,R) + m_n(0), \hat{\cle}_R^n]\one\{ \hat{m}_n(0) \leq R \}
    + \limsup_{n\to \infty} \mathbb{P}[(\hat{\cle}_R^n)^c ]\one\{ \hat{m}_n(0) \leq R \}}\\
    &{\di =} \limsup_{n\to \infty} \mathbb{P}[m_n(T)\geq \Gamma(T,R) +m_n(0), \cle_R^n ] + \limsup_{n\to \infty} \mathbb{P}[(\hat{\cle}_R^n)^c ]\one\{ \hat{m}_n(0) \leq R \} \leq \frac{4(c_1 + c_2R)}{R^2},
    \end{align*}
    {\di where the equality on the last line uses that $\hat{\cle}_R^n \cap \{ \hat{m}_n(0) \leq R \}= {\cle}_R^n$} and the last inequality uses \eqref{eq:550} and Lemma \ref{1}.
Until now we had taken the initial data $\bfX^n(0) = \bfx^n$ a.s. for some non-random $\bfx^n \in \RR^n$.
Now consider the general case where the initial data instead of being non-random satisfies Assumption \ref{mun0}. In that case we have from the above estimate, and an application of Fatou's lemma, with $M(T,R)=\Gamma(T,R) + R$,
\begin{align*}
    \limsup_{n\to \infty} \mathbb{P}[m_n(T) \geq M(T,R) ] 
    \le \frac{4(c_1 + c_2R)}{R^2} + \limsup_{n\to \infty} \PP[m_n(0)\geq R]
    + \limsup_{n\to \infty} \PP[\hat m_n(0)\geq R].
\end{align*}
The result now follows on noting that, from Assumption \ref{mun0}, we have
$$\lim_{R\to \infty} \limsup_{n\to \infty} \mathbb{P}[m_n(0)\geq R] + \lim_{R\to \infty} \limsup_{n\to \infty} \mathbb{P}[\hat{m}_n(0)\geq R]=0.$$
\end{proof}
%
\subsection{Tightness}
\label{sec:tight}
In this section we will apply Theorem \ref{thm.mbd} to establish the tightness of $\{\mu_n, n \in \NN\}$ in $\cld([0,\infty): \clp_1(\RR))$.
The first lemma below gives pointwise tightness.

        %

\begin{lemma}\label{lem.tght}
Suppose Assumptions \ref{Z} and \ref{ass.t1} hold. Then, for any $t \geq 0$, $\eta >0$, $$\lim_{R\to \infty} \limsup_{n\to \infty} \mathbb{P}\left( \int |x| \mathbf{1}(|x|\geq R) \mu_n(t,dx)\geq \eta \right)=0.$$
Consequently, for any $t\geq 0$, $\{\mu_n(t)\}_{n\geq 1}$ is tight in $\mathcal{P}_1(\mathbb{R})$.
\end{lemma}
\begin{proof}
Note that Assumption \ref{ass.t1} ensures that $\{\mu_n(0)\}_{n\geq 1}$ is tight in $\mathcal{P}_1(\mathbb{R})$ (cf. \cite[Lemma D.17]{feng2006large}).
Consider now $t>0$.
For any $\eta >0$,
\begin{multline}
     \mathbb{P}\left( \int |x| \mathbf{1}(|x|\geq R) \mu_n(t,dx)\geq \eta \right) = \mathbb{P}\left(\frac{1}{n}\sum_{i=1}^n |X_i^n(t)| \mathbf{1}\{|X_i^n(t)|\geq R \} \geq \eta \right) \\
        \leq \mathbb{P}\left(\frac{1}{n}\sum_{i=1}^n (X_i^n(t))^- \mathbf{1}\{X_i^n(t)\leq -R \} \geq \eta/2 \right) + \mathbb{P}\left(\frac{1}{n}\sum_{i=1}^n (X_i^n(t))^+ \mathbf{1}\{X_i^n(t)\geq R \} \geq \eta/2 \right)\\
        = T_1^n(R)  + T_2^n(R).\label{eq:1130}
  \end{multline} 
Since $t \mapsto X_i^n(t)$ is increasing, we have using  Assumption \ref{ass.t1}, 
\begin{align}\label{t.i}
    \lim_{R\to \infty} \limsup_{n\to \infty} T_1^n(R)\leq \lim_{R\to \infty} \limsup_{n\to \infty} \mathbb{P} \left(\frac{1}{n} \sum_{i=1}^n (X_i^n(0))^- \mathbf{1}\{X_i^n(0)\leq -R \} \geq \eta/2 \right) =0.
\end{align}
To estimate $T_2^n(R)$,
for each $i \in [n]$ and $A>0$, define a jump  process $\{\bar{X}_i^{n,A}(s), s\ge 0\}$, pathwise, as follows. 
 Define $$\tau_i^n \doteq \inf \{s\geq  0: X_i^n(s)\geq 0 \}.$$ 
 Then for $0\le s \le \tau_i^n$,
$\bar{X}_i^{n,A} (s)\doteq X_i^n(s)$, and for $s>\tau_i^n$, 
$$\bar{X}_i^{n,A}(s)=X_i^n(\tau_i^n)+ V_i( w(-A) (s-\tau_i^n)),$$  
where, $\{V_i, i \in [n]\}$ are iid $\mathcal{J}(\theta)$-jump processes which are independent of the collection $\{X^n_i(0),  i\in [n]\}$. 

Note that $\{X_i^n, i \in [n]\}$ and $\{\bar{X}_i^{n,A}, i \in [n]\}$ can be naturally constructed on the same probability space such that, on the event $\{ m_n(t)\leq A\}$,
$X_i^n(s)\leq \bar{X}_i^{n,A}(s)$, for all  $s\leq t$.
Also, on the event $\{m_n(t)\le A\}$, using the above construction, we get 
$$(X_i^n(t))^+ \leq (\bar{X}_i^{n,A}(t))^+ \leq (X_i^n(\tau_i^n)) + V_i( w(-A) (t-\tau_i^n))\one\{t\ge \tau^n_i\}, \quad  \; 1\leq i\leq n.$$
Thus, using the inequality $(x+y)\one\{x+y \ge A\} \le 2x\one\{x \ge A/2\} 
+2y\one\{y \ge A/2\}$,
\begin{align}\label{t.ii}
     T_2^n(R)&\leq \mathbb{P}(m_n(t)\geq A) + \mathbb{P}\left( \frac{1}{n}\sum_{i=1}^n (\bar{X}_i^{n,A}(t))^+ \mathbf{1}\{\bar{X}_i^{n,A}(t)\geq R \} \geq \frac{\eta}{2} \right) \\
    &\leq \mathbb{P}(m_n(t)\geq A) + \mathbb{P}\left( \frac{2}{n}\sum_{i=1}^n X_i^n(\tau_i^n) \mathbf{1} \left\{X_i^n(\tau_i^n)\geq \frac{R}{2} \right\} \geq \frac{\eta}{4} \right)\\
    &\quad + \mathbb{P}\left( \frac{2}{n}\sum_{i=1}^n V_i(w(-A)(t-\tau^n_i)^+) \mathbf{1}
    \left\{V_i( w(-A) (t-\tau_i^n)^+)\geq \frac{R}{2} \right\} \geq \frac{\eta}{4} \right).
\end{align}
Next, 
\begin{align*}
    &\mathbb{P}\left( \frac{2}{n}\sum_{i=1}^n X_i^n(\tau_i^n) \mathbf{1}\left\{X_i^n(\tau_i^n)\geq \frac{R}{2} \right\} \geq \frac{\eta}{4} \right) \\
    &\leq \mathbb{P}\left( \frac{1}{n}\sum_{i=1}^n X_i^n(0) \mathbf{1}\left\{X_i^n(0)\geq \frac{R}{4} \right\} \geq \frac{\eta}{16} \right) + \mathbb{P}\left( \frac{1}{n}\sum_{i=1}^n \cO_{X_i^n(0)^-} \mathbf{1}\left\{\cO_{X_i^n(0)^-}\geq \frac{R}{4} \right\} \geq \frac{\eta}{16} \right).
\end{align*}
Here, $\cO_x$ is the overshoot at level $x$ by the renewal process $\xi$ as defined in Lemma \ref{R}. \\
Thus, by Assumption \ref{ass.t1} and Lemma \ref{R}, we get
\begin{equation}\label{t.iii}
    \lim_{R\to \infty} \limsup_{n\to \infty} \mathbb{P} \left( \frac{2}{n}\sum_{i=1}^n X_i^n(\tau_i^n) \mathbf{1} \left\{X_i^n(\tau_i^n)\geq \frac{R}{2} \right\}\geq \frac{\eta}{4}\right) =0.  
\end{equation}
Moreover, since $$\mathbb{E}[V_i(w(-A)t)]^2 \leq \mathbb{E}Z^2 \cdot[w(-A)^2t^2 + w(-A)t],$$
 we have
\begin{equation}\label{t.iv}
    \limsup_{R\to \infty}\limsup_{n\to \infty} \mathbb{P} \left( \frac{2}{n}\sum_{i=1}^n 
    V_i(w(-A)t) \mathbf{1}\left\{ V_i(w(-A)t)\geq \frac{R}{2}\right\} \geq \frac{\eta}{2} \right) =0.   
\end{equation}
Using \eqref{t.iii} and \eqref{t.iv} in \eqref{t.ii}, for any $A>0$,
$$\lim_{R\to \infty} \limsup_{n\to \infty} T_2^n(R)\leq \limsup_{n\to \infty} \mathbb{P}(m_n(t)\geq A).$$
Hence, using Theorem \ref{thm.mbd}, and sending $A\to \infty$,
$$\lim_{R\to \infty} \limsup_{n\to \infty} T_2^n(R)=0.$$
This, along with \eqref{t.i} and \eqref{eq:1130}, completes the proof of the lemma.
\end{proof}

For the next lemma we will use the following slightly weaker assumption than Assumption \ref{ass.char}.
%
%
\begin{assumption}\label{ass.t2}
    For any $c \in \mathbb{R}$, 
        \begin{equation}
            \lim_{B \to \infty} \limsup_{n\to \infty} \mathbb{P}\left(\int w(x-c)\mu_n(0,dx) \geq B \right)=0
        \end{equation}
\end{assumption}
 The claim that  Assumption \ref{ass.char} is in fact stronger than Assumption \ref{ass.t2} can be seen as follows.
Suppose Assumption \ref{ass.char} holds. Then for any $\eta > 0$, we can choose $A>0$ such that $$\limsup_{n\to \infty} \mathbb{P} \left[ \int_{-\infty}^{-A} w(x-c)\mu_n(0,dx) \geq 1\right] < \eta.$$
Let $B_0 \doteq w(-A-c)+1$ Then, for any $B\geq B_0$, 
\begin{align*}
    \limsup_{n\to \infty} \mathbb{P} \left[ \int_{-\infty}^{\infty} w(x-c)\mu_n(0,dx) \geq B\right]
    &\leq \limsup_{n\to \infty} \mathbb{P} \left[ \int_{-\infty}^{-A} w(x-c)\mu_n(0,dx) + w(-A-c) \geq B\right] \\
    &\leq \limsup_{n\to \infty} \mathbb{P} \left[ \int_{-\infty}^{-A} w(x-c)\mu_n(0,dx) \geq 1\right] 
    < \eta.
\end{align*}
The claim is now immediate.

The next lemma gives the key estimate on oscillations that is needed for establishing tightness in the path space.
\begin{lemma}\label{lem.tght.b}
    Suppose Assumptions  \ref{Z}, \ref{mun0} and \ref{ass.t2} hold. Then, for any $T, \epsilon \in (0,\infty)$ 
    $$\lim_{\delta \to 0} \limsup_{n\to \infty} \mathbb{P} \left( \sup_{\substack{t,s \in [0,T]: \\ 0\leq t-s \leq \delta}} \mathcal{W}_1(\mu_n(s),\mu_n(t))> \epsilon\right) =0.$$ 
\end{lemma}
\begin{proof}
    \noindent Observe that for any $f \in \text{Lip}_1$ and $0<s<t \le T$,
\begin{multline*}
  |\langle f,\mu_n(t)\rangle -\langle f, \mu_n(s)\rangle| = \left | \frac{1}{n}\sum_{i=1}^n (f(X^n_i(t)) - f(X^n_i(s)) \right| \le \frac{1}{n}\sum_{i=1}^n |X^n_i(t)) - X^n_i(s)|\\ = 
        \frac{1}{n}\sum_{i=1}^n (X^n_i(t)) - X^n_i(s)) = \langle x,\mu_n(t)\rangle-\langle x,\mu_n(s)\rangle,
   \end{multline*}
   where the first equality on the second line is from the monotonicity of $t\mapsto X^n_i(t)$.
 Hence, it suffices to show that, for any $T, \epsilon \in (0,\infty)$,
 \begin{equation}
\lim_{\delta \to 0} \limsup_{n\to \infty} \mathbb{P} \left( \sup_{\substack{t,s \in [0,T]: \\ 0\leq t-s \leq \delta}} (m_n(t)- m_n(s)) > \epsilon\right) =0. \label{eq:406}
 \end{equation}
For $i \in [n]$ and $A>0$, define a jump process $X_i^{n,A}$ as follows.
\begin{align*}
X_i^{n,A}(t) = X^n_i(0) + V_i(w(X^n_i(0) -A)t), \ t \ge 0,
\end{align*}
where $\{V_i, i\in [n]\}$,  are iid $\mathcal{J}(\theta)$-jump processes that are independent of the collection $\{X^n_i(0),  i\in [n]\}$.
We now give a distributionally equivalent construction of $\bfX^n$ which is coupled with the dynamics of 
$\bfX^{n,A} = (X_i^{n,A}, i \in [n])$ as follows. Define
$$\tilde X^n_i(t), i \in [n], \ \tilde m_n(t) \doteq \frac{1}{n} \sum_{i=1}^n \tilde X^n_i(t),\, \text{ for } t \le \tilde \tau^n_A \doteq \inf\{s\ge 0: \tilde m_n(s) \ge A\}$$
recursively as $\tilde X^n_i(0) = X^n_i(0)$ and for $0< t< \tilde \tau^n_A$, conditioned on $\clf^{n,A}_{t-}$,
$$\tilde X^n_i(t) - \tilde X^n_i(t-)  =
\begin{cases}
X_i^{n,A}(t) - X_i^{n,A}(t-) &\mbox{ with probability } \frac{w(\tilde X_i^n(t-)-\tilde m_n(t-))}{w(X_i^n(0)-A)}\\
0 &\mbox{ with probability } 1-\frac{w(\tilde X_i^n(t-)-\tilde m_n(t-))}{w(X_i^n(0)-A)}
\end{cases}
$$
where $\clf^{n,A}_{s} \doteq \sigma \{\bfX^{n,A}(u), u \le s\}$, $s\ge 0$.
For $t\ge \tilde \tau^n_A$, $\tilde\bfX^n (t) \doteq \{\tilde X^n_i(t), i \in [n]\}$ evolves independently of
$\bfX^{n,A}$ according to a Markov process with generator $\cll_n$ defined in \eqref{eq:gene}.
It then follows that $\tilde\bfX^n$ has the same distribution as $\bfX^n$.
By construction it also follows that 
$$\tilde X_i^n(t)\leq X_i^{n,A}(t),\; \mbox{ and, } \tilde X_i^n(t)-\tilde X_i^n(s) \leq X_i^{n,A}(t)-X_i^{n,A}(s), \; 0\leq s\leq t \leq \tilde \tau_n^A. $$
We thus have that
\begin{equation}\label{eq:502}
    \begin{split}
        \mathbb{P}\left( \sup_{\substack{t,s \in [0,T]: \\ 0\leq t-s \leq \delta}} (m_n(t)-m_n(s))> \epsilon \right) & = \mathbb{P}\left( \sup_{\substack{t,s \in [0,T]: \\ 0\leq t-s \leq \delta}} (\tilde m_n(t)-\tilde m_n(s))> \epsilon \right)\\
        &\leq \mathbb{P}(m_n(T)\geq A) + \mathbb{P}\left( \sup_{\substack{t,s \in [0,T]: \\ 0\leq t-s \leq \delta}} (m_n^A(t)-m_n^A(s))> \epsilon \right).
    \end{split}
\end{equation}
From the Markov property of $\bfX^{n,A}$ it follows that
 $$m_n^A(t)=m_n^A(0)+t\mn\int w(x-A)\mu_n(0,dx) +M_n^A(t),$$ 
 where $M_n^A(\cdot)$is a $\{\clf^{n,A}_t\}$-local martingale.
Thus, \begin{equation}\label{eq:501}
\sup_{\substack{t,s \in [0,T]: \\ 0\leq t-s \leq \delta}} (m_n^A(t)-m_n^A(s)) \leq \delta \mn\int w(x-A)\mu_n(0,dx)+ 2 \bar{M}_n^A,
\end{equation}
where $\bar{M}_n^A\doteq \sup_{t\in [0,T]} |M_n^A(t)|$.

Recalling that $\vt = \EE Z^2$, we  can bound the quadratic variation of $M_n^A(\cdot)$ as,
\begin{align*}
\langle M_n^A \rangle_t &\le  \frac{\vt \cdot t}{n} \int w(x-A) d\mu_n(0, dx).
\end{align*} 
Using Doob’s maximal inequality,  for any $\eta > 0$, 
\begin{align*}
\mathbb{P} (\Bar {M}_n^A  \geq \eta) & = \mathbb{P} \left(\Bar {M}_n^A  \geq \eta, \int w(x-A)\mu_n(0,dx) \leq B \right) + \mathbb{P} \left(\int w(x-A)\mu_n(0,dx) > B \right)\\
&\le \frac{4B\vt T}{n\eta^2} + \mathbb{P} \left(\int w(x-A)\mu_n(0,dx) > B \right).
\end{align*}
Sending $n\to \infty$, we obtain 
$$\limsup_{n\to \infty} \mathbb{P} (\Bar {M}_n^A  \geq \eta) \leq \limsup_{n\to \infty} \mathbb{P}\left(\int w(x-A)\mu_n(0,dx) \geq B \right).$$
Finally taking $B\to \infty$ and using Assumption \ref{ass.t2} we get, for any $\eta>0$,
\begin{equation}\label{mglebd}
    \limsup_{n\to \infty} \mathbb{P} (\Bar {M}_n^A  \geq \eta) =0.
\end{equation}
From \eqref{eq:501} and \eqref{eq:502}, we have
\begin{align*}
    \mathbb{P}\left( \sup_{\substack{t,s \in [0,T]: \\ 0\leq t-s \leq \delta}} (m_n(t)-m_n(s))> \epsilon \right)
    \quad\leq \mathbb{P}(m_n(T)\geq A) +\mathbb{P}\left( \int w(x-A) \mu_n(0,dx)> \frac{\epsilon}{2\delta\mn } \right) + \mathbb{P}\left( \bar{M}_n^A > \frac{\epsilon}{4} \right).
\end{align*}
Hence, using Assumption \ref{ass.t2}, and \eqref{mglebd}, for any $\epsilon>0$, $A>0$, 
$$\lim_{\delta\to 0} \limsup_{n\to \infty} \mathbb{P} \left( \sup_{\substack{t,s \in [0,T]: \\ 0\leq t-s \leq \delta}} (m_n(t)-m_n(s))> \epsilon\right) \leq \limsup_{n\to \infty} \mathbb{P} (m_n(T)\geq A).$$
The statement in \eqref{eq:406} now follows 
on sending $A\to \infty$ and using  Theorem  \ref{thm.mbd}.
\end{proof}
\begin{theorem}\label{thm.tght}
   Suppose Assumptions \ref{Z}, \ref{ass.t1}, and \ref{ass.t2} hold. Then $\mu_n$ is $\mathcal{C}$-tight in $\mathcal{D}([0,\infty):\mathcal{P}_1(\mathbb{R}))$.
\end{theorem}
\begin{proof}
    Using \cite[Theorem 3.7.2, Page 128]{Ethier2005MarkovPC}, tightness of $\mu_n(.)$ in $\mathcal{D}([0,\infty):\mathcal{P}_1(\mathbb{R}))$ is immediate from Lemma \ref{lem.tght} and \ref{lem.tght.b}, which respectively prove parts (a) and (b) of that theorem. 
    The $\clc$-tightness is a consequence of the observation that, for  $n \in \NN$,
    $\sup_{t\ge 0} \EE \mathcal{W}_1(\mu_n(t), \mu_n(t-)) \le \frac{\mn}{n}$.
\end{proof}

\subsection{Characterization of Limit Points}
\label{sec:char}
In this section we give a characterization of the weak limit points of the sequence $\{\mu_n\}_{n\in \NN}$ for which the tightness was shown in Section \ref{sec:tight}.
\begin{theorem}\label{thm.ch}
    Suppose Assumptions \ref{Z}, \ref{ass.w}, \ref{ass.t1}, and  \ref{ass.char} are satisfied. Then $\{\mu_n\}_{n \in \NN}$ is $\mathcal{C}$-tight in $\mathcal{D}([0,\infty):\mathcal{P}_1(\mathbb{R}))$ and
    subsequential distributional limit points  $\mu$ of $\mu_n$ satisfy the McKean-Vlasov equation \eqref{eq:mfeq1}.
\end{theorem}
\begin{proof}
Recall from discussion below Assumption \ref{ass.t2} that this assumption is implied by Assumption 
\ref{ass.char}. Thus, under the assumptions of the current theorem, we have from Theorem \ref{thm.tght}
that $\mu_n$ is $\mathcal{C}$-tight in $\mathcal{D}([0,\infty):\mathcal{P}_1(\mathbb{R}))$.
Now let $\mu$ be a subsequential distributional limit point   of $\mu_n$.
 Without loss of generality, by appealing to Skorohod representation theorem, we can assume that
 $\mu_n \to \mu$ a.s. in $\mathcal{D}([0,\infty):\mathcal{P}_1(\mathbb{R}))$ and that $\mu \in \mathcal{C}([0,\infty):\mathcal{P}_1(\mathbb{R}))$ a.s..
 This, in particular, says that for every $f \in \mbox{Lip}_1$ and $T>0$, as $n\to \infty$,
 \begin{equation}\label{eq:737}
\sup_{0\le t \le T} |\langle f, \mu_n(t)\rangle - \langle f, \mu(t)\rangle| \to 0, \mbox{ a.s.. }
 \end{equation}
Fix $\epsilon, \eta, T \in (0, \infty)$. Then,  using Theorem \ref{thm.mbd}, we can choose $R>0$ such that \begin{equation}\label{eq:450}
\limsup_{n\to \infty} \mathbb{P}[m_n(T)> R] \leq \frac{\eta}{2}.\end{equation}
Furthermore, since $m_n(T) \to m(T) \doteq \langle x, \mu(T)\rangle$, a.s., we have by Fatou's lemma that
\begin{equation}\label{eq:451}
\mathbb{P}[m(T)> R] \leq \frac{\eta}{2}.\end{equation}
Next, using Assumption \ref{ass.char},  choose $A>0$ such that 
\begin{equation}\label{eq:445}
\limsup_{n\to \infty} \mathbb{P} \left[ \int_{(-\infty,-A)} w(x-R) \mu_n(0,dx) > \frac{\epsilon}{2T}\right] \leq \frac{\eta}{2}.\end{equation}
Using the fact that $\mu_n(0) \to \mu(0)$ a.s., and lower semicontinuity of the map
$$\gamma \mapsto \one_{(\epsilon/2T, \infty)} \left\{\int_{(-\infty,-A)} w(x-R) \gamma(dx)\right\}$$
from $\clp_1(\RR) \to \RR_+$, we have from Fatou's lemma that
\begin{equation}\label{eq:448}
\mathbb{P} \left[ \int_{(-\infty,-A)} w(x-R) \mu(0,dx) > \frac{\epsilon}{2T}\right] \leq \frac{\eta}{2}.
\end{equation}
Using \eqref{eq:445} and \eqref{eq:450}, and monotonicity of $t\mapsto \mu_n(t)$ (under the stochastic ordering $\le_d$) we have that
\begin{align}
    &\limsup_{n\to \infty} \mathbb{P} \left[ \int_0^T \int_{(-\infty,-A)} w(x-m_n(s)) \mu_n(s,dx) ds > \frac{\epsilon}{2}\right] \\
    &\leq \limsup_{n\to \infty} \mathbb{P}[m_n(T)> R] + \limsup_{n\to \infty} \mathbb{P} \left[ T
    \int_{(-\infty,-A)} w(x-R)\; \mu_n(0,dx)> \frac{\epsilon}{2}\right] \leq \eta.\label{ch2}
\end{align}
Using,  in addition, \eqref{eq:448} and \eqref{eq:451} we see that the same estimate holds if either of $m_n$ or $\mu_n$ are replaced by $m$ or $\mu$, respectively.
In particular using this estimate with $(m, \mu)$ instead of $(m_n, \mu_n)$ and recalling that $\eta, \epsilon$ are arbitrary, we see that
$$\int_0^T \langle w(\cdot-m(s)),  \mu(s)\rangle ds < \infty \mbox{ a.s. },
$$
proving the first property in \eqref{eq:mfeq1}.

The  estimate in \eqref{ch2} shows that for any $f\in \mbox{Lip}_1$, and $g_f$ defined as in 
\eqref{eq:gf},
\begin{equation}\label{ch3}
    \limsup_{n\to \infty} \mathbb{P} \left[ \int_0^T \int_{(-\infty,-A)} |g_f(x)| w(x-m_n(s))\, \mu_n(s,dx)\, ds> \frac{\mn\epsilon}{2} \right] \leq \eta,
\end{equation}
where we have used the fact that, since $f \in \mbox{Lip}_1$, $g_f(x) \le \mn$ for all $x$. Furthermore,  the same estimate as in the last display holds if $\mu_n$ and $m_n$ are replaced by $\mu$ and $m$ in the above display.
Thus noting that
\begin{align*}
&\left|\int_{-\infty}^{\infty} g_f(x) (w(x-m_n(s)) - w((x\vee (-A))-m_n(s))) \mu_n(s,dx)\right| \\
&\quad \le 2 \int_{(-\infty,-A)} |g_f(x)| w(x-m_n(s))\, \mu_n(s,dx),
\end{align*}
we have
\be\label{eq:804a}
 \limsup_{n\to \infty} \mathbb{P}\left[
 \sup_{0\le t \le T} \left|\int_0^t \int_{-\infty}^{\infty} g_f(x) (w(x-m_n(s)) - w((x\vee (-A))-m_n(s))) \mu_n(s,dx) ds \right| > \mn\epsilon
 \right] \le \eta.
\ee
Similar calculation shows that
\be\label{eq:804b}
 \limsup_{n\to \infty} \mathbb{P}\left[
 \sup_{0\le t \le T} \left|\int_0^t \int_{-\infty}^{\infty} g_f(x) (w(x-m(s)) - w((x\vee (-A))-m(s))) \mu(s,dx) ds \right| > \mn\epsilon
 \right] \le \eta.
\ee

Next, using  Assumption \ref{ass.w} we see that, on the set $\{ m_n(T)\leq R, \; m(T)\leq R\}$,   as $n\to \infty$,
\begin{align*}
    \theta_1^{n,A} &\doteq \sup_{t\in [0,T]} \left|\int_0^t \int_{-\infty}^\infty \left[g_f(x)w((x\vee (-A))-m_n(s))  -  g_f(x)w((x\vee (-A))-m(s))\right] \mu_n(s,dx)ds\right| \\
    &\leq C_{R,A}\cdot T \sup_{t\in [0,T]} |m_n(t)-m(t)|  \to 0, \mbox{ a.s.},
\end{align*}
where $C_{R,A} \doteq \mn L(-A-R)$,  $L(\cdot)$ is as introduced in Assumption \ref{ass.w}, and last convergence is from \eqref{eq:737}.
Hence, 
\be\label{eq:804c}
\limsup_{n\to \infty} \mathbb{P} [\theta_1^{n,A} > \mn\epsilon ] \leq \limsup_{n\to \infty} \mathbb{P} [m_n(T)> R] +  \mathbb{P} [m(T)> R] \le \eta,
\ee
where we have used \eqref{eq:450} and \eqref{eq:451}.
Next, from a.s. convergence of  $\mu_n\to \mu$, we have, as $n\to \infty$,
\begin{equation}\label{ch5}
   \sup_{t \in [0,T]}\left|\int_0^t \int_{-\infty}^\infty g_f(x) w((x\vee (-A))-m(s))(\mu_n(s,dx)-\mu(s,dx))ds\right| \to 0, \; \mbox{ a.s. }
\end{equation}
Here we have used the fact that $x \mapsto g_f(x) w((x\vee (-A))-m(s))$ is a Lipschitz function for every $s \in [0,T]$ and the Lipschitz constant is uniformly bounded for $s\in [0,T]$.

Combining the above with \eqref{eq:804a}, \eqref{eq:804b}, and \eqref{eq:804c}, we have
\begin{align*}
    &\limsup_{n\to \infty} \mathbb{P} \Big[ \sup_{t\in [0,T]} \Big| \int_0^t \int_{-\infty}^{\infty} g_f(x)w(x-m_n(s))\mu_n(s,dx)ds \\
    &\quad\quad\quad - \int_0^t\int_{-\infty}^{\infty} g_f(x)w(x-m(s))\mu(s,dx)ds\Big| > 4\mn\epsilon\Big]
    \leq 3\eta.
\end{align*}
Since $\epsilon, \eta >0$ are arbitrary, combining the above with \eqref{eq:737}, and recalling the definition of $A_{t,f}$, we get, as $n\to \infty$,
\begin{equation}\label{ch7}
    \sup_{t\in [0,T]} |A_{t,f}(\mu_n)-A_{t,f}(\mu)| \to 0, \quad \text{in probability, for every } f\in \mbox{Lip}_1.
\end{equation}
Now we argue that $A_{t,f}(\mu)=0$ for every $t \in [0,T]$ and $f \in \mbox{Lip}_1$.
From the Markov property of $\bfX^n$ it follows that $A_{t,f}(\mu_n) \doteq M_n^f(t)$ is a $\{\clf^n_t\}$-local martingale. Also, the  quadratic variation of this martingale is given by 
\begin{equation*}
    \langle M_n^f \rangle (t)=\frac{1}{n} \int_0^t \int_{-\infty}^{\infty} g_f^2(x) w(x-m_n(s)) \mu_n(s,dx)\; ds, \; t\ge 0.
\end{equation*}
Fix $\epsilon, \eta>0$ and choose $R>0$ such that \eqref{eq:450} holds.
Define $\tau_R^n \doteq \inf \{ t\geq 0: m_n(t)> R\}$.
Then from Doob's maximal inequality, and the monotonicity of $t\mapsto \mu_n(t)$, it follows that
\begin{align*}
    \mathbb{P}\left[ \sup_{t\in [0,T]} |M_n^f(t)|>\epsilon | \mathcal{F}_0^n \right] &\leq \mathbb{P} \left[ \sup_{t\in [0,T]} |M_n^f(t\wedge \tau_R^n)| \geq \epsilon|\mathcal{F}_0^n\right] + \mathbb{P}[m_n(T)> R| \mathcal{F}_0^n] \\
    &\leq \frac{4T\mn^2}{n\epsilon^2} \int_{-\infty}^{\infty} w(x-R) \mu_n(0,dx) + \mathbb{P}[m_n(T)> R|\mathcal{F}_0^n].
\end{align*}
Thus, 
\begin{align*}
    \mathbb{P}\left[ \sup_{t\in [0,T]} |M_n^f(t)|>\epsilon\right]
    &\le  \eta + \mathbb{P}[m_n(T)> R] + \mathbb{P}\left[\frac{4T\mn^2}{n\epsilon^2} \int_{-\infty}^{\infty} w(x-R) \mu_n(0,dx) \ge \eta\right].
\end{align*}
Taking limit as $n \to \infty$ we now see that
\begin{align*}
   &\limsup_{n\to \infty} \mathbb{P}\left[ \sup_{t\in [0,T]} |M_n^f(t)|>\epsilon\right]\\
    &\le \eta + \limsup_{n\to \infty}\mathbb{P}[m_n(T)> R] + \limsup_{n\to \infty}\mathbb{P}\left[\frac{4T\mn^2}{n\epsilon^2} \int_{-\infty}^{-A} w(x-R) \mu_n(0,dx) \ge \eta\right] \\
    &\le \frac{3\eta}{2} + \limsup_{n\to \infty}\mathbb{P}\left[\frac{4T\mn^2}{n\epsilon^2} \int_{-\infty}^{-A} w(x-R) \mu_n(0,dx) \ge \eta\right].
\end{align*}
Since $\epsilon, \eta>0$ are arbitrary, sending $A\to \infty$ in the above display  and recalling Assumption \ref{ass.char}, we now see that
$$
\sup_{t\in [0,T]} |A_{t,f}(\mu_n)|= \sup_{t\in [0,T]} |M_n^f(t)| \to 0 \mbox{ in probability, as } n \to \infty.
$$
The second property in \eqref{eq:mfeq1} now follows on combining this fact with \eqref{ch7}, which completes the proof.
\end{proof}

\subsection{Proof of Theorem \ref{mv}.}
\label{sec:mv}
To address the uniqueness of solutions to the McKean-Vlasov equation \eqref{eq:mfeq1}, we will find it convenient to work with the associated nonlinear Markov processes. Thus, we first prove Theorem \ref{mv}.
The proof is inspired by  \cite[Lemma 10]{Oelschlager1984AMA}.

Recall $\cll_{\mu}$ defined after \eqref{eq:338}. Consider the equation
\begin{equation}\label{un1}
\langle f, \upsilon(t)\rangle = \langle f, \upsilon(0)\rangle + \int_0^t \int_{-\infty}^{\infty}\cll_{\mu} f(s,x) \upsilon(s, dx)ds, \mbox{ for all } t\ge 0 \mbox{ and } f \in \text{Lip}_1,
\end{equation}
for $\upsilon \in \clc([0,\infty) : \clp_1(\RR))$ satisfying
\begin{equation}\label{un2}
\int_0^t\int_{-\infty}^{\infty}w(x-m(s))\upsilon(s,dx)ds< \infty.
\end{equation}
By assumption, $\upsilon =\mu$  solves \eqref{un1} and \eqref{un2}. Write $\upsilon^Y(t) \doteq \cll(Y(t))$. Since $\upsilon^Y(0) = \mu(0)$, we have using the monotonicity of $Y(\cdot)$, $m$ and $w$, for any $s \in [0,t]$,
\begin{multline}
\int_{-\infty}^{\infty}w(x-m(s)) \upsilon^Y(s, dx) \le \int_{-\infty}^{\infty}w(x-m(t)) \mu(0, dx)\\
\le \int_{-\infty}^{\infty}w(x-m(0)) \mu(0, dx) + |m(t)-m(0)| \int_{-\infty}^{\infty}L(x-m(t))\mu(0,dx)  < \infty,
\end{multline}
where the last inequality is from our assumption on $\mu(0, \cdot)$. Hence, for every $f \in \text{Lip}_1$, we see that the quantity in \eqref{eq:338} is in fact a martingale and thus $\upsilon^Y$ also solves \eqref{un1} and \eqref{un2}.
Consequently, to prove the lemma, it suffices to show that there is a unique solution in $\clc([0,\infty) : \clp_1(\RR))$ to \eqref{un1} and \eqref{un2}.

Fix $T >0$. Let $\clc_{\mu}$ be the class of continuous functions $f: [0,T]\times \RR \rightarrow \RR$ such that (i) For any $t \in [0,T]$, $f(t, \cdot)$ is Lipschitz on $[-M, \infty)$ for any $M>0$, (ii) the partial derivative of $f$ with respect to the first variable exists and, for any $t \in [0,T)$, $\frac{\partial f}{\partial t}(t, \cdot)$ is Lipschitz on $[-M, \infty)$ for any $M>0$, and (iii) the following holds:
\begin{equation}\label{fclass}
\sup_{(t,x) \in [0,T) \times \RR} \; \left[|f(t,x)| + \frac{1}{w(x-m(t))}\left|\frac{\partial f}{\partial t}(t,x)\right|\right] < \infty.
\end{equation}
Fix $M>0$. Let $\psi_M$ be a bounded Lipschitz function such that $\psi_M(x) = 1$ for $x \in [-M, \infty)$ and $\psi_M(x)=0$ for $x \in [-\infty, -M-1]$. Define the measure $\upsilon_M$ by $\upsilon_M(t,A) \doteq \int_A\psi_M(x)\upsilon(t,dx), \, A \in \clb(\RR),\, t \ge 0$. For $f \in \clc_{\mu}$ and $t \in [0,T]$, we have for any $l \in \NN$, setting $\delta = \delta(l) \doteq t/l$,
\begin{align}
    \langle f(t, \cdot) , \upsilon_M(t) \rangle - \langle f(0, \cdot), \upsilon_M(0) \rangle &= \sum_{k=0}^{l-1} [ \langle f((k+1)\delta, \cdot) , \upsilon_M((k+1)\delta)) \rangle - \langle f(k\delta, \cdot) , \upsilon_M(k\delta)) \rangle ]\\
    &= \sum_{k=0}^{l-1} [ \langle f((k+1)\delta, \cdot) , \upsilon_M((k+1)\delta) \rangle - \langle f((k+1)\delta, \cdot) , \upsilon_M(k\delta) \rangle ] \\ 
    &\quad + \sum_{k=0}^{l-1} [ \langle f((k+1)\delta, \cdot) , \upsilon_M(k\delta) \rangle - \langle f(k\delta, \cdot) , \upsilon_M(k\delta) \rangle ]. \label{disc}
\end{align}
We now handle the first term in \eqref{disc}. Write $f_M(t,x) \doteq f(t,x) \psi_M(x),\, (t,x) \in [0,T] \times \RR$. Recalling $f \in \clc_{\mu}$, $f_M((k+1)\delta, \cdot)$ is Lipschitz and one deduces from \eqref{un1} that
\begin{align*}
    &\langle f((k+1)\delta, \cdot) , \upsilon_M((k+1)\delta) \rangle - \langle f((k+1)\delta, \cdot) , \upsilon_M(k\delta) \rangle \\ 
    &= \langle f_M((k+1)\delta, \cdot) , \upsilon((k+1)\delta) \rangle - \langle f_M((k+1)\delta, \cdot) , \upsilon(k\delta) \rangle \\ 
    &= \int_{k\delta}^{(k+1)\delta}\int_{-\infty}^{\infty} \mathbb{E} [f_M((k+1)\delta,x+Z) - f_M((k+1)\delta,x)] w(x-m(s))\upsilon(s,dx) ds \\
    &= \int_{k\delta}^{(k+1)\delta}\int_{-\infty}^{\infty} \mathbb{E}_Z [f_M(s,x+Z) - f_M(s,x)] w(x-m(s))\upsilon(s,dx) ds\\  
    & + \int_{k\delta}^{(k+1)\delta}\int_{-\infty}^{\infty} r_M(k,s,x,\delta) w(x-m(s))\upsilon(s,dx) ds,
\end{align*}
where $r_M(k,s,x,\delta) \doteq \mathbb{E}_Z [f_M((k+1)\delta,x+Z) - f_M((k+1)\delta,x) - f_M(s,x+Z) + f_M(s,x) ]$.
Using  \eqref{fclass}, 
and $\sup_{x \in [-M-1, \infty), t \in [0,T]} w(x-m(t)) <\infty$,
it follows that
\begin{equation}
\lim_{l \rightarrow \infty}\sum_{k=0}^{l-1}\int_{k\delta}^{(k+1)\delta}\int_{-\infty}^{\infty} r_M(k,s,x,\delta) w(x-m(s))\upsilon(s,dx) ds = 0,
\end{equation}
and hence,
\begin{multline}\label{ii1}
    \lim_{l \rightarrow \infty}\sum_{k=0}^{l-1} \left[ \langle f((k+1)\delta, \cdot) , \upsilon_M((k+1)\delta) \rangle - \langle f((k+1)\delta, \cdot) , \upsilon_M(k\delta) \rangle\right]\\
    = \int_0^t\int_{-\infty}^{\infty}\mathbb{E} [f_M(s,x+Z) - f_M(s,x)] w(x-m(s))\upsilon(s,dx) ds.
\end{multline}
Now, we handle the second term in \eqref{disc}. Observe that
\begin{align*}
    &\langle f((k+1)\delta, \cdot) , \upsilon_M(k\delta) \rangle - \langle f(k\delta, \cdot) , \upsilon_M(k\delta) \rangle 
    = \left\langle \int_{k\delta}^{(k+1)\delta} \frac{\partial f}{\partial s}(s,\cdot) ds , \upsilon_M(k\delta) \right\rangle\\
    &= \int_{k\delta}^{(k+1)\delta}\left\langle \frac{\partial f}{\partial s}(s,\cdot) , \upsilon_M(s) \right\rangle ds 
    + \int_{k\delta}^{(k+1)\delta} \left\langle  \frac{\partial f}{\partial s}(s,\cdot), \upsilon_M(k\delta) - \upsilon_M(s) \right\rangle ds \\
    &= \int_{k\delta}^{(k+1)\delta}\left\langle \frac{\partial f}{\partial s}(s,\cdot) , \upsilon_M(s) \right\rangle ds 
    - \int_{k\delta}^{(k+1)\delta} \left\langle  \frac{\partial f}{\partial s}(s,\cdot)\psi_M(\cdot), \upsilon(s) - \upsilon(k\delta) \right\rangle ds.
\end{align*}
As $f \in \clc_{\mu}$, $\frac{\partial f}{\partial s}(s,\cdot)\psi_M(\cdot)$ is bounded and Lipschitz. Thus, writing $g_M(s,\cdot) \doteq \frac{\partial f}{\partial s}(s,\cdot)\psi_M(\cdot)$,
\begin{align*}
&\lim_{l \rightarrow \infty}\sum_{k=0}^{l-1}\int_{k\delta}^{(k+1)\delta} \left\langle  \frac{\partial f}{\partial s}(s,\cdot)\psi_M(\cdot), \upsilon(s) - \upsilon(k\delta) \right\rangle ds\\
&= \lim_{l \rightarrow \infty}\sum_{k=0}^{l-1}\int_{k\delta}^{(k+1)\delta}\int_{k\delta}^s \int_{-\infty}^{\infty} \mathbb{E}_Z\left[g_M(s,x + Z) - g_M(s,x)\right]w(x-m(u))\upsilon(u,dx) du ds\\
&= \lim_{l \rightarrow \infty}\sum_{k=0}^{l-1}\int_{k\delta}^{(k+1)\delta} \int_{-\infty}^{\infty} \int_{u}^{(k+1)\delta} \mathbb{E}_Z\left[g_M(s,x + Z) - g_M(s,x)\right]ds \, w(x-m(u))\upsilon(u,dx) du = 0,
\end{align*}
where, in the last step, we used the fact that $\sup_{(t,x) \in [0,T) \times \RR}|g_M(t,x)| < \infty$ along with \eqref{un2}.
Hence,
\begin{equation}\label{ii2}
\lim_{l \rightarrow \infty}\sum_{k=0}^{l-1} [\langle f((k+1)\delta, \cdot) , \upsilon_M(k\delta) \rangle - \langle f(k\delta, \cdot) , \upsilon_M(k\delta) \rangle] = \int_0^t \left\langle \frac{\partial f}{\partial s}(s,\cdot) , \upsilon_M(s) \right\rangle ds.
\end{equation}
Using \eqref{ii1} and \eqref{ii2} in \eqref{disc}, we get
\begin{multline}\label{nu.mu}
\langle f(t, \cdot) , \upsilon_M(t) \rangle - \langle f(0, \cdot), \upsilon_M(0) \rangle\\
= \int_0^t\int_{-\infty}^{\infty} \left(\mathbb{E}_Z [f_M(s,x+Z) - f_M(s,x)] w(x-m(s)) + \frac{\partial f}{\partial s}(s,x)\psi_M(x)\right)\upsilon(s,dx) ds.
\end{multline}
Now, from \eqref{fclass}, there exists a finite constant $C>0$ such that for any $(s,x) \in [0,T) \times \RR$, $$\left|\mathbb{E} [f_M(s,x+Z) - f_M(s,x)] w(x-m(s)) + \frac{\partial f}{\partial s}(s,x)\psi_M(x)\right| \le C w(x - m(s))$$ for all $M>0$. Thus, using \eqref{un2} and the dominated convergence theorem, taking $M \rightarrow \infty$ in \eqref{nu.mu} gives
\begin{multline}\label{steq}
\langle f(t, \cdot) , \upsilon(t) \rangle - \langle f(0, \cdot), \upsilon(0) \rangle\\
= \int_0^t\int_{-\infty}^{\infty} \left(\mathbb{E}_Z [f(s,x+Z) - f(s,x)] w(x-m(s)) + \frac{\partial f}{\partial s}(s,x)\right)\upsilon(s,dx) ds
\end{multline}
for any $t \in [0,T]$ and $f \in \clc_{\mu}$.

Let $h: \RR \rightarrow \RR$ be a bounded Lipschitz function. For $x \in \RR$, $t \in [0,T]$, consider a $\mathcal{D}([t,T]:\RR)$ valued process $\{Y^{t,x}(s) : s \in [t,T]\}$ constructed as $Y$, but with time started from $t$ and $Y^{t,x}(t) = x$. Define the function
\begin{equation}\label{funrep}
f^{T}(t,x) \doteq \EE\left[h\left(Y^{t,x}(T)\right)\right], \ (t,x) \in [0,T] \times \RR.
\end{equation}
We now verify that $f^T \in \clc_{\mu}$. Fix $M>0$ and $t \in [0,T]$. For $x_1, x_2 \in [-M, \infty)$, consider the coupling of $Y^{t,x_1}$ and $Y^{t,x_2}$ described by the following (time-inhomogeneous) generator: for $s \in [t,T]$ and bounded measurable $\phi: [-M, \infty)^2 \rightarrow \RR$,
\begin{align*}
\cll^{x_1,x_2} \phi(s, x,x') & = \EE_Z\left[\phi(x + Z, x' + Z) - \phi(x,x')\right]w(x-m(s)) \wedge w(x'-m(s))\\
& \quad + \EE_Z\left[\phi(x + Z, x') - \phi(x,x')\right](w(x-m(s)) - w(x'-m(s)))^+\\
&\quad + \EE_Z\left[\phi(x, x'+Z) - \phi(x,x')\right](w(x-m(s)) - w(x'-m(s)))^-.
\end{align*}
From the above generator, we have that for the coupled process $(Y^{t,x_1}(\cdot), Y^{t,x_2}(\cdot))$,
\begin{align*}
\EE |Y^{t,x_1}(T) - Y^{t,x_2}(T)| 
& \le |x_1 - x_2| + \mn\int_t^T\EE \left|w(Y^{t,x_1}(s) - m(s)) - w(Y^{t,x_2}(s)-m(s))\right|ds\\
& \le |x_1 - x_2| + L(-M-m(T))\mn\int_t^T\EE|Y^{t,x_1}(s) - Y^{t,x_2}(s)|ds
\end{align*}
where we have used Assumption \ref{ass.w} in the last step, namely, $w(\cdot)$ is  Lipschitz on $[-M-m(T), \infty)$ with Lipschitz constant $L(-M-m(T))$. Thus, by Gr\"onwall's lemma, denoting the Lipschitz constant of $h$ by $C_h$,
\begin{equation}\label{eq:516n}
|f^{T}(t,x_1) - f^{T}(t,x_2)| \le C_h\EE|Y^{t,x_1}(T) - Y^{t,x_2}(T)| \le C_he^{L(-M-m(T))\mn(T-t)}|x_1 - x_2|
\end{equation}
showing that property (i) of a function in $\clc_{\mu}$ holds for $f^T$.

It is straightforward to check using Assumption \ref{ass.w} and the continuity of $m(\cdot)$ that $\frac{\partial f^T(t,x)}{\partial t}$ exists for each $(t,x) \in [0,T) \times \RR$ and $f^T$ satisfies the Kolmogorov equation
\begin{equation}\label{parrep}
\begin{aligned}
\mathbb{E}_Z [f^T(t,x+Z) - f^T(t,x)] w(x-m(t)) + \frac{\partial f^T}{\partial t}(t,x) &= 0, \ (t,x) \in [0,T) \times \RR.\\
f^T(T,x) &= h(x), \, x \in \RR.
\end{aligned}
\end{equation}
{\di For completeness we give a proof in the Appendix (See Section \ref{sec:appa}).}

The Lipschitz property of $\frac{\partial f^T(t,\cdot)}{\partial t}$ on $[-M, \infty)$ for fixed $t \in [0,T), M>0$ follows from \eqref{parrep} and the fact that both $f^T(t, \cdot)$ and $w(\cdot - m(t))$ are bounded and Lipschitz on $[-M, \infty)$. This proves property (ii) of a function in $\clc_{\mu}$ holds for $f^T$.

Finally property (iii) in $\clc_{\mu}$ for $f^T$   follows from the boundedness of $h$, and the explicit representations of $f^T$ and $\frac{\partial f^T}{\partial t}$ respectively displayed in \eqref{funrep} and \eqref{parrep}. This completes the proof that $f^T \in \clc_{\mu}$.

Thus, plugging in $f^T$ for $f$ in \eqref{steq} and using \eqref{parrep}, we conclude that
$$
\langle h(\cdot), \upsilon(T)\rangle = \langle f^T(T,\cdot), \upsilon(T)\rangle = \langle f^T(0,\cdot), \upsilon(0)\rangle = \langle h(\cdot), \upsilon^Y(T)\rangle
$$
for any bounded Lipschitz $h$. As $T>0$ is arbitrary, this establishes the uniqueness of the solution to \eqref{un1} and \eqref{un2} and thereby completes the proof of the lemma.
\hfill \qed

\subsection{Uniqueness of McKean-Vlasov Equation Solutions} 
Theorem \ref{thm.ch}  in Section \ref{sec:char} established existence of solutions to the McKean-Vlasov equation \eqref{eq:mfeq1}, as any distributional limit point of $\{\mu_n\}$ (which always exists in view of this theorem) is a solution of \eqref{eq:mfeq1}. 
Note also that if Assumption \ref{ass.init} holds then this solution is in the class $\clm$
introduced in Section \ref{sec:modres}.
In the current section, using Theorem \ref{mv} (which was proved in Section \ref{sec:mv}) we show that, under conditions, for a suitable given initial condition, \eqref{eq:mfeq1} has a unique solution in  $\clm$. This, in particular, will show that the whole sequence $\{\mu_n\}_{n \in \NN}$ converges under the assumptions of Theorem \ref{thm.mainFL} (c) and thereby complete the proof of Theorem \ref{thm.mainFL}.
\begin{lemma}\label{lem.uniq}
Suppose Assumption \ref{ass.w} holds.
Let $\mu \in \clm$ and let 
     $\{Y(t), t \ge 0\}$ be the associated nonlinear Markov  process. Then, with $\beta$ as in \eqref{cM2}, for every $T \in (0,\infty)$,
    $$\sup_{t \in [0,T]} \mathbb{E}|Y(t)|^{1+\beta} < \infty.$$
\end{lemma}
\begin{proof}
    Define $\tau \doteq \inf \{ t\geq 0: Y(t)\geq 0\}.$ Fix $T \in (0,\infty)$ and
     let $V$ be a $\mathcal{J}(\theta)$-jump process which is independent of the particle system.
     
    
    Using the monotonicity of $w$, we see that, for every $t\in [0,T]$, $$Y(t)\leq_d |Y(0)| + Y(\tau) + V(t w(-m(T))).$$ 
    Hence, using Lemma \ref{R} and the fact that since $\mu \in \clm$,  $\mu(0)$ satisfies conditions \eqref{cM1} and \eqref{cM2}, for some $C_{\beta} \in (0,\infty)$,
    \begin{align*}
    \sup_{t\in [0,T]}\mathbb{E} |Y(t)|^{1+\beta}
    \leq C_\beta \left[ \mathbb{E}|Y(0)|^{1+\beta} + \sup_{l\geq 0} ( \mathbb{E}\cO_l^2 )^\frac{1+\beta}{2} +  \left(\vt (w(-m(T))^2 T^2 +w(-m(T))T) \right)^{\frac{1+\beta}{2}}\right] < \infty.
    \end{align*}
The result follows.
\end{proof}
The following is the key uniqueness result for solutions of the McKean-Vlasov equation \eqref{eq:mfeq1} which gives part (c) of Theorem \ref{thm.mainFL}.
\begin{theorem}\label{thm.un}
    Suppose  Assumptions \ref{Z}, \ref{ass.w}, \ref{ass.t1},  \ref{ass.char} hold. 
    Suppose $\mu_1, \mu_2 \in \clm$ solve the McKean-Vlasov equation \eqref{eq:mfeq1} and satisfy $\mu_1(0) = \mu_2(0) = \gamma$. Then $\mu_1(t) = \mu_2(t)$ for all $t\ge 0$.
\end{theorem}
\begin{proof}
 Consider nonlinear Markov processes $Y_1(\cdot)$ and $Y_2(\cdot)$, associated with $\mu_1$ and $\mu_2$ respectively (existence of which is guaranteed by Theorem \ref{mv}). In particular, $\mathcal{L}(Y_i(t))=\mu_i(t)$, for $i=1,2$ and $t\ge 0$, and the jump rate of $Y_i(\cdot)$ at time $t$ is given by $w(Y_i(t)-m_i(t))$, where $m_i(t)=\mathbb{E}Y_i(t)$ and the jump lengths are distributed as $\theta$. Note that, the law of the initial conditions are given by $\mathcal{L}(Y_1(0))=\mathcal{L}(Y_2(0))=\gamma$.

Thus, to prove the result, it is enough to construct processes $\tilde Y_1$ and $\tilde Y_2$ on a common probability space, such that, $\cll(\tilde Y_i(t)) = \cll(Y_i(t))$,  for $i=1,2$ and $t\ge 0$, and $$\mathbb{P}[\tilde Y_1(t)= \tilde Y_2(t)]=1.$$

For this we proceed as follows. Let 
$(\tilde Y_1(t), \tilde Y_2(t))_{t \ge 0}$ be given on a common probability space with 
$\tilde Y_1(0) = \tilde Y_2(0) = \xi$, where $\xi$ is distributed as $\gamma$, and (time inhomogeneous) generator defined as
\begin{align*}
\tilde \cll \phi(s, x,x') & = \EE_Z\left[\phi(x + Z, x' + Z) - \phi(x,x')\right]w(x-m_1(s)) \wedge w(x'-m_2(s))\\
& \quad + \EE_Z\left[\phi(x + Z, x') - \phi(x,x')\right](w(x-m_1(s)) - w(x'-m_2(s)))^+\\
&\quad + \EE_Z\left[\phi(x, x'+Z) - \phi(x,x')\right](w(x-m_1(s)) - w(x'-m_2(s)))^-
\end{align*}
for $\phi \in \MM_b(\RR^2)$ and $(x,x') \in \RR^2$.

It is easy to verify that, as probability measures on $\cld([0,\infty): \RR)$, $\cll(\tilde Y_i) = \cll(Y_i)$ for $i=1,2$. Henceforward, for notational simplicity, we write $\tilde Y_i$ as $Y_i$.
Define the $\sigma$-field, $\mathcal{F}_t\doteq \sigma\{Y_1(s),Y_2(s):\; s\leq t\}$.
From the above construction it follows that, the rate of jump of the process, $| Y_1(\cdot)-Y_2(\cdot)|$  at time $t$ is given by $$|w( Y_1(t-)-m_1(t))-w(Y_2(t-)-m_2(t))|.$$ 
Also, for fixed $T \in (0,\infty)$, letting $m_T\doteq m_1(T)\vee m_2(T)$,
we have that $Y_i(s)- m_i(s) \in (\xi - m_T, \infty)$ for $s \in [0,T]$.
{\di Since $w$ is bounded on this interval, we see, by an application of Dynkin's formula,
\begin{equation}\label{eq:1039n}
\begin{aligned}
    \mathbb{E}_0|Y_1(t)-Y_2(t)| &\le \mn\mathbb{E}_0\left[ \int_0^t |w(Y_1(s)-m_1(s))-w(Y_2(s)-m_2(s))|ds\right]\\
        &\leq 2\mn\int_0^t\mathbb{E}_0(|Y_1(s)-Y_2(s)|) L(\xi-m_T))ds + 2\mn\int_0^t|m_1(s)-m_2(s)| L(\xi-m_T)ds,
\end{aligned}
\end{equation}
where $\EE_0$ denotes the expectation conditioned on $\clf_0$ and
$L(\cdot)$ is as introduced in Assumption \ref{ass.w}{\di .}
The proof of \eqref{eq:1039n} is given in the Appendix (Section \ref{sec:appb}).}
Taking expectations, and writing, for $K\ge 0$, $1 = \one\{\xi > -K\} + \one\{\xi \le -K\}$, we have
\begin{align*}
\mathbb{E}|Y_1(t)-Y_2(t)|
        &\leq 2\mn\int_0^t L(-K-m_T)\mathbb{E}|Y_1(s)-Y_2(s)|ds\\
        &\quad + 2\mn\int_0^t \mathbb{E}[(|Y_1(s)|+|Y_2(s)|)L(\xi-m_T)\mathbf{1}[\xi\leq -K]]ds \\ 
        &\quad + 2\mn C_T\int_0^t |m_1(s)-m_2( s)|ds, 
\end{align*}
where $C_T\doteq \mathbb{E}[L(\xi-m_T)]< \infty$, since $\gamma$ satisfies \eqref{cM1} (with $\nu$ replaced by $\gamma$). Hence, using H\"older's inequality, we obtain, with $C_K \doteq 2\mn (C_T+L(-K-m_T))$, and $\beta$ as in \eqref{cM2},
\begin{multline*}
        \mathbb{E}|Y_1(t)-Y_2(t)| \leq C_K\int_0^t\mathbb{E}|Y_1(s)-Y_2(s)|ds \\ 
        + 2\mn \int_0^t \left[\left(\mathbb{E}|Y_1(s)|^{1+\beta}\right)^\frac{1}{1+\beta}+\left(\mathbb{E}|Y_2(s)|^{1+\beta}\right)^\frac{1}{1+\beta} \right] \left[\mathbb{E}\left(L^{1+\frac{1}{\beta}}(\xi-m_T)\mathbf{1}\{\xi\leq -K\} \right) \right]^\frac{\beta}{1+\beta}\; ds.
\end{multline*}
Next, using Lemma \ref{lem.uniq}, we have that  
$$\sup_{t \in [0,T]} \left[\left(\mathbb{E}|Y_1(s)|^{1+\beta}\right)^\frac{1}{1+\beta}+\left(\mathbb{E}|Y_2(s)|^{1+\beta}\right)^\frac{1}{1+\beta} \right] \doteq C_T'< \infty.$$ 
Hence, it follows that, with $\delta$ as in \eqref{cM1}, with 
$(\nu, a)$ there replaced by $(\gamma, m_T)$, and  $c_* \doteq 2\mn\sup_{\kappa \in \RR_+} \kappa e^{-\delta \frac{\kappa \beta}{2(1+\beta)}}$,
\begin{equation*}
    \mathbb{E}|Y_1(t)-Y_2(t)|\leq C_K\int_0^t \mathbb{E}|Y_1(s)-Y_2(s)|ds 
    + \; C_T' c_*\left[\mathbb{E}(e^{\delta L(\xi-m_T)/2}
    \mathbf{1}[\xi\leq -K])\right]^{\frac{\beta}{1+\beta}}.
\end{equation*}
Using Gr\"onwall's lemma, for any $t \in [0,T]$,
\begin{align*}
\mathbb{E}|Y_1(t)-Y_2(t)| &\leq
c_*C_T' \left[\mathbb{E}(e^{\delta L(\xi-m_T)/2}
    \mathbf{1}[\xi\leq -K])\right]^{\frac{\beta}{1+\beta}} e^{C_K t}.
\end{align*}
Thus, recalling that $C_K \doteq 2\mn (C_T+L(-K-m_T))$ and using the monotonicity property of $L$, for any $t \in (0, \frac{\delta \beta}{4\mn (1+\beta)} \wedge T]$,
\begin{align*}
\mathbb{E}|Y_1(t)-Y_2(t)| &\leq
c_*C_T' e^{2\mn C_T} \left[\mathbb{E}(e^{\delta L(\xi-m_T)}
    \mathbf{1}[\xi\leq -K])\right]^{\frac{\beta}{1+\beta}}.
\end{align*}
By our choice of $\delta$, $\mathbb{E}(e^{\delta L(\xi-m_T)})<\infty$.
Hence, taking the limit as $K\to \infty$ in the above display, we have that, with 
$t_0 = \frac{\delta \beta}{4\mn (1+\beta)} \wedge T$, for every 
$t \in [0, t_0]$,
$\mathbb{E}|Y_1(t)-Y_2(t)|=0$.

Finally,  from the monotonicity properties of $L(\cdot)$ and  $\mu_i(\cdot)$,  $i=1,\,2$, for every $ a \in \mathbb{R}$ and  $\delta>0$ 
$$\int e^{\delta L(x-a)}\mu_i(t,dx)\leq \int e^{\delta L(x-a)}\mu_i(0,dx),\;
\mbox{for every  } t\geq 0.$$ 
Hence, we can bootstrap the above argument to conclude that for every $t\in [0,T]$, $\mathbb{E}|Y_1(t)-Y_2(t)|=0$. Since $T>0$ is arbitrary, the result follows.
\end{proof}
%
%
\subsection{\textbf{Proof of Theorem \ref{thm.mainFL}}}
\label{sec:pf2.8}
Part (a) of Theorem \ref{thm.mainFL} follows from Theorem \ref{thm.tght} on recalling that Assumption \ref{ass.t2} is implied by Assumption \ref{ass.char}. Theorem \ref{thm.ch} proves part (b) of the theorem.

For part (c), due to Assumption \ref{ass.init}, any such weak limit point $\mu$
satisfies $\mu(0)=\gamma$, where $\gamma$ is the non-random probability measure introduced in Assumption \ref{ass.init}. It then follows that, a.s. $\mu$ takes values in $\clm$, 
it solves the McKean-Vlasov equation, and $\mu(0)=\gamma$. Finally, the uniqueness result shown in Theorem \ref{thm.un} now completes the proof of part (c) of the theorem, namely  that there is exactly one element  $\mu^*$ of $\clm$ that solves the McKean-Vlasov equation and satisfies $\mu^*(0) = \gamma$, and furthermore, since all weak limit points are the same as $\mu^*$,  $\mu_n \to \mu^*$ in probability, in $\mathcal{D}([0,\infty):\mathcal{P}_1(\mathbb{R}))$. 

For part (d), observe that if $(X^n_1(0), \dots, X^n_n(0))$ are iid with distribution $\gamma$ satisfying \eqref{cM1} and \eqref{cM2}, then $\mu_n(0) \to \gamma$ in probability which says that Assumption \ref{ass.init} holds. Furthermore Assumptions \ref{ass.t1} and \ref{ass.char} clearly hold in this case. The first assertion in (d) now follows from part (c). The second assertion in part (d) follows from the first by a standard argument, cf. \cite[Lemma 3.19]{chaintron2022propagation}.
\hfill \qed


\section{Stationary Distributions of the Particle System}\label{s.LT}
In this section we prove Theorems \ref{stat.ex} and \ref{stat.pr}. The first theorem gives, under conditions, the existence and uniqueness of the stationary distributions of the  centered $n$-particle system described by the process $\bfY^n$ introduced in Section \ref{sec:modres}. It also gives  uniform in $n$  bounds on certain moments computed under the stationary distributions.  The second theorem gives  a key uniform integrability estimate, under the stationary distributions. In addition to giving information about the stationary distribution (for which, typically, there is no closed form expression), the estimates from the two theorems  are crucial for studying convergence properties of these stationary distributions as $n \rightarrow \infty$.

We will need the following observations about the particle system in the proof of Theorem \ref{stat.ex}. Recall that a Markov process $\{\bfZ(t): t \ge 0\}$ on a Polish space $S$ is said to satisfy the weak Feller property if, for any $f \in C_b(S)$ and any $t \ge 0$, the function $\bfz \mapsto \EE_{\bfz}\left(f(\bfZ(t))\right)$ (with $\EE_{\bfz}$ denoting the expectation under which  $\bfZ(0) = \bfz$ a.s.) is in $C_b(S)$. Recall that by the centered $n$-particle system we mean the Markov process with state space $\cls_0$ and with generator given by \eqref{eq:geny}.

\begin{lemma}\label{lem:wF}
Suppose that Assumptions \ref{Z} and \ref{ass.w} hold. Fix $n \in \NN$. For any $\bfy, \tilde \bfy \in \cls_0$, there exists a coupling $(\bfY^n, \tilde \bfY^n)$ of the centered $n$-particle systems with $\bfY^n(0) = \bfy$ and $\tilde \bfY^n(0) = \tilde \bfy$, with joint law denoted by $\PP_{\bfy, \tilde \bfy}$, such that for any $t \ge 0$, $\delta>0$ and $\bfy \in \cls_0$,
$$
\lim_{\tilde \bfy \rightarrow \bfy}\PP_{\bfy,\tilde \bfy}\left(\sum_{i=1}^n |Y^n_i(t) - \tilde Y^n_i(t)| \ge \delta\right)=0.
$$
In particular, the centered $n$-particle system has the weak Feller property.
\end{lemma}

\begin{proof}
Consider the continuous time Markov process $(\bfX^n, \tilde \bfX^n)$, with $(\bfX^n(0), \tilde \bfX^n(0)) = (\bf y, \tilde \bfy)$, described by the following generator: for bounded measurable $f: \RR^{2n} \rightarrow \RR$,
\begin{align*}
\cll_n^{(\bfX,\tilde \bfX)}f(\bfx,\tilde \bfx) &= \sum_{i=1}^n\EE_Z\left[f(\bfx + Ze_i, \tilde \bfx + Ze_i) - f(\bfx, \tilde \bfx)\right]w(x_i - \bar \bfx) \wedge w(\tilde x_i - \bar {\tilde \bfx})\\
&\quad + \sum_{i=1}^n\EE_Z\left[f(\bfx + Ze_i , \tilde \bfx) - f(\bfx, \tilde \bfx)\right] (w(x_i - \bar \bfx) - w(\tilde x_i - \bar {\tilde \bfx}))^+\\
&\quad + \sum_{i=1}^n\EE_Z\left[f(\bfx, \tilde \bfx + Ze_i) - f(\bfx, \tilde \bfx)\right] (w(x_i - \bar \bfx) - w(\tilde x_i - \bar {\tilde \bfx}))^-, \ (\bfx, \tilde \bfx) \in \RR^{2n},
\end{align*}
where $\bar \bfx \doteq n^{-1} \sum_{i=1}^n x_i$ and $\bar{\tilde \bfx} \doteq n^{-1} \sum_{i=1}^n \tilde x_i$ and $Z \sim \theta$. Observe that $(\bfX^n, \tilde \bfX^n)$ gives a coupling of the (uncentered) $n$-particle systems with generator given by \eqref{eq:gene} started from $(\bfy, \tilde \bfy)$. In particular, writing $m_n(t) \doteq n^{-1}\sum_{j=1}^n X^n_j(t)$ and $\tilde m_n(t) \doteq n^{-1}\sum_{j=1}^n \tilde X^n_j(t)$, it follows that the $\cls_0^2$ valued Markov process $(\bfY^n, \tilde \bfY^n)$ given by $Y^n_i(t) = X^n_i(t) - m_n(t),\, \tilde Y^n_i(t) = \tilde X^n_i(t) - \tilde m_n(t)$, $t \ge 0$, $1 \le i \le n$, gives a coupling of the centered $n$-particle systems started from $(\bfy, \tilde \bfy)$. Denote the probability measure on the space these processes are defined by  $\PP_{\bfy,\tilde \bfy}$ and corresponding expectation by $\EE_{\bfy,\tilde \bfy}$.

Define $X^n_{\max}(t) \doteq \max\{X^n_i(t) : 1 \le i \le n\}$ and similarly define $\tilde X^n_{\max}(t)$. For $R>0$, define the stopping time $\tau_R \doteq \inf\{t \ge 0: X^n_{\max}(t) \vee \tilde X^n_{\max}(t) \ge R\}$. Then, using \eqref{eq:maxbd} in the proof of Proposition \ref{prop:exis}, we obtain a constant $C>0$, depending on $n,t, \bfy, \tilde \bfy$ such that $\EE_{\bfy,\tilde \bfy}\left(X^n_{\max}(t \wedge \tau_R) \vee \tilde X^n_{\max}(t \wedge \tau_R)\right) \le C$ for all $R>0$. Moreover, using the generator of the coupled Markov process, we conclude that
\begin{align*}
&\EE_{\bfy,\tilde \bfy}\left(\sum_{i=1}^n |Y^n_i(t \wedge \tau_R) - \tilde Y^n_i(t\wedge \tau_R)|\right) \\
&\le 
\sum_{i=1}^n |y_i - \tilde y_i|+
\frac{(2n-2)\mn}{n}\EE_{\bfy,\tilde \bfy}\left(\int_0^{t\wedge \tau_R} \sum_{i=1}^n|w(Y^n_i(s)) - w(\tilde Y^n_i(s))|ds\right), \ t \ge 0.
\end{align*}
Further, $Y^n_i(s) \ge y_i - R$ and $\tilde Y^n_i(s) \ge \tilde y_i - R$ for $s \in [0, t \wedge \tau_R)$. Recalling from Assumption \ref{ass.w} that $w$ is Lipschitz on $[(\min_i y_i) \wedge (\min_i \tilde y_i-R), \infty)$, we deduce from the above that there exists $C'>0$ depending on $n,t, \bfy, \tilde \bfy,R$ such that
$$
\EE_{\bfy,\tilde \bfy}\left(\sum_{i=1}^n |Y^n_i(t \wedge \tau_R) - \tilde Y^n_i(t\wedge \tau_R)|\right) \le \sum_{i=1}^n |y_i - \tilde y_i|+ C'\int_0^t \EE_{\bfy,\tilde \bfy}\left(\sum_{i=1}^n |Y^n_i(s \wedge \tau_R) - \tilde Y^n_i(s\wedge \tau_R)| \right)ds
$$
and thus, by Gr\"onwall's inequality,
$$
\EE_{\bfy,\tilde \bfy}\left(\sum_{i=1}^n |Y^n_i(t \wedge \tau_R) - \tilde Y^n_i(t\wedge \tau_R)|\right) \le e^{C't} \sum_{i=1}^n |y_i - \tilde y_i|.
$$
Using the above observations along with Markov's inequality, we obtain for $R>0,\delta >0$ and $t \ge 0$,
\begin{align*}
&\PP_{\bfy,\tilde \bfy}\left(\sum_{i=1}^n |Y^n_i(t) - \tilde Y^n_i(t)| \ge \delta\right) \le \PP_{\bfy,\tilde \bfy}(\tau_R \le t) + \PP_{\bfy,\tilde \bfy}\left(\sum_{i=1}^n |Y^n_i(t \wedge \tau_R) - \tilde Y^n_i(t\wedge \tau_R)| \ge \delta\right)\\
&\quad = \PP_{\bfy,\tilde \bfy}(X^n_{\max}(t \wedge \tau_R) \vee \tilde X^n_{\max}(t \wedge \tau_R)\ge R) + \PP_{\bfy,\tilde \bfy}\left(\sum_{i=1}^n |Y^n_i(t \wedge \tau_R) - \tilde Y^n_i(t\wedge \tau_R)| \ge \delta\right)\\
&\quad \le \frac{C}{R} + \frac{1}{\delta}e^{C't} \sum_{i=1}^n |y_i - \tilde y_i|.
\end{align*}
Thus, for any fixed $t \ge 0$, $\delta>0$ and $\bfy \in \cls_0$,
$$
\lim_{\tilde \bfy \rightarrow \bfy}\PP_{\bfy,\tilde \bfy}\left(\sum_{i=1}^n |Y^n_i(t) - \tilde Y^n_i(t)| \ge \delta\right) \le \frac{C}{R}.
$$
As $R>0$ is arbitrary, this concludes the proof of the lemma.
\end{proof}

We will need the following elementary but useful general fact.

\begin{lemma}\label{lem:unique}
    Suppose $\bfZ(\cdot)$ is a c\`adl\`ag Markov process taking values in a Polish space $\cls$. Suppose that, for any $\bfx, \bfy \in \cls$, there exists a coupling $(\bfZ, \tilde \bfZ)$ of this Markov process with $(\bfZ(0), \tilde \bfZ(0)) = (\bfx,\bfy)$, with law  on $\cld([0,\infty) \, : \, \cls^2)$ denoted by $\PP_{\bfx, \bfy}$, such that the associated coupling time $\tau \doteq \inf\{t \ge 0: \bfZ(s) = \tilde \bfZ(s) \text{ for all } s \ge t\}$ satisfies $\PP_{\bfx,\bfy}(\tau < \infty)>0$. Moreover, suppose for any $B \in \cld([0,\infty) \, : \, \cls^2)$, the map $(\bfx, \bfy) \mapsto \PP_{\bfx,\bfy}(B)$ is measurable. Then there exists at most one stationary distribution for this process.
\end{lemma}

\begin{proof}
We argue by contradiction. Suppose there exist more than one stationary distribution. Then there exist two stationary distributions $\pi_1$ and $\pi_2$ which are mutually singular. Hence, there exists $A \in \clb(\cls)$ such that $\pi_1(A) = \pi_2(A^c)=0$. We can define the coupling $\PP_{\pi_1, \pi_2}$ of the associated stationary processes as $\PP_{\pi_1, \pi_2}(B) \doteq \int_{\cls^2}\PP_{\bfx,\bfy}(B)\pi_1(d\bfx) \pi_2(d\bfy), \ B \in \cld([0,\infty) \, : \, \cls^2)$, which is well defined by the measurability assumption on the kernel. Then $\PP_{\pi_1, \pi_2}(\tau<\infty) = \int_{\cls^2}\PP_{\bfx,\bfy}(\tau<\infty)\pi_1(d\bfx) \pi_2(d\bfy)>0$ by hypothesis. Take $t>0$ large enough such that $\PP_{\bfx,\bfy}(\tau \le t)>0$. Then,
denoting coordinate maps on $\cld([0,\infty) \, : \, \cls^2)$ once more as $\bfZ, \tilde \bfZ$,
\begin{align*}
\pi_1(A) &= \PP_{\pi_1,\pi_2}(\bfZ(t) \in A) \ge  \PP_{\pi_1,\pi_2}(\bfZ(t) \in A, \tau \le t) = \PP_{\pi_1,\pi_2}(\tilde \bfZ(t) \in A, \tau \le t)\\
&\ge \PP_{\pi_1,\pi_2}(\tau \le t) - \PP_{\pi_1,\pi_2}(\tilde \bfZ(t) \in A^c) >0,
\end{align*}
where, in the last step, we used the fact that $\PP_{\pi_1,\pi_2}(\tilde \bfZ(t) \in A^c) = \pi_2(A^c)=0$. The above contradicts $\pi_1(A)=0$, completing the proof.
\end{proof}

 \subsection{Proof of Theorem \ref{stat.ex}.} 
 \label{subsec6.1}
 The proof is based on constructing a suitable Lyapunov function.
Define $V:\cls_0 \to \RR_+$ as
$$V(\textbf{y}) \doteq \frac{1}{n}\sum_{i=1}^n y_i^2, \; \bfy = (y_1, \ldots, y_n)\in \RR^n.$$
Then, with $q_i(\bfy) \doteq y_i^2$, $i \in [n]$, $\bfy \in \cls_0$, and the generator, $\cll_n^{\bfY}$ defined in \eqref{eq:geny},
\begin{align*}
     \cll_n^{\bfY} q_i(\bfy)&= \left[\mathbb{E}_Z\left(y_i + Z\left(\frac{n-1}{n}\right)\right)^2 - y_i^2\right]w(y_i) + \sum_{j:j\neq i}\left[\mathbb{E}_Z\left(y_i - \frac{Z}{n}\right)^2 - y_i^2\right]w(y_j).
\end{align*}
Thus, for $\bfy \in \cls_0$,
\begin{align*}
    \cll_n^{\bfY}V(\bfy) &= \frac{1}{n}\sum_{i=1}^n\cll_n^{\bfY} q_i(\bfy) \\
    &= \frac{1}{n}\sum_{i=1}^n  \left[\mathbb{E}_Z\left(y_i + Z\left(\frac{n-1}{n}\right)\right)^2 - y_i^2\right]w(y_i) + \frac{1}{n}\sum_{i=1}^n \sum_{j:j\neq i}\left[\mathbb{E}_Z\left(y_i - \frac{Z}{n}\right)^2 - y_i^2\right]w(y_j).
\end{align*}
Rearranging terms and recalling that $\vt = \EE Z^2$ and $\mn = \EE Z$, we have
\begin{align}
   \cll_n^{\bfY}V(\bfy) &= \frac{1}{n}\sum_{i=1}^n  \left[\mathbb{E}_Z\left(y_i + Z\left(\frac{n-1}{n}\right)\right)^2 - y_i^2\right]w(y_i)  + \frac{1}{n}\sum_{i=1}^n \sum_{j=1}^n
   \left[\mathbb{E}_Z\left(y_i - \frac{Z}{n}\right)^2 - y_i^2\right]w(y_j)\\
   &\quad- \frac{1}{n}\sum_{i=1}^n \left[\mathbb{E}_Z\left(y_i - \frac{Z}{n}\right)^2 - y_i^2\right]w(y_i) \\
    &= \frac{1}{n}\sum_{i=1}^n \left[ \left(\frac{n-1}{n}\right)^2\vt + 2\frac{n-1}{n}\mn y_i\right]w(y_i) - \frac{1}{n} \sum_{i=1}^n \left[ \frac{\vt}{n^2} - \frac{2\mn y_i}{n} \right]w(y_i) \\ &\quad + \frac{1}{n}\sum_{i=1}^n \left[ \frac{\vt}{n^2} - \frac{2\mn y_i}{n} \right] \sum_{j=1}^n w(y_j) \\
    &= \frac{(n-1)}{n^2} \vt\sum_{i=1}^n w(y_i) + \frac{2}{n} \mn \sum_{i=1}^n y_i w(y_i)
    \le \frac{2}{n}  \sum_{i=1}^n  \left(\frac{\vt}{2} w(y_i) +  \mn  y_i w(y_i)\right),
\end{align}
where in obtaining the last equality, we have used $\sum_{i=1}^n y_i=0$.
Recall that by assumption in the theorem, $w$ is a non-constant function that is   non-increasing.
Thus we can obtain $B\in (0,\infty)$ and $\delta \in (0,1)$ such that 
$$w(-B)(1-\delta)^2 > w(B)+\delta,\; \mbox{ and } \frac{\vt}{2(1-\delta)\mn B}<\delta.$$
Using the second inequality above in the first inequality below, and the first inequality above in the third inequality below, we have, for  $y \in (-\infty, -B]$,
\begin{multline*}
    \mn yw(y) + \frac{\vt}{2}w(y) = \delta \mn yw(y) + (1-\delta) \mn y w(y) + \frac{\vt}{2}w(y) \\
    = \delta \mn y w(y) +(1-\delta) \mn y w(y) \left(1- \frac{\vt}{2\mn (1-\delta)(-y)}\right) 
    \leq \delta \mn y w(y) + (1-\delta)^2 \mn y w(y) \\
    \leq \delta \mn y w(y) + (1-\delta)^2 \mn w(-B)y 
    < \delta \mn y w(y) + (w(B)+\delta)\mn y.
\end{multline*}
Writing the sum: $\sum_{i=1}^n =  \sum_{i: y_i \le -B} + \sum_{i: y_i \ge B} + \sum_{i: |y_i| < B}$, we now have, from the previous inequality,
\begin{align}\label{st2}
\cll_n^{\bfY}V(\bfy) & \le \frac{2}{n}\sum_{i: y_i \le -B}\Big(\delta \mn y_i w(y_i) + (w(B)+\delta)\mn y_i \Big) + \frac{2}{n}\sum_{i: y_i \ge B}\left(\frac{\vt}{2} w(y_i) +  \mn  y_i w(y_i)\right)\notag\\
&\quad + \frac{2}{n}\sum_{i: |y_i|<B} (\mn y_iw(y_i) + \frac{\vt}{2}w(y_i))\notag\\
&\le -\frac{2\delta\mn}{n}\sum_{i=1}^n y_i^-w(y_i) + \frac{2\delta\mn}{n}\sum_{i:y_i\in (-B,0]}  y_i^-w(y_i)
-\frac{2\mn}{n}(w(B)+\delta)\sum_{i=1}^n y_i^-\notag\\
&\quad + \frac{2\mn}{n}(w(B)+\delta)\sum_{i:y_i\in (-B,0]}  y_i^- + \frac{2\mn}{n}w(B)\sum_{i=1}^n y_i^+ - \frac{2\mn}{n}w(B)\sum_{i:y_i\in [0, B)}  y_i^+ + \frac{\vt}{n} 
\sum_{i: y_i\ge B} w(y_i)\notag\\
&\quad + \frac{2}{n}\sum_{i: |y_i|<B} (\mn y_iw(y_i) + \frac{\vt}{2}w(y_i))\notag\\
&\le C_1 - \delta \mn \left(\frac{1}{n}\sum_{i=1}^n |y_i|\right) -2\delta \mn\left(\frac{1}{n}\sum_{i=1}^n y_i^- w(y_i) \right),
\end{align}
where we have used $\frac{2}{n}\sum_{i=1}^n y_i^- = \frac{1}{n}\sum_{i=1}^n |y_i| \text{ for } \bfy \in \cls_0$ in the last line, and $C_1$ is a finite positive constant defined as
\begin{align*} 
C_1 & \doteq \sup_{n\ge 1, y \in \cls_0} \Big[\frac{2\mn}{n}\sum_{i:y_i\in(-B,0]}  y_i^-w(y_i)
+\frac{2\mn}{n}(w(B)+1)\sum_{i:y_i\in(-B,0]}  y_i^-+ \frac{\vt}{n} 
\sum_{i: y_i\ge B} w(y_i)\\
&\quad  +  \frac{2}{n}\sum_{i: |y_i|<B} (\mn y_iw(y_i) + \frac{\vt}{2}w(y_i)) \Big].
\end{align*}
Defining the compact set $K \doteq \left\{ \bfy\in \mathbb{R}^n: \frac{1}{n}\sum_{i=1}^n|y_i| \leq  \frac{2C_1}{\mn\delta}\right\}$, we have
\begin{equation}\label{eq:904}
\cll_n^{\bfY}V(\bfy) \leq -C_1 +2C_1 \mathbf{1}_K(\bfy), \; \bfy \in \RR^n.\end{equation}
This, along with the weak Feller property of the Markov process $\bfY^n(\cdot)$ deduced in Lemma \ref{lem:wF}, says that there must be at least one  stationary distribution $\pi_n$ for the process (cf.  \cite[Corollary 1.18]{eberle2015markov}). 

To show the moment bound in \eqref{st.thm1}, consider for $K>0$,  $V_K(\textbf{y})\doteq V(\textbf{y}) \wedge K$, $\textbf{y}\in \RR^n$. Also, for $M \in (0, \infty)$, let 
$$\tau_M \doteq \inf\{t \ge 0: Y_i(t) \le -M \mbox{ for some } i \in [n]\}.$$
Denoting the probability (resp. expectation) operator under which $\cll(\bfY^n(0)) = \pi_n$ as $\PP_{\pi_n}$
(resp. $\EE_{\pi_n}$), we have 
 from the Markov property of $\bfY^n$  that
$$
V_K(\bfY^n(t \wedge \tau_M)) - V_K(\bfY^n(0))- \int_0^{t \wedge \tau_M} \cll_n^{\bfY}V_K(\bfY^n(s)) ds
$$
is a $\{\clf_t^n\}$-martingale under $\PP_{\pi_n}$, where $\clf_t^n \doteq 
\sigma\{\bfY^n(s): s\le t\}$, for $t\ge 0$.
Thus
$$
\EE_{\pi_n}[V_K(\bfY^n(t \wedge \tau_M))] - \EE_{\pi_n}[V_K(\bfY^n(0))]- \int_0^{t} \EE_{\pi_n}[\one\{s\le  \tau_M\}\cll_n^{\bfY}V_K(\bfY^n(s))] ds =0.
$$
Sending $M \to \infty$, and using dominated convergence theorem,
\begin{align}
\limsup_{M\to \infty} \int_0^{t} \EE_{\pi_n}[\one\{s\le  \tau_M\}\cll_n^{\bfY}V_K(\bfY^n(s))] ds &= \limsup_{M\to \infty}\left(\EE_{\pi_n}[V_K(\bfY^n(t \wedge \tau_M))] - \EE_{\pi_n}[V_K(\bfY^n(0))]\right)\nonumber\\
&= 
\EE_{\pi_n}[V_K(\bfY^n(t))] - \EE_{\pi_n}[V_K(\bfY^n(0))]=0.\label{eq:531}
\end{align}
We will now like to apply Fatou's lemma to the term on  the left side above. For this we give an upper bound on $\cll_n^{\bfY}V_K(\bfy)$.
Consider first the case $V(\bfy)\leq K$. Then,
\begin{align*}
   \cll^{\bfY}_n V_K(\bfy) &= \sum_{i=1}^n w(y_i) \mathbb{E}_Z \left[ V_K \left(y+ Z e_j - n^{-1} Z\one\right) -V_K(\bfy)\right] \\
    &= \sum_{i=1}^n w(y_i) \mathbb{E}_Z \left[ V_K \left(y+ Z e_j - n^{-1} Z\one\right) -V(\bfy)\right] \\
    &\leq \sum_{i=1}^n w(y_i) \mathbb{E}_Z \left[ V \left(y+ Z e_j - n^{-1} Z\one\right) -V(\bfy)\right] 
    = \cll^{\bfY}_n V(\bfy)
    \leq C_1,
\end{align*}
where $C_1$ is as introduced earlier in the proof.
Also, if $V(\bfy)>K$, 
$$\cll^{\bfY}_n V_K(\bfy) = \sum_{i=1}^n w(y_i) \mathbb{E}_Z \left[ V_K \left(y+ Z e_j - n^{-1} Z\one\right) -K\right] \leq 0.$$
Hence, we have that for every $K>0$, $\cll^{\bfY}_n V_K(\cdot)\leq C_1$.
Thus, using Fatou's lemma in \eqref{eq:531}, 
\begin{align*}
 t\EE_{\pi_n}[\cll_n^{\bfY}V_K(\bfY^n(0))] = \int_0^{t} \EE_{\pi_n}[\cll_n^{\bfY}V_K(\bfY^n(s))] ds \ge 0.
\end{align*}
Using the fact that $\cll_n^{\bfY}V_K(\bfY^n(0)) \to \cll_n^{\bfY}V(\bfY^n(0))$ a.s., as $K\to \infty$, and using Fatou's lemma together with the inequality $\cll_n^{\bfY}V_K(\cdot)\leq C_1$ again, we now see that
$\EE_{\pi_n}[\cll_n^{\bfY}V(\bfY^n(0))] \ge 0$.
Using this inequality together with \eqref{st2}, we now obtain $$\mathbb{E}_{\pi_n} \left[ \frac{\delta\mn}{n} \sum_{i=1}^n |Y_i^n(0)| + \frac{2\delta\mn}{n}\sum_{i=1}^n (Y_i^n(0))^- w(Y_i^n(0)) - C_1\right] \leq 0,$$
which together with the inequality $y^+w(y) \le |y|w(0)$
completes the proof of \eqref{st.thm1}. 

Next, we argue uniqueness of the stationary distribution under the spread-out condition. In view of Lemma \ref{lem:unique}, it suffices to construct a coupling $(\bfY, \tilde \bfY)$ of the Markov process under consideration, with $\PP_{\bfy,\tilde \bfy}$ denoting its law when started from $(\bfy, \tilde \bfy)$, such that the associated coupling time $\tau$ satisfies $\PP_{\bfy,\tilde \bfy}(\tau < \infty)>0$ for all $(\bfy, \tilde \bfy) \in \cls_0^2$. The measurability condition required in Lemma \ref{lem:unique} will be immediate from the construction.

We will couple the uncentered $n$-particle systems $(\bfX, \tilde \bfX)$, started from $(\bfy, \tilde \bfy)$, which will produce the coupling for the centered systems as in Lemma \ref{lem:wF}. The \emph{jump rates} will be coupled in the same way as in Lemma \ref{lem:wF}, namely, for $t \ge 0$, if $(\bfX(t-), \tilde \bfX(t-)) = (\bfx, \tilde \bfx)$, then at time $t$ the $i$-th particles in both systems jump together (synchronous jump) at rate $w(x_i - \bar \bfx) \wedge w(\tilde x_i - \bar{\tilde{\bfx}})$, the $i$-th particle jumps in $\bfX$ but not in $\tilde \bfX$ at rate $(w(x_i - \bar \bfx) - w(\tilde x_i - \bar{\tilde{\bfx}}))^+$, and the $i$-th particle jumps in $\tilde \bfX$ but not in $\bfX$ at rate $(w(x_i - \bar \bfx) - w(\tilde x_i - \bar{\tilde{\bfx}}))^-$. The synchronous \emph{jump sizes}, however, will not be the same as in Lemma \ref{lem:wF}. This is where we will use the spread-out property to construct a coupling $(Z_{i,k}, \tilde Z_{i,k})_{k \ge 1}$ of the successive jump sizes of the $i$-th particle in the coupled systems to ensure a non-zero chance of them meeting in a finite number of steps.

Observe that, ignoring the jump rates, the process $S_{i,k} \doteq y_i+ \sum_{j=1}^k Z_{i,j}$
(resp. $\tilde S_{i,k}  \doteq \tilde y_i+ \sum_{j=1}^k \tilde Z_{i,j}$),
$k \in \NN$,  with $S_{i,0} = y_i$ (resp.  $\tilde S_{i,0} = \tilde y_i$), gives epochs of a renewal process starting from $y_i$ (resp. $\tilde y_i$). By \cite[Chapter III, Section 5]{lindvall2002lectures}, there exists $m \in \NN$ and a coupling $(Z_{i,k}, \tilde Z_{i,k})_{1 \le k \le m}$ of the jump sizes such that $\PP(S_{i,m}= \tilde S_{i,m})>0$ for all $1 \le i \le n$. For each $i$, we use this coupling $(Z_{i,k}, \tilde Z_{i,k})$ for $k \le m$. The coupling is extended by requiring $Z_{i,k}=\tilde Z_{i,k}$ for $k>m$.
{\di A brief sketch of the coupling from \cite{lindvall2002lectures}, which relies on the key `spread-out' property, is as follows. For a fixed $i$, one first tracks alternating overshoots between independently run copies of these processes to bring the discrepancy between the two renewal clocks into an arbitrarily small interval with positive probability. This uses the fact that, as \(\theta\) is spread-out, there exists \(m\in\mathbb N\) such that \(\theta^{*m}\) has a nontrivial absolutely continuous component. Now, by employing the maximal coupling between the (overlapping laws of the) \(m\)-step future increments, this yields a strictly positive probability of exact coincidence after finitely many further steps.}

With the above coupling, we now consider the coupling of $(\bfX, \tilde \bfX)$ in which the first $m$ jumps of the $2n$-dimensional process are synchronous jumps for the first particle given by 
$(S_{1,j}, \tilde S_{1,j})_{1\le j \le m}$ followed by $m$ synchronous jumps for the second particle and so on.
The positivity of $w$ ensures that with this coupling, after the first $mn$ total jumps, the states of $\bfX$ and $\tilde \bfX$ are the same with positive probability proving that $\PP_{\bfy,\tilde \bfy}(\tau < \infty)>0$.
By Lemma \ref{lem:unique}, the uniqueness of stationary distribution follows.

Finally, we show \eqref{velm}. Observe that
\begin{equation}\label{mn1}
m_n(t) = \mn\int_0^t \frac{1}{n}\sum_{i=1}^n w(Y^n_i(s))ds + M_n(t), \ t \ge 0,
\end{equation}
where $M_n$ is a martingale under $\PP_{\pi_n}$ with quadratic variation
$$
\langle M_n \rangle (t) = \frac{\vt}{n^2}\int_0^t\sum_{i=1}^n w(Y^n_i(s))ds.
$$
As $\pi_n$ is the unique stationary distribution for the process $\bfY^n$, it is ergodic and hence, as the function $\bfy \mapsto \frac{1}{n}\sum_{i=1}^n w(y_i)$ is in $L^1(\pi_n)$ by \eqref{st.thm1}, using Birkhoff's ergodic theorem, we have for $\bfX^n(0) = \bfx$,
\begin{equation}\label{mn2}
\frac{1}{t}\int_0^t \frac{1}{n}\sum_{i=1}^n w(Y^n_i(s))ds \to \EE_{\pi_n}\left[\frac{1}{n}\sum_{i=1}^n w(Y^n_i(0))\right]
\end{equation}
almost surely for $\pi_n$-a.e. $\bfx$. Recalling the Lyapunov function $V$ and the stopping time $\tau_M$, note that for any $\bfx \in \cls_0$ and $M > 0$,
$
V(\bfY^n(t \wedge \tau_M)) - V(\bfY^n(0))- \int_0^{t \wedge \tau_M} \cll_n^{\bfY}V(\bfY^n(s)) ds
$
is a $\{\clf^n_t\}$-martingale under $\PP_{\bfx}$, where the latter is the probability measure under which $\bfX^n(0) = \bfx$. Thus, from \eqref{st2}, letting $M \to \infty$,
$$
2\delta \mn\EE_{\bfx}\left[\int_0^{t}\frac{1}{n}\sum_{i=1}^n Y^n_i(s)^-w(Y^n_i(s))ds\right] \le C_1 t + V(\bfx).
$$
In particular, since $w(y) \le y^{-} w(y) + w(-1)$, for any $\bfx \in \cls_0$,
$$
\sup_{t > 0}t^{-1}\EE_{\bfx}\langle M_n \rangle (t) = \frac{\vt}{n}\sup_{t > 0}\EE_{\bfx}\left[\frac{1}{t}\int_0^t\frac{1}{n}\sum_{i=1}^n w(Y^n_i(s))ds\right] < \infty.
$$
Thus, by Doob's $L^2$-maximal inequality, for any $\bfx \in \cls_0$,  $\frac{M_n(t)}{t} \to 0$ in probability
under $\PP_x$, as $t \to \infty$. 
Using this along with \eqref{mn1} and \eqref{mn2}, we obtain
$$
\frac{m_n(t)}{t} \to \mn\EE_{\pi_n}\left[\frac{1}{n}\sum_{i=1}^n w(Y^n_i(0))\right]
$$
in probability under $\PP_x$, as $t \to \infty$ for $\pi_n$-a.e. $\bfx$. The last equality in \eqref{velm} is a consequence of the exchangeability of $\pi_n$. This can be deduced, for example, by observing that $t^{-1} \int_0^t\PP_{\bfU^n}\left(\bfY^n(s) \in \cdot\right)ds$ converges to $\pi_n$ when $\bfU^n = (U_1,\dots, U_n)'$ is distributed as an iid product measure $\nu^{\otimes n}$ satisfying $\EE[U_1^2]<\infty$, using the uniqueness of the stationary distribution, the weak Feller property and the tightness obtained from \eqref{eq:904} (cf.  \cite[Corollary 1.18]{eberle2015markov}).
\hfill
\qedsymbol

\subsection{Proof of Theorem \ref{stat.pr}.}
\label{subsec6.2}
   \noindent From the moment bound in Theorem \ref{stat.ex}, and using Assumption \ref{ass.w.LT}, 
   $$\lim_{A\to \infty} \limsup_{n\to \infty} \mathbb{E}_{\pi_n} \left[ \frac{1}{n}\sum_{i=1}^n |Y^n_i(t)|\mathbf{1} [Y^n_i(t)\leq -A]\right] \leq \lim_{A\to \infty} \limsup_{n\to \infty} \frac{1}{w(-A)}\mathbb{E}_{\pi_n} \left[ \frac{1}{n}\sum_{i=1}^n |Y^n_i(t)|w(Y^n_i(t)) \right] =0.$$
   Combining this fact with the observation that, for $y \in \RR$, $|y|\mathbf{1}[|y|\ge A] \le 2(y-A/2)^+$,
  we see that, in order to prove the theorem, it suffices to show that for any $\eta>0$, 
    \begin{equation}\label{st.suff}
        \lim_{A\to \infty}\limsup_{n\to \infty} \mathbb{P}_{\pi_n} \left[ \frac{1}{n}\sum_{i=1}^n(Y^n_i(t)-A)^+ \geq \eta \right]=0.
    \end{equation}
    For this, we consider, for fixed $A>0$ the function, $V_A(\bfy) \doteq \frac{1}{n}\sum_{i=1}^n (y_i-A)^{+^2}$, $\bfy = (y_1, \ldots , y_n) \in \cls^0$.
    Note that, for $\bfy \in \cls^0$,
    \begin{multline}
        \cll^{\bfY}_n V_A(\bfy)=\frac{1}{n}\sum_{i=1}^n w(y_i)\mathbb{E}_Z \left[ \left(y_i+\frac{n-1}{n}Z-A \right)^{+^2} - \left(y_i-\frac{Z}{n}-A\right)^{+^2}\right]\\ + \left[ \frac{1}{n}\sum_{j=1}^n w(y_j)\right]\sum_{i=1}^n \mathbb{E}_Z \left[ \left(y_i-\frac{Z}{n}-A\right)^{+^2} - (y_i-A)^{+^2}\right].\label{eq:851}
    \end{multline}
    Consider the first term in the above sum. For $y_i\geq A$, 
    \begin{multline*}
        \mathbb{E}_Z\left[ \left(y_i+\frac{n-1}{n}Z-A \right)^{+^2} - \left(y_i-\frac{Z}{n}-A\right)^{+^2}\right] \\ \leq \mathbb{E}_Z \left[ \left\{ \left(y_i-\frac{Z}{n}-A\right)^+ +Z\right\}^2 - \left(y_i-\frac{Z}{n}-A\right)^{+^2}\right] 
        = \mathbb{E}_Z \left[ 2\left(y_i-\frac{Z}{n} -A\right)^+ Z + Z^2\right] 
        \leq 2\mn (y_i -A)^+ + \vt.
    \end{multline*}
    Also, if $y_i<A$,
    \begin{align*}
        \mathbb{E}_Z \left[ \left(y_i+\frac{n-1}{n}Z-A \right)^{+^2} - \left(y_i - \frac{Z}{n}-A \right)^{+^2}\right] = \mathbb{E}_Z \left[ y_i+ \left(1-\frac{1}{n}\right)Z-A \right]^{+^2}
        \leq \mathbb{E}_Z [Z-(A-y_i)]^{+^2}.
    \end{align*}
    Combining the above estimates,
    \begin{align}
        &\frac{1}{n}\sum_{i=1}^n w(y_i) \mathbb{E}_Z \left[ \left(y_i+\frac{n-1}{n}Z-A \right)^{+^2} - \left(y_i - \frac{Z}{n} -A\right)^{+^2} \right] \\
        &\leq \frac{2}{n}\mn \sum_{i: y_i\geq A} (y_i-A)^+ w(y_i) + w(A) \vt + \frac{1}{n}\sum_{i:y_i<A} w(y_i) \mathbb{E}_Z [Z-(A-y_i)]^{+^2} \\
        &\leq \epsilon_A + \frac{2w(A)\mn}{n}\sum_{i=1}^n (y_i-A)^+,\label{st3}
    \end{align}
    where  $\epsilon_A\doteq \sup_{x\geq 0} w(A-x) \mathbb{E} (Z-x)^{+^2} + w(A) \vt$. Using Assumptions \ref{Z2} and \ref{ass.w.LT}, we see that $\lim_{A\to \infty} \epsilon_A =0$.

  Next, we consider the second term in the sum in \eqref{eq:851}. Note that,
    \begin{multline}
        \sum_{i=1}^n \mathbb{E}_Z \left[ \left(y_i-\frac{Z}{n}-A\right)^{+^2} - (y_i-A)^{+^2}\right] \leq \mathbb{E}_Z\left[ \sum_{i: y_i\geq A+\frac{Z}{n}} \left\{ -\frac{2Z}{n}(y_i-A) + \frac{Z^2}{n^2}\right\}\right] \\
        \leq \frac{\vt}{n} + \mathbb{E}_Z\left[ -\frac{2Z}{n}\sum_{i: y_i\geq A}  (y_i-A)\right] + \mathbb{E}_Z\left[ \frac{2Z}{n}\sum_{i: A+\frac{Z}{n}>y_i\geq A}  (y_i-A)\right]
        \le 
        \frac{3\vt}{n} - \frac{2\mn}{n}\sum_i (y_i-A)^+  . \label{st4}     
    \end{multline}
    Thus, using \eqref{st3} and \eqref{st4}, 
    \begin{equation}\label{st.a}
        \cll^{\bfY}_nV_A(\textbf{y}) \leq \epsilon_A + \frac{2w(A)\mn}{n} \sum_{i=1}^n (y_i-A)^+ + \alpha(\bfy) \left[ \frac{3\vt}{n} - \frac{2\mn}{n}\sum_i (y_i-A)^+\right],
    \end{equation}
    where $\alpha(\textbf{y})\doteq \frac{1}{n}\sum_i w(y_i)$, $\bfy \in \RR^n$.
    Also, using a similar truncation and localization idea as in the proof of Theorem \ref{stat.ex} (see below \eqref{eq:904}), we  have that
    $\mathbb{E}_{\pi_n}[\cll^{\bfY}_nV_A(\textbf{Y}^n(0))]\geq 0$.
    Also, using Theorem \ref{stat.ex}, we have 
    $$\sup_n \mathbb{E}_{\pi_n} [\alpha(\textbf{Y}^n(0))]
    \le w(-1) + \sup_n \mathbb{E}_{\pi_n} \frac{1}{n}\sum_{i: Y^n_i(0) \le -1} w(Y^n_i(0))
    \le w(-1) + \sup_n \mathbb{E}_{\pi_n} \frac{1}{n}\sum_{i=1}^n |Y^n_i(0)| w(Y^n_i(0)) <\infty.$$
   Combining these observations, we now have that
    \begin{equation}\label{st.b}
        2\mn\mathbb{E}_{\pi_n} \left[ (\alpha(\textbf{Y}^n(0))-w(A))\left\{ \frac{1}{n}\sum_i (Y^n_i(0)-A)^+\right\}\right]\leq \frac{B_1}{n} + \epsilon_A,
    \end{equation}
    where
    $B_1 \doteq 3\vt \sup_n \mathbb{E}_{\pi_n} [\alpha(\textbf{Y}^n(0))]
    $.
    Using Theorem \ref{stat.ex} again, we see that 
    $$2\mn\mathbb{E}_{\pi_n} \left[ (\alpha(\textbf{Y}^n(0))-w(A))^- \left\{ \frac{1}{n}\sum_i (Y^n_i(0)-A)^+\right\} \right] \leq 2\mn w(A) \mathbb{E}_{\pi_n} \left[\frac{1}{n}\sum_i |Y^n_i(0)|\right] \leq w(A)B_2,$$
    where
    $B_2 \doteq 
    2\mn\sup_n \mathbb{E}_{\pi_n} \left[ \frac{1}{n}\sum_i |Y^n_i(0)| \right]<\infty$.
    Using this in \eqref{st.b}, we obtain 
    \begin{equation}\label{st.c}
        \mathbb{E}_{\pi_n} \left[ (\alpha(\textbf{Y}^n(0))-w(A))^+ \left\{ \frac{1}{n}\sum_i (Y^n_i(0)-A)^+\right\} \right] \leq \frac{1}{2\mn} \left[\frac{B_1}{n} + \epsilon_A + w(A)B_2 \right].
    \end{equation}
    Fix $\delta_0>0$ such that $B(\delta_0) \doteq \sup \{x: w(x)\geq 2\delta_0 \} >0$.
    Let $\delta \in (0, \delta_0)$ and choose $A_0>0$ such that $w(A_0) <\delta$, which can be done as Assumption \ref{ass.w.LT} holds. Then, for all $A\ge A_0$, and for fixed $\eta>0$,
    \begin{multline*}
        \mathbb{P}_{\pi_n} \left[\frac{1}{n}\sum_i (Y^n_i(0)-A)^+ \geq \eta \right] \leq \mathbb{P}_{\pi_n} [\alpha(\textbf{Y}^n(0))<\delta] + \mathbb{P}_{\pi_n} \left[ \frac{1}{n}\sum_i(Y^n_i(0)-A)^+ \geq \eta, \; \alpha(\textbf{Y}^n(0)) \geq \delta \right] \\
        \leq \mathbb{P}_{\pi_n}[\alpha(\textbf{Y}^n(0)) < \delta] + \mathbb{P}_{\pi_n} \left[ (\alpha(\textbf{Y}^n(0))-w(A))^+ \left\{ \frac{1}{n}\sum_i (Y^n_i(0)-A)^+\right\} \geq \eta(\delta-w(A))\right] \\
        \leq \mathbb{P}_{\pi_n}[\alpha(\textbf{Y}^n(0)) < \delta] + \frac{1}{2\mn\eta(\delta-w(A))}
         \left[\frac{B_1}{n} + \epsilon_A + w(A)B_2 \right].
    \end{multline*}
   Sending $n\to \infty$ and then $A\to \infty$, and recalling that $w(A)$ and $\epsilon_A$ converge to  $0$ as $A\to \infty$, we see
    \begin{equation}\label{st.d}
        \lim_{A\to \infty} \limsup_{n\to \infty} \mathbb{P}_{\pi_n} \left[ \frac{1}{n}\sum_i (Y^n_i(0)-A)^+ \geq \eta \right] \leq \limsup_{n\to \infty} \mathbb{P}_{\pi_n} [\alpha(\textbf{Y}^n(0)) < \delta].
    \end{equation}
    Finally we argue that the right-hand side of \eqref{st.d} converges to $0$ as $\delta\to 0$. 
    Observe that if for some $\bfy \in \cls^0$,
    $\alpha(\bfy) = n^{-1}\sum_i w(y_i)< \delta$, then 
    $$\frac{1}{n}\# \{ i: y_i\leq B(\delta)\} \leq \frac{1}{w(B(\delta))}\frac{1}{n}\sum_{i: y_i\leq B(\delta)} w(y_i) < \frac{\delta}{w(B(\delta))}\leq \frac{1}{2}.$$
    Consequently, $\frac{1}{n}\# \{ i: y_i\ge B(\delta)\} \ge 1/2$, and therefore, for such $\bfy$,
   $$\frac{1}{2n}\sum_i |y_i| = \frac{1}{n}\sum_i y_i^+ \geq \frac{1}{n}\sum_{i:y_i\geq B(\delta)} y_i \geq \frac{B(\delta)}{2}.$$
    From this it follows that, $$\limsup_{n\to \infty} \mathbb{P}_{\pi_n} [\alpha(\textbf{Y}^n(0))< \delta] \leq \limsup_{n\to \infty} \mathbb{P}_{\pi_n} \left[ \frac{1}{n}\sum_i |Y^n_i(0)|\geq B(\delta)\right]\leq \frac{\sup_{n} \EE_{\pi_n} Y_1(0)}{B(\delta)}.$$
    Since $B(\delta) \to \infty$ as $\delta \to 0$, we now have  from Theorem \ref{stat.ex} that
    $\lim_{\delta \to 0} \limsup_{n\to \infty} \mathbb{P}_{\pi_n} [\alpha(\textbf{Y}^n(0){\di)}<\delta] =0$.
    The statement in \eqref{st.suff} now follows on sending $\delta \to 0$ in  \eqref{st.d}, and thus, as discussed previously, the result follows.
\hfill
\qedsymbol
\subsection{Proof of Theorem \ref{thm.stat.sol}.}
\label{subsec6.3}
Let $\bfX^n(0) = (X^n_1(0), \ldots , X^n_n(0))$ be such that $\cll(\bfX^n(0)) = \pi_n$. Note that, a.s., we must have $m_n(0) = \frac{1}{n} \sum_{i=1}^n X^n_i(0) = 0$, and thus $\bfY^n(0) = \bfX^n(0) - m_n(0)\one = \bfX^n(0)$. Now define $\{\bfX^n(t), t \ge 0\}$ as in Section \ref{sec:modres} with the above initial condition, and let $\bfY^n(t) = \bfX^n(t) - m_n(t) \one$, where
$m_n(t) = \frac{1}{n} \sum_{i=1}^n X^n_i(t)$. By construction $\{\bfY^n(t), t \ge 0\}$ is a stationary process with the law at each time instant given as $\pi_n$.

To emphasize dependence on $\pi_n$, we will denote the probability space on which $\bfX^n, \bfY^n$ are defined as $(\Om^n, \clf^n, \PP_{\pi_n})$.
From Theorem \ref{stat.pr} we see that Assumption \ref{ass.t1} holds with $\PP$ replaced by
$\PP_{\pi_n}$. Also, using Assumption \ref{eq:wtails} and Theorem \ref{stat.ex} we see that
Assumption \ref{ass.char} holds with $\PP$ replaced by
$\PP_{\pi_n}$; see Remark \ref{rem:conds}.

Thus we have that all assumptions in Theorem \ref{thm.mainFL} parts (a) and (b) hold with
$\PP$ replaced by $\PP_{\pi_n}$. From part (a) of this theorem we now have that, with 
$\mu_n(t) \doteq \frac{1}{n} \sum_{i=1}^n \delta_{X^n_i(t)}
$, $\{ \mu_n\}_{n\in \NN}$ is $C$-tight in $\mathcal{D}([0,\infty):\mathcal{P}_1(\mathbb{R}))$. This also shows that $\{m_n\}_{n\in \NN}$ is $C$-tight in $\mathcal{D}([0,\infty):\RR)$. It then follows that, with $\nu_n(t) \doteq
\frac{1}{n} \sum_{i=1}^n \delta_{Y^n_i(t)}$,  $\{ \nu_n\}_{n\in \NN}$ is $C$-tight in $\mathcal{D}([0,\infty):\mathcal{P}_1(\mathbb{R}))$. 
Next, using part (b) of Theorem \ref{thm.mainFL} (with
$\PP$ replaced by $\PP_{\pi_n}$) we see that any distributional limit point $\mu$ of $\mu_n$
in $\mathcal{D}([0,\infty):\mathcal{P}_1(\mathbb{R}))$ solves the McKean-Vlasov equation \eqref{eq:mfeq1}. It then follows that, along the same subsequence $\mu_n$ converges, the sequence $\nu_n$ also converges in distribution in $\mathcal{D}([0,\infty):\mathcal{P}_1(\mathbb{R}))$ and   the corresponding weak limit point $\nu$ solves the equation \eqref{FL}. Furthermore, for each $n$, from the stationarity of $\{\bfY^n\}$ we have that $\nu_n$ is stationary as well, and then it follows that $\nu$ is stationary too. Next, using Fatou's lemma and Theorem \ref{stat.ex} it follows that $\EE\langle w, \nu(0)\rangle <\infty$.
Since $\nu$ is a stationary solution and $\EE\langle w, \nu(0)\rangle <\infty$,
 by assumption in the theorem, we must have that, for each $t\ge 0$, $\nu(t)$ is non-random a.s., and in fact $\nu(t)= \nu(0) = \nu^*$ a.s. Since the weakly convergent subsequence was arbitrary, this shows that $\nu_n(0) = \mu_n(0)$ in fact converges in probability to $\nu^*$ as $n\to \infty$. The first statement in the theorem now follows on noting that $\cll(\nu_n(0)) = \pi_n \circ (\Theta^n)^{-1} = \Pi_n$. The second statement
now follows by standard method, cf. \cite[Lemma 3.19]{chaintron2022propagation}.

To obtain the final statement of the theorem, note that the first equality in \eqref{meanlt} follows by Theorem \ref{stat.ex} (see \eqref{velm}). The second equality is verified from the convergence $\Pi_n \to \delta_{\nu^*}$ in probability and a uniform integrability estimate that follows from \eqref{st.thm1}.
\hfill
\qedsymbol
\section{Proof of Theorem \ref{thm.ss.exp}}\label{s.Exp}
Throughout this section we consider the exponential-exponential case, namely we assume that $w$ and $\theta$ are as in the statement of Theorem \ref{thm.ss.exp}.

Let $\{\bar\nu(t), t \ge 0\}$  be a stationary solution of \eqref{FL}.
Then a.s., 
$$\theta(t) \doteq \langle w, \bar \nu(t) \rangle <\infty, \mbox{ for a.e. } t>0,  \quad \Theta(t) \doteq \int_0^t \theta(s) ds <\infty, \mbox{ for every } t>0.$$
Let $\Om_0 \subset \Om$ be such that $\PP(\Om_0)=1$ and for all $\om \in \Om_0$,
 $\Theta(t)(\om) <\infty$ for every $t>0$, and $\{\bar \nu(t)(\om), t\ge 0\}$ is a solution of \eqref{FL}. In what follows, we will occasionally suppress $\om$ from the notation.

The next two lemmas play a key role in the proof of Theorem \ref{thm.ss.exp}. 

Lemma \ref{lem_cdf} constructs, for a given $\om \in \Om_0$, an explicit solution $\nu^*$ to equation \eqref{pde_m} below, which has a similar form as \eqref{FL}, except that the McKean-Vlasov interaction term $\langle w, \nu(s)\rangle$ appearing as the `downward drift' in \eqref{FL} (which depends on the evolving distribution) is replaced by $\theta(t)(\om)$ obtained from the stationary solution. This makes the downward drift function a fixed input to the equation and makes the resulting equation easier to solve.
\begin{lemma}\label{lem_cdf}
Fix $\om \in \Om_0$, and define, for $(t,x) \in [0,\infty)\times \RR$,
$$F^*(t,x)=\frac{\gamma}{\Gamma(1+\gamma\beta^{-1})} \int_{-\infty}^{\zeta(x, t)} \exp\left( -\gamma y-e^{-\beta y}\right)\, dy,$$
where 
$$\zeta(x, t)\doteq {x+\gamma^{-1}\Theta(t) - \beta^{-1}\log\left( 1+\int_0^t e^{\beta \gamma^{-1}\Theta(s)}\,ds\right)}.$$
Then the following hold.
\begin{enumerate}[(a)]
\item For each $t\ge 0$, $F^*(t,\cdot)$ is a cumulative distribution function.

\item For $t\ge 0$, let $\nu^*(t)$ be the probability measure associated with the cumulative distribution function $F^*(t,\cdot)$. Then $\{\nu^*(t), t \ge 0\}$ satisfies
\begin{equation}\label{pde_m}
        \langle f, \nu^*(t) \rangle = \langle f, \nu^*(0) \rangle + \int_0^t \int g_f(y)w(y)\nu^*(s,dy)-\int_0^t \gamma^{-1} \theta(s) \langle f',\nu^*(s) \rangle ds,
        \; f \in C^1_c(\RR), \; t \ge 0.
    \end{equation}
\end{enumerate}

    
\end{lemma}
\begin{proof}
Part (a) is immediate on observing that, for every $t\ge 0$, $x \mapsto F^*(t,x)$ is a continuous increasing function, with
$\lim_{x\to -\infty} F^*(t,x)=0$ and $\lim_{x\to \infty} F^*(t,x)=1$, where the last equality follows from the identity
$$\frac{\gamma}{\Gamma(1+\gamma\beta^{-1})} \int_{-\infty}^{\infty} \exp\left( -\gamma y-e^{-\beta y}\right)\, dy =1.$$
Consider now part (b).
We claim that, for a.e. $t\in (0,\infty)$, and $x \in \RR$,
 \begin{equation}\label{pde}
        \frac{\partial F^*(t,x)}{\partial t} = -e^{-\gamma x} \int_{-\infty}^x e^{\gamma y} w(y) F^*(t, dy) + \gamma^{-1}\theta(t) \frac{\partial F^*(t,x)}{\partial x}.
    \end{equation}
To verify the claim we compute each of the three terms in the above equation.
Let \begin{equation}
\label{eq:809}
\alpha(t) \doteq \gamma^{-1}\Theta(t) - \beta^{-1}\log\left( 1+\int_0^t e^{\beta \gamma^{-1}\Theta(s)}\,ds\right), \; K \doteq \frac{\gamma}{\Gamma(1+\gamma\beta^{-1})}.\end{equation}
Then, for a.e. $t$,
$$
 \frac{\partial F^*(t,x)}{\partial t} = K \alpha'(t) e^{-\gamma(x+\alpha(t))} \exp (-e^{-\beta (x+ \alpha(t))}).
$$
Next,
 \begin{align*}
            \gamma^{-1}\theta(t) \frac{\partial F^*(t,x)}{\partial x} = K\gamma^{-1}\theta(t) e^{-\gamma(x+\alpha(t))} \exp (-e^{-\beta (x+ \alpha(t))}).
        \end{align*}
Finally,
\begin{align*}
-e^{-\gamma x} \int_{-\infty}^x e^{\gamma y} w(y) F^*(t, dy)
&= -Ke^{-\gamma x}  \int_{-\infty}^x e^{(\gamma -\beta)y} \exp\left\{-\gamma(y+\alpha(t)) - e^{-\beta(y+\alpha(t))}\right\} dy\\
&=-Ke^{-\gamma x}  \int_{-\infty}^x e^{-\beta y} \exp\left\{-\gamma \alpha(t) - e^{-\beta(y+\alpha(t))}\right\} dy\\
&= -Ke^{-\gamma x} e^{(\beta-\gamma)\alpha(t)} \int_{-\infty}^x \exp\left\{-\beta( y+\alpha(t)) - e^{-\beta(y+\alpha(t))}\right\} dy\\
&= -K\beta^{-1} e^{\beta \alpha(t)} e^{-\gamma (x+\alpha(t))} \exp\{-\beta( x+\alpha(t))\}.
\end{align*}
Combining the above calculations, we now see that
\begin{align}
 &\frac{\partial F^*(t,x)}{\partial t} +e^{-\gamma x} \int_{-\infty}^x e^{\gamma y} w(y) F^*(t, dy) - \gamma^{-1}\theta(t) \frac{\partial F^*(t,x)}{\partial x}\\
 &= c(t,x)\left ( \alpha'(t) - \gamma^{-1}\theta(t) + \beta^{-1}e^{\beta\alpha(t)}\right)\label{eq:1110}
\end{align}
where
$c(t,x) = K  e^{-\gamma(x+\alpha(t))} \exp (-e^{-\beta (x+ \alpha(t))})$.
Furthermore, from \eqref{eq:809}, for a.e. $t$,
\begin{align*}
\alpha'(t) &= \gamma^{-1}\theta(t) - \frac{\beta^{-1} e^{\beta \gamma^{-1} \Theta(t)}}{1+ \int_0^t e^{\beta \gamma^{-1} \Theta(s)} ds}\\
&= \gamma^{-1}\theta(t) - \beta^{-1} e^{ \beta\left(\gamma^{-1}\Theta(t) - \beta^{-1} \log\left( 1+ \int_0^t e^{\beta\gamma^{-1} \Theta(s)} ds\right)\right)}
= \gamma^{-1}\theta(t)- \beta^{-1} e^{\beta \alpha(t)}.
\end{align*}
This shows that the term on the right side of \eqref{eq:1110} is $0$ for a.e. $t$ and proves the claim in \eqref{pde}.

We now prove the statement in part (b) of the lemma. Fix $f \in C^1_c(\RR)$. Then, multiplying \eqref{pde} throughout by $f'(x)$ and integrating over $x$ and $t$, we get
\begin{multline}
\langle f, \nu^*(t)\rangle - \langle f, \nu^*(0)\rangle
= - \int_{\RR} f'(x) F^*(t,x) dx + \int_{\RR} f'(x) F^*(0,x) dx\\
= - \gamma^{-1}\int_{\RR} \int_0^t f'(x) \theta(s) \frac{\partial F^*(s,x)}{\partial x} ds dx
+ \int_{\RR} f'(x) e^{-\gamma x} \int_0^t \int_{-\infty}^x e^{\gamma y} w(y) F^*(s,dy) ds dx\\
= - \gamma^{-1}\int_0^t \theta(s) \langle f', \nu^*(s)\rangle ds
- \int_0^t \int_{\RR} f(x) \frac{\partial}{\partial x}  \left(e^{-\gamma x} \int_{-\infty}^x e^{\gamma y} w(y) F^*(s,dy)\right) dx,
\label{eq:349}
\end{multline}
where to obtain the first equality and the last equality we have integrated by parts.
Next note that
\begin{align*}
&- \int_0^t \int_{\RR} f(x) \frac{\partial}{\partial x}  \left(e^{-\gamma x} \int_{-\infty}^x 
e^{\gamma y} w(y) F^*(s,dy)\right) dx \\
& = - \int_0^t \int_{\RR} f(x) \left( - \gamma e^{-\gamma x} \int_{-\infty}^x e^{\gamma y} w(y) F^*(s,dy)
+ w(x) \frac{\partial F^*(s,x)}{\partial x} \right) dx\\
&= \int_0^t \int_{\RR} f(x) \gamma e^{-\gamma x} \int_{-\infty}^x  e^{\gamma y} w(y) F^*(s,dy) dx
- \int_0^t \int_{\RR} f(x) w(x) F^*(s, dx)\\
&= \int_0^t \int_{\RR} g_f(y) w(y) \nu^*(s, dy) ds,
\end{align*}
where the last equality follows on recalling that $g_f(y) = \EE_Z(f(y+Z)) - f(y)$, and
\begin{align*}
\int_{\RR}  \EE_Z f(y+Z) w(y) F^*(s,dy)&= \gamma \int_{\RR} \int_0^{\infty} f(y+z) e^{-\gamma z} dz w(y) 
\frac{\partial F^*(s,y)}{\partial y} dy \\
&= \gamma \int_{\RR} e^{\gamma y} \int_y^{\infty} f(x) e^{-\gamma x}  w(y) \frac{\partial F^*(s,y)}{\partial y} dx \, dy\\
&= \gamma \int_{\RR} e^{-\gamma x} f(x) \int_{-\infty}^x e^{\gamma y} w(y)  F^*(s,dy) \, dx.
\end{align*}
Combining the above observation with \eqref{eq:349} we now see that $\{\nu^*(t), t\ge 0\}$ solves \eqref{pde_m} which completes
the proof of part (b) of the lemma.
\end{proof}
The next lemma shows that, for every $\om \in \Om_0$, the sup-norm distance between the cdf of $\nu^*$ introduced in the previous lemma and that of the stationary solution $\bar \nu$ approaches zero with time. Combining these results with the stationarity of $\bar \nu$ gives a tractable distributional representation for $\bar \nu$ which is exploited in the proof of Theorem \ref{thm.ss.exp}.
\begin{lemma}\label{lem_coup}
Let $\{\bar \nu(t), t \ge 0\}$ be as at the start of the section. Let  $\bar F(t,x) \doteq \bar \nu(t)(-\infty,x]$, $x\in \RR$,  be the random cumulative distribution function associated with $\bar \nu(t)$.
For a given $\om \in \Om_0$, let  $\{F^*(t,x), t \ge 0, x \in \RR\}$ be as defined in Lemma \ref{lem_cdf}. Then, for a.e. $\omega$,
    $$\lim_{t\to \infty}\sup_{x\in \mathbb{R}}|\Bar{F}(t,x)-F^*(t,x)| = 0.$$
\end{lemma}
\begin{proof}
Since for every $\om \in \Om_0$, $\{\bar \nu(t)(\om), t \ge 0\}$ is a solution of \eqref{FL},
\eqref{pde_m} holds with $\nu^*(\om)$ replaced by $\bar \nu(\om)$.
This shows that for every such $\om$, there exist (time inhomogeneous) Markov processes $\{\tilde Z^i(t), t \ge 0\}$, $i=1,2$, with infinitesimal generator
$$\tilde \cll_{Z} f(s, x) = g_f(x) w(x) - \gamma^{-1}\theta(s) f'(x), \; s \ge 0, x \in \RR, f \in C^1_c(\RR)$$
such that for every $t\ge 0$,
$\cll(\tilde Z^1(t)) = \bar \nu(t)$ and $\cll(\tilde Z^2(t)) =  \nu^*(t)$.
We will now construct a coupling $(Z^1, Z^2)$  of these Markov processes (namely, $Z^i$ and $\tilde Z^i$ have the same law, for $i=1,2$, and $Z^1, Z^2$ are defined on a common probability space), such that, 
\begin{equation}\label{eq:maintpt}
\mbox{ with } \tau \doteq \inf\{t \ge 0: Z^1(t) = Z^2(t)\},\;\; \lim_{t\to \infty} \PP(\tau> t) = 0.
\end{equation}
The result then follows on observing that
\begin{equation}\label{tveq}
\sup_{x \in \RR} |\bar F(t,x) - F^*(t,x)| \le \|\bar \nu(t) - \nu^*(t)\|_{\mbox{\tiny{TV}}}
\le 2 \PP(\tau >t),
\end{equation}
where $\|\cdot\|_{\mbox{\tiny{TV}}}$ is the total variation distance.

\noindent {\bf Construction of the Coupling.}
Recall the `optimal coupling' of two real valued random variables $U$ and $V$, with densities $f_U$ and $f_V$, that minimizes total variation distance. The associated joint law $\PP_{U,V}$ is given as follows. For $A \in \clb(\RR)$, $\PP(U \in A, U=V) \doteq \int_A f_U(x) \wedge f_V(x)dx$, $\PP(U \in A, U \neq V) \doteq \int_A (f_U(x) - f_V(x))^+dx$ and $\PP(V \in A, U \neq V) \doteq \int_A (f_U(x) - f_V(x))^-dx$. For $\gamma>0$ and $u,v \in \RR$, $u<v$, we will call the optimal coupling of $u + Exp(\gamma)$ and $v + Exp(\gamma)$ as a \emph{$(u,v)$-optimal coupling of $Exp(\gamma)$ random variables}. Note that with this coupling $\PP(U=V) = e^{-\gamma|u-v|}$.

Before we concretely describe the coupling of $(Z^1, Z^2)$, we give the intuitive content of the construction. The idea is to \emph{synchronously couple} the two processes till a stopping time $\sigma_1$ when the lower process comes within a prescribed distance of the upper one. During this coupled evolution, the lower process jumps whenever the upper one jumps, and they jump by the same amount. The lower process, in addition, can have individual jumps. This coupling gives the lower process an `additional push' to get closer to the upper one. 

Once the processes are close enough, the jumps immediately after time $\sigma_1$ in the two processes are optimally coupled to ensure a positive probability of them meeting in this jump. If they don't meet, the synchronous-optimal coupling strategy is repeated. This algorithm guarantees eventual coalescence of the two processes with probability one.

Note that, as we are dealing with jump processes, the lower process can actually overshoot the upper process at time $\sigma_1$. Thus, we describe the coupling in terms of the minimum and maximum processes respectively denoted by $\{Z^{i_{2k}}\}_{k \ge 0}$ and $\{Z^{j_{2k}}\}_{k \ge 0}$. 


Fix $\om \in \Om_0$ and $a>0$. Define $Z^1(0) \sim \bar \nu(0)$, $Z^2(0) \sim  \nu^*(0)$, given on a common probability space. We will recursively construct the coupling on random time intervals $\{[\sigma_l, \sigma_{l+1}]: l \ge 0\}$ given in terms of suitable stopping times $\{\sigma_l: l \ge 0\}$. Let $\sigma_0=0$. Suppose the stopping times $\{\sigma_l: l \le 2k\}$ and the processes $\{Z^1(t), Z^2(t), t \le \sigma_{2k}\}$ have already been constructed.

        \begin{enumerate}
            \item We work conditional on $\clf_{\sigma_{2k}}$, on the set $\{\sigma_{2k}<\infty\}$, where $\clf_t \doteq \sigma\{Z^1(s), Z^2(s), s \le t\}$, $t \ge 0$. Define coupled jump processes $\{Z^{i_{2k}}(t), Z^{j_{2k}}(t), 
            \sigma_{2k} \le t \le \sigma_{2k+1}\}$, where $$\sigma_{2k+1} \doteq  \inf\{t\ge \sigma_{2k}: Z^{i_{2k}}(t)\geq Z^{j_{2k}}(t)-a\},$$ recursively, as follows. 
            $$Z^{i_{2k}}(\sigma_{2k})\doteq Z^1(\sigma_{2k}) \wedge Z^2(\sigma_{2k}), \; Z^{j_{2k}}(\sigma_{2k})\doteq Z^1(\sigma_{2k}) \vee Z^2(\sigma_{2k}).$$
            On the interval $[\sigma_{2k},\sigma_{2k+1}]$,  processes $Z^{i_{2k}}, Z^{j_{2k}}$ are synchronously coupled in the following manner. They jump simultaneously at instant $t$, with rate $w(Z^{j_{2k}}(t))$, and common jump random variable distributed as $\theta = Exp(\gamma)$.
            In addition, on this interval, independently, $Z^{i_{2k}}$ jumps (and $Z^{j_{2k}}$ does not), with rate
            $w(Z^{i_{2k}}(t)) - w(Z^{j_{2k}}(t))$ (which is nonnegative on this interval), and jump distribution $\theta$. In between jumps, both processes drift downwards with speed $\gamma^{-1} \theta(\cdot)$. 
            Let, for $t \in [\sigma_{2k}, \sigma_{2k+1}]$,
            \begin{align*}
            Z^i(t) &\doteq Z^{i_{2k}}(t) \one\{Z^i(\sigma_{2k})\le Z^{i^*}(\sigma_{2k})\} + Z^{j_{2k}}(t) \one\{Z^i(\sigma_{2k})> Z^{i^*}(\sigma_{2k})\},\; i = 1,2,
            \end{align*}
            where $i^* = 3-i$. In the above and in the rest of the construction we follow the convention that the interval $[a, \infty]$ is taken to be $[a, \infty)$ for $a \in [0,\infty)$.

            \item\label{oc} Next, on $\{\sigma_{2k+1} <\infty\}$, we will define coupled jump processes $\{Z^{i_{2k+1}}(t), Z^{j_{2k+1}}(t), \sigma_{2k+1} \le t \le \sigma_{2k+2}\}$. The description below is conditional on $\clf_{\sigma_{2k+1}}$. Set
            $$Z^{i_{2k+1}}(\sigma_{2k+1}) \doteq Z^{1}(\sigma_{2k+1})\wedge Z^{2}(\sigma_{2k+1}), \; Z^{j_{2k+1}}(\sigma_{2k+1}) \doteq Z^{1}(\sigma_{2k+1})\vee Z^{2}(\sigma_{2k+1}).$$ 
            For $t \ge \sigma_{2k+1}$, $(Z^{i_{2k+1}}, Z^{j_{2k+1}})$ jumps at instant $t$, with rate $w(Z^{j_{2k+1}}(\sigma_{2k+1}) - \int_{\sigma_{2k+1}}^t\gamma^{-1}\theta(s)ds)$ and jump size given by $\left(E^1_k- \int_{\sigma_{2k+1}}^t\gamma^{-1}\theta(s)ds, E^2_k-\int_{\sigma_{2k+1}}^t\gamma^{-1}\theta(s)ds\right)$ where $(E^1_k,E^2_k)$ is a 
            $(Z^{i_{2k+1}}(\sigma_{2k+1}),$ $ Z^{j_{2k+1}}(\sigma_{2k+1}))$-optimal coupling of $Exp(\gamma)$ random variables.  
            
            In addition, at instant $t \ge \sigma_{2k+1}$, independently with rate $w(Z^{i_{2k+1}}(\sigma_{2k+1})- \int_{\sigma_{2k+1}}^t\gamma^{-1}\theta(s)ds) - w(Z^{j_{2k+1}}(\sigma_{2k+1})- \int_{\sigma_{2k+1}}^t\gamma^{-1}\theta(s)ds)$, $Z^{i_{2k+1}}$ jumps 
            (and $Z^{j_{2k+1}}$ does not) with jump distribution $\theta$. As before, up until the first time after $\sigma_{2k+1}$ when at least one of the two processes jumps, both processes drift downwards with speed $\gamma^{-1} \theta(\cdot)$.
            Let $\sigma_{2k+2}$ be this first jump time after $\sigma_{2k+1}$.
            For $t \in (\sigma_{2k+1}, \sigma_{2k+2}]$, let
           \begin{align*}
            Z^i(t) &\doteq Z^{i_{2k+1}}(t) \one\{Z^i(\sigma_{2k+1})\le Z^{i^*}(\sigma_{2k+1})\} + Z^{j_{2k+1}}(t) \one\{Z^i(\sigma_{2k+1})> Z^{i^*}(\sigma_{2k+1})\},\; i = 1,2.
            \end{align*}
        \end{enumerate}
        This completes the recursive construction of processes $\{Z^1(t), Z^2(t), t\ge 0\}$
        such that for $t\ge 0$, $\bar \nu(t) = \cll(Z^1(t))$, $\nu^*(t) = \cll( Z^2(t))$.
        Note that, if $Z^1(t)= Z^2(t) $ for some $ t \geq 0$, then $ Z^1(s)= Z^2(s)$ for every $ s\geq t$.
        
Now, we will show that the probability of $Z^1$ and $Z^2$ coalescing in any `excursion' $[\sigma_{2k}, \sigma_{2k+2}]$ is bounded below by a positive number independent of $k$.  Define, for $k \ge 0$, the event $$\cle_k \doteq \{ Z^1(\sigma_{2k+2})-  Z^1(\sigma_{2k+2}-)  \neq 0
\text{ and } Z^2(\sigma_{2k+2})-  Z^2(\sigma_{2k+2}-) \neq 0\} \cap \{\sigma_{2k+1} <\infty\}.$$
This corresponds to the event that at the time instant $\sigma_{2k+2}$, both $Z^1$ and $Z^2$ jump together. Let $\bar \clf_k$ denote the sigma field generated by $\cle_k$ and all events in $\clf_{\sigma_{2k+1}}$. Observe that
\begin{align}\label{cb1}
\PP(\tau \le \sigma_{2k+2} \, \vert \, \clf_{\sigma_{2k}}) &\ge \PP \left(\cle_k \cap \{ Z^1({\sigma_{2k+2}}) = Z^2({\sigma_{2k+2}})\} \, \vert \, \clf_{\sigma_{2k}}\right)\notag\\
&= \EE\left[\one_{\cle_k} \, \PP\left(Z^1({\sigma_{2k+2}}) = Z^2({\sigma_{2k+2}}) \, \vert \, \bar \clf_k\right) \, \vert \, \clf_{\sigma_{2k}}\right].
\end{align}
By the explicit construction of the optimal coupling, on the event $\cle_k$,
\begin{equation}
\PP\left( Z^1({\sigma_{2k+2}}) = Z^2({\sigma_{2k+2}}) \, \vert \, \bar \clf_k\right) = \exp\{ -\gamma(Z^{j_{2k+1}}({\sigma_{2k+1}})- Z^{i_{2k+1}}({\sigma_{2k+1}}))\}.
\end{equation}
Conditional on $\clf_{\sigma_{2k}}$ and $\{\sigma_{2k+1} <\infty\}$, if $Z^{i_{2k}}({\sigma_{2k}})<Z^{j_{2k}}({\sigma_{2k}})-a$, we have using the memoryless property of exponentials that $E_k \doteq Z^{i_{2k}}(\sigma_{2k+1})-Z^{j_{2k}}({\sigma_{2k+1}}) + a$ is an $Exp(\gamma)$ random variable. Also, note that $Z^{j_{2k+1}}({\sigma_{2k+1}})- Z^{i_{2k+1}}({\sigma_{2k+1}}) = |E_k-a|$ on this event. Otherwise, $Z^{i_{2k}}(\sigma_{2k})-Z^{j_{2k}}(\sigma_{2k})\in [-a,0]$ (as by definition, $Z^{i_{2k}}(\sigma_{2k})\le Z^{j_{2k}}(\sigma_{2k})$). In this case, $\sigma_{2k+1} = \sigma_{2k}$ and thus $Z^{j_{2k+1}}({\sigma_{2k+1}})- Z^{i_{2k+1}}({\sigma_{2k+1}}) \le a$. Combining these observations, we conclude
\begin{equation}\label{zeq}
Z^{j_{2k+1}}({\sigma_{2k+1}})- Z^{i_{2k+1}}({\sigma_{2k+1}}) \le a + |E_k-a|,
\end{equation}
and thus,
\begin{equation}\label{cb2}
    \PP\left( Z^1({\sigma_{2k+2}}) = Z^2({\sigma_{2k+2}}) \, \vert \, \bar \clf_k\right) \ge e^{-\gamma(a + |E_k-a|)}.
\end{equation}
Using this in \eqref{cb1}, we obtain
\begin{equation}\label{cb3}
\PP(\tau \le \sigma_{2k+2} \, \vert \, \clf_{\sigma_{2k}}) \ge \EE\left[\one_{\cle_k} \, e^{-\gamma(a + |E_k-a|)} \, \vert \, \clf_{\sigma_{2k}}\right] = \EE\left[e^{-\gamma(a + |E_k-a|)} \, \mathbb{P}(\cle_k |\mathcal{F}_{\sigma_{2k+1}}) \vert \, \clf_{\sigma_{2k}}\right].
\end{equation}
Next, recalling $w(x) = e^{-\beta x}$, observe that
\begin{align}
    \mathbb{P}(\cle_k |\mathcal{F}_{\sigma_{2k+1}}) &= \mathbb{E} \left[\left. \frac{w\left(Z^1({\sigma_{2k+1}})\vee Z^2({\sigma_{2k+1}})- \int_{\sigma_{2k+1}}^{\sigma_{2k+2}}\gamma^{-1}\theta(s)ds\right)}{w\left(Z^1({\sigma_{2k+1}})\wedge Z^2({\sigma_{2k+1}})- \int_{\sigma_{2k+1}}^{\sigma_{2k+2}}\gamma^{-1}\theta(s)ds\right)}\right|\mathcal{F}_{\sigma_{2k+1}} \right]
    \one_{\{\sigma_{2k+1} <\infty\}}\\
    &= \mathbb{E} \left[\left. \frac{w(Z^1({\sigma_{2k+1}})\vee Z^2({\sigma_{2k+1}}))}{w(Z^1({\sigma_{2k+1}})\wedge Z^2({\sigma_{2k+1}}))}\right|\mathcal{F}_{\sigma_{2k+1}} \right]\one_{\{\sigma_{2k+1} <\infty\}}\\
    &= \exp\{-\beta\left(Z^{j_{2k+1}}({\sigma_{2k+1}})- Z^{i_{2k+1}}({\sigma_{2k+1}})\right)\}\one_{\{\sigma_{2k+1} <\infty\}}\\
    &\geq e^{-\beta(a + |E_k-a|)}\one_{\{\sigma_{2k+1} <\infty\}}, \label{lemcoup.a}
\end{align}
where the last step uses \eqref{zeq}. 
We now claim that
\begin{equation}\label{eq:913}
\mathbb{P} (\sigma_{2k+1} < \infty) = 1, \quad \text{for every } k\geq 0. \end{equation}
We will first prove the result assuming the claim holds and then finally establish the claim.
Using  \eqref{eq:913} and \eqref{lemcoup.a} in \eqref{cb3}, we conclude that,
\begin{equation}
 \PP(\tau \le \sigma_{2k+2} \, \vert \, \clf_{\sigma_{2k}}) \ge p \doteq \EE\left[e^{-(\beta + \gamma)(a + |E_k-a|)}\right]>0,   
\end{equation}
where note that $p$ does not depend on $k$.
This, in conjunction with the strong Markov property, implies for every $k\geq 1$,
\begin{equation}\label{lemcoup.c}
\mathbb{P} (\tau > \sigma_{2k}) \leq (1-p)^{k}.
\end{equation}
This shows that $\mathbb{P} (\tau < \infty) =1$ which, as noted previously, completes the proof of the lemma.
We now return to proving the claim.

We proceed by induction.  Suppose that for some $k \ge 0$, $\sigma_{2k}<\infty$.
Conditionally on $\clf_{\sigma_{2k}}$, on the event $Z^{i_{2k}}({\sigma_{2k}})<Z^{j_{2k}}({\sigma_{2k}})-a$, the evolution of $(Z^1, Z^2)$ after time $\sigma_{2k}$ can be alternatively constructed as follows. At time $t > \sigma_{2k}$, the process $(Z^1, Z^2)$ jumps (at least one of its coordinates changes) at rate $w(Z^{i_{2k}}(t))$. At a jump epoch $t$, both co-ordinates move with probability $w(Z^{j_{2k}}(t))/w(Z^{i_{2k}}(t))$ and only $Z^{i_{2k}}$ moves with probability
$
(w(Z^{i_{2k}}(t))-w(Z^{j_{2k}}(t)))/w(Z^{i_{2k}}(t)).
$
This evolution is continued till time $\sigma_{2k+1}$.
Note that, in the time interval $(\sigma_{2k}, \sigma_{2k+1})$, the probability that a given jump epoch corresponds to a move of $Z^{i_{2k}}$ alone is at least $$\inf_{x\in \RR} \frac{w(x-a)-w(x)}{w(x-a)}=1-e^{-\beta a}. $$
Thus, denoting by $H_k$ the total number of jumps of the process $(Z^1, Z^2)$ (namely jumps for which at least one of the coordinates jumps) on the time interval $(\sigma_{2k}, \sigma_{2k+1})$, the number of jumps where only $Z^{i_{2k}}$ moves, given $H_k=l$, is stochastically lower bounded by $\text{Bin}(l,1-e^{-\beta a})$.

Hence, on $ \{Z^{i_{2k}}({\sigma_{2k}})<Z^{j_{2k}}({\sigma_{2k}})-a\}$, for $l\geq 2$, 
$$
\mathbb{P}(H_k\geq l|\mathcal{F}_{\sigma_{2k}})\leq \mathbb{P}\left(\left. \sum_{m=1}^{\text{Bin}(l,1-e^{-\beta a})} E_m\leq Z^{j_{2k}}(\sigma_{2k})-Z^{i_{2k}}(\sigma_{2k})\right|\mathcal{F}_{\sigma_{2k}}\right)\overset{\text{a.s.}}{\underset{l\to \infty}{\longrightarrow}}0,
$$ 
where $E_m$'s are independent Exp$(\gamma)$ random variables independent of $\mathcal{F}_{\sigma_{2k}}$. On $ \{Z^{i_{2k}}({\sigma_{2k}}) \ge Z^{j_{2k}}({\sigma_{2k}})-a\}$, $\sigma_{2k+1} = \sigma_{2k}$ and hence $\mathbb{P}(H_k\geq 1|\mathcal{F}_{\sigma_{2k}})=0$. Moreover,  for any $l \ge 1$, $\PP(\sigma_{2k+1} = \infty, \, H_k \le l |\mathcal{F}_{\sigma_{2k}})=0$ as $w$ is strictly positive and each jump is almost surely finite. Combining these observations, we have $\PP(\sigma_{2k+1} < \infty |\mathcal{F}_{\sigma_{2k}}) = 1$. Since, for any $k \ge 0$,  the first jump time after $\sigma_{2k+1}$ is $\sigma_{2k+2}$ and $w$ is strictly positive,
$\PP(\sigma_{2k+2} < \infty |\mathcal{F}_{\sigma_{2k+1}}) = 1$ on $\{\sigma_{2k+1} < \infty\}$. Hence, since by the induction hypothesis $\sigma_{2k} <\infty$ we obtain that both $\sigma_{2k+1}$ and $\sigma_{2k+2}$ are finite a.s. as well. By induction the claim in \eqref{eq:913} follows and the proof is complete.
\end{proof}

We can now complete the proof of Theorem \ref{thm.ss.exp}.

\noindent \textbf{Proof of Theorem \ref{thm.ss.exp}.}
From stationarity, for all $k \in \NN$ and $x_1, \ldots, x_k \in \RR$, 
$$\cll(\Bar{F}(t,x_1), \ldots , \Bar{F}(t,x_k))=
\cll(\Bar{F}(0,x_1), \ldots ,\Bar{F}(0,x_k) ),$$
and therefore, by Lemma \ref{lem_coup}, as $t \to \infty$,
\begin{equation}\label{eq:819a}
(F^*(t,x_1), \ldots , F^*(t,x_k)) \Rightarrow (\Bar{F}(0,x_1), \ldots , \Bar{F}(0,x_k)).
\end{equation}
Define $\Psi: \RR \to [0,1]$ as
$$\Psi(z) \doteq \frac{\gamma}{\Gamma(1+\gamma\beta^{-1})} 
\int_{-\infty}^{z} \exp\left( -\gamma y-e^{-\beta y}\right)\, dy, \; z \in \RR.$$
Note that $\Psi$ is strictly increasing and continuous and so $\Psi^{-1}: [0,1] \to \RR$ is as well.
Recall the definition of $\alpha(t)$ from below \eqref{pde}. Then for any $x\in \RR$, 
$\alpha(t) = \Psi^{-1}(F^*(t,x)) - x$.
Thus, by the continuous mapping theorem, for any $x \in \RR$,
\begin{equation}
  \alpha(t) =  \Psi^{-1}(F^*(t,x)) - x \Rightarrow \Psi^{-1}(\bar F(0,x)) - x.
\end{equation}
In particular, taking $x=0$, we have $\alpha(t) \doteq \Psi^{-1}(F^*(t,0)) \Rightarrow
\Psi^{-1}(\bar F(0,0)) \doteq W$.
Writing, $\Psi(x+W) = F^W(x)$,
another application of continuous mapping theorem then says that, for all $k \in \NN$ and $x_1, \ldots, x_k \in \RR$,
\begin{multline}\label{eq:819b}
(F^*(t,x_1), \ldots , F^*(t,x_k))
= (\Psi(x_1+\alpha(t)), \ldots, \Psi(x_k+\alpha(t)))\\
\Rightarrow (\Psi(x_1+W), \ldots, \Psi(x_k+W)) =
(F^W(x_1), \ldots , F^W(x_k)).
\end{multline}
Combining \eqref{eq:819a} with the above,
$F^W(\cdot)$ and $\bar F(0, \cdot)$ have the same distribution on
$\cld(\RR, [0,1])$ (the space of c\`adl\`ag functions from $\RR$ to $[0,1]$, equipped with local uniform topology).
Since $\int_{\RR} x \bar F(0, dx) =0$, we now have that
$\int_{\RR} x F^W(dx) = 0$ as well.
Also, a direct calculation, using the fact that
$$
F^W(x) = \frac{\gamma}{\Gamma(1+\gamma\beta^{-1})} 
\int_{-\infty}^{x+W} \exp\left( -\gamma y-e^{-\beta y}\right)\, dy, 
$$
shows that
$$\int_{\RR} x F^W(dx) = -
\left(W + \beta^{-1} \frac{\Gamma'(\gamma \beta^{-1})}{\Gamma(\gamma \beta^{-1})}\right).
$$
Setting it to $0$, we get $W = - \beta^{-1} \frac{\Gamma'(\gamma \beta^{-1})}{\Gamma(\gamma \beta^{-1})}$.
Substituting the above value of $W$ in the definition of $F^W$, we see that $F^W$ (and therefore $\bar F(0, \cdot)$) is the (non-random) cumulative distribution function associated with the probability measure $\nu^*$ defined in \eqref{eq:838}. Since the stationary solution $\{\bar \nu(t), t \ge 0\}$ was arbitrary, this completes the proof of part (a) of Theorem
\ref{thm.ss.exp}.

Part (b) is now immediate from Theorem \ref{thm.stat.sol}.

Now we consider parts (c) and (d). With $w$ and $\theta$  as in the statement of the theorem, we have that Assumptions \ref{Z}, \ref{ass.w},  \ref{Z2},  \ref{ass.w.LT}, and \ref{eq:wtails} are satisfied. Thus, in particular, Theorems \ref{stat.ex} and \ref{stat.pr} hold in this case. Next, for this setting, Theorem \ref{thm.mainFL} parts (a) and (b) hold if the initial distributions of the particles satisfy Assumptions \ref{ass.t1}, and \ref{ass.char}. From Theorems \ref{stat.ex} and \ref{stat.pr} (see Remark \ref{rem:conds}) it follows that initial distributions given by the stationary distribution sequence $\{\pi_n\}_{n\in \NN}$  satisfy these conditions. Next, from part (b) of the theorem we see that, with this initial distribution, $\mu_n(0)$ converges to $\nu^*$ in probability. Also,  it is easily verified that with this $\nu^*$, Assumption \ref{ass.init} holds. {\di For completeness we give a proof in the Appendix (Section \ref{sec:nustarprops}).} So when the initial distribution is $\pi_n$, parts (a), (b), and (c) of Theorem  \ref{thm.mainFL} hold. Thus, we conclude that there is a unique  $\mu \in \clm$ that solves  the McKean-Vlasov equation \eqref{eq:mfeq1} and satisfies
    $\mu(0) = \nu^*$. This proves the first statement in part (c) of the current theorem. Furthermore, since part (c) of Theorem \ref{thm.mainFL} holds, when the initial distribution of the particle system is $\pi_n$, we have that  $\mu_n$ converges in probability to $\mu$ in $\mathcal{D}([0,\infty):\mathcal{P}_1(\mathbb{R}))$, giving the first statement in part (d) of the current theorem. In particular, this says that $m_n \to m$
    in probability, in $\mathcal{D}([0,\infty):\mathbb{R})$.
    Combining this with the observation that, for $f \in \mbox{Lip}_1$, the empirical measure $\nu_n$ of the centered $n$-particle system satisfies 
    $\langle f, \nu_n(t)\rangle = \langle f(\cdot - m_n(t)), \mu_n(t)\rangle$, and using the fact (from part (b)) that $\nu_n(t)$ converges in probability to $\nu^*$ (as $\mathcal{P}_1(\mathbb{R})$ valued random variables) for every $t \ge 0$, we see that, the limit
    $\mu(t)$ of $\mu_n(t)$ can be explicitly given as 
    $$\mu(t,dx)=\frac{\gamma}{\Gamma(1+\gamma\beta^{-1})}\text{exp}\left[ -\gamma\left(x-\bar m(t)\right) - e^{-\beta\left(x-\bar m(t)\right)}\right]\;dx,$$
    where $\bar m(t) = m(t) + \beta^{-1}\Psi(\gamma \beta^{-1})$.
    Finally, taking $f(x) = x$ in the McKean-Vlasov equation \eqref{eq:mfeq1} we see that
    \begin{align*}
m(t) = \gamma^{-1}\int_0^t \langle w(\cdot - m(s)), \mu(s)\rangle
ds = \gamma^{-1}\int_0^t \langle w, \nu^*\rangle ds = t \gamma^{-1}\langle w, \nu^*\rangle.
    \end{align*}
    A direct calculation, using the form of $\nu^*$ in \eqref{eq:838}, shows that $\gamma^{-1}\langle w, \nu^*\rangle = \beta^{-1} e^{- \Psi(\gamma \beta^{-1})}$. Thus
    $$\bar m(t) = t \beta^{-1} e^{- \Psi(\gamma \beta^{-1})} + \beta^{-1}\Psi(\gamma \beta^{-1}).$$
    This proves the last two statements in part (c) and the last statement in part (d). The result follows.
\hfill
\qedsymbol

\paragraph{Acknowledgements.}
SB was supported in part by the NSF-CAREER award (DMS-2141621).
AB and DI were supported in part by the NSF (DMS-2506010, DMS-2152577). SB and AB were supported by the RTG award (DMS-2134107) from the NSF.

{\di We thank two anonymous referees and an associate editor whose careful reading and detailed suggestions significantly improved the paper.}

\appendix
\renewcommand{\thesection}{Appendix~\Alph{section}.}
\renewcommand{\thesubsection}{\Alph{section}.\arabic{subsection}}

\section{Time derivative of $f^T$.}
\label{sec:appa}
For notational simplicity, we write $f^T \equiv f$ and $w(x-m(t)) = \lambda(t,x)$. Fix $(t,x) \in [0,T) \times \RR$
and $\delta >0$ such that $t+\delta:= T_1 <T$.
We first argue that $s\mapsto f(s,x)$ is continuous on $[0,T_1)$.
Since $m$ is continuous,
$
\bar\lambda_x \doteq \sup_{s\in [0,T]} \lambda(s,x)
 < \infty$.
Also, since $Y^{t,x}(s)\ge x$ for all $s\ge t$, 
\[
\lambda(s, Y^{t,x}(s)) \le \bar\lambda_x \quad \text{for } s\in [t,T].
\]
Thus, for $r \in [0,\delta]$, the number of jumps $N_r$ of $Y^{t,x}$ on $[t,t+r]$ is stochastically dominated by a Poisson random variable with mean $\bar\lambda_x r$. In particular, for some $C>0$ (depending on $t,x$),
\begin{equation*}
\PP(N_r \ge 2) \le Cr^2, \qquad \EE N_r \le Cr.
\end{equation*}
Next, using the Markov property, for any $r\in (0, T-t]$,
\begin{equation}\label{eq:819n}
f(t,x) = \EE\big[f(t+r, Y^{t,x}(t+r))\big].
\end{equation}
Since $Y^{t,x}(t+r)\ge x$, and from \eqref{eq:516n} $f(s,\cdot)$ is Lipschitz on $[x,\infty)$ uniformly in $s\in[0,T_1]$, there exists $C_x<\infty$ such that
\[
|f(s,y_1) - f(s,y_2)| \le C_x |y_1-y_2|, \quad y_1,y_2 \ge x, \; s \in [0,T_1].
\]
Hence, using \eqref{eq:819n} and the stochastic dominance noted above, for small $r$,
\[
|f(t,x) - f(t+r,x)|
\le C_x\, \EE|Y^{t,x}(t+r) - x|
\le C_x\,\EE Z \cdot \EE N_r \le C_x\,\varsigma \bar\lambda_x r,
\]
which shows continuity of $t\mapsto f(t,x)$.

Now we complete the proof of differentiability.
 Since $f$ is  bounded,
\begin{multline*}
f(t,x)
=
\EE\big[f(t+r, Y^{t,x}(t+r))\big]\\
=
\EE\big[f(t+r, Y^{t,x}(t+r));\, N_r=0\big] + \EE\big[f(t+r, Y^{t,x}(t+r));\, N_r=1\big] + O(r^2)\\
= f(t+r, x) \left(1 - \int_t^{t+r} \lambda(s,x) ds\right) +
 \EE\big[f(t+r, x+Z)] \left(\int_t^{t+r} \lambda(s,x) ds\right) + o(r).
\end{multline*}

Thus,
\[
\frac{f(t+r,x) - f(t,x)}{r}
=
- \left(\frac{1}{r} \int_t^{t+r} \lambda(s,x)\, ds\right)\;
\EE\big[f(t+r,x+Z) - f(t+r,x)\big]
+ o(1).
\]
By the continuity of $m(\cdot)$,
\[
\lim_{r\to 0}\frac{1}{r} \int_t^{t+r} \lambda(s,x)\, ds = \lambda(t,x).
\]
Also, for each $z\ge 0$, continuity of $f$ gives
\[
\lim_{r\to 0} (f(t+r,x+z) - f(t+r,x))
=
f(t,x+z) - f(t,x).
\]
Dominated convergence yields
\[
\EE\big[f(t+r,x+Z) - f(t+r,x)\big]
\to
\EE\big[f(t,x+Z) - f(t,x)\big].
\]
Therefore, the right derivative has the form
\[
\partial_t^+ f^T(t,x)
=
- \lambda(t,x)\, E\big[f(t,x+Z) - f(t,x)\big].
\]

A symmetric argument on $[t-r,t]$ gives the left derivative for $t>0$, with the same limit. Hence $\partial_t f(t,x)$ exists and  \eqref{parrep} holds.

{\di
\section{Estimating $\mathbb{E}_0|Y_1(t)-Y_2(t)|$.}
\label{sec:appb}
In this section we prove \eqref{eq:1039n}.
We will use the notation introduced in the proof of Theorem 5.10.
Assume without loss of generality that $\xi$ is non-random.  The general case follows by a simple conditioning argument.
Let $(Y_1,Y_2)$ be as in the proof of Theorem 5.10. 
Define $
h(x,x') \doteq |x-x'|$, for $x, x' \in \RR$ and fix $T>0$.
Then by Dynkin's formula, and boundedness of $w$ on $[\xi-m_T, \infty)$ we have that, for $0\le t \le T$,
\begin{align*}
M_t
\doteq\;
& h(Y_1(t),Y_2(t))
- h(Y_1(0),Y_2(0)) 
 - \int_0^{t}
\tilde{\mathcal{L}}h(s, Y_1(s),Y_2(s))\,ds
\end{align*}
is a $\mathcal{F}_t$-local martingale.

For $x, x' \in \RR$ and $t \in [0,T]$, let $d = x-x'$, $r_1(t,x) = w(x-m_1(t))$, $r_2(t,x') = w(x'-m_2(t))$. Then
\begin{align*}
\tilde{\mathcal{L}}h(s,x,x')
&= (r_1(s,x)-r_2(s,x'))^+\,\EE_Z\big[|d+Z|-|d|\big]
+ (r_1(s,x)-r_2(s,x'))^-\,\EE_Z\big[|d-Z|-|d|\big],
\end{align*}
and therefore
\begin{equation}\label{eq:lhbd}
|\tilde{\mathcal{L}}h(s,x,x')|
\le \varsigma\, |r_1(s,x)-r_2(s,x')|.
\end{equation}
For $t \in [0,T]$, the carr\'e-du-champ operator is given by
\begin{align*}
\Gamma_s h(x,x')
&=
(r_1(s,x)-r_2(s,x'))^+\,\EE_Z\big[(|d+Z|-|d|)^2\big]
+
(r_1(s,x)-r_2(s,x'))^-\,\EE_Z\big[(|d-Z|-|d|)^2\big] \\
&\le \vartheta\, |r_1(s,x)-r_2(s,x')|.
\end{align*}
Thus  noting that, for $t \in [0,T]$,
$
Y_i(s)\ge \xi$,
$m_i(s)\le m_T
$, and $w(Y_i(t)-m_i(t)) \le w(\xi-m_T)$, $i=1,2$,
\[
\EE\!\left[\int_0^{T}
\Gamma_s h(Y_1(s),Y_2(s))\,ds \right]
\le 2\vartheta T\, w(\xi-m_T) < \infty.
\]
Consequently 
\[
M_t
\doteq
|Y_1(t)-Y_2(t)|
-
\int_0^t \tilde{\mathcal{L}}h(s, Y_1(s),Y_2(s))\,ds
\]
is  a martingale.
Taking expectations and using \eqref{eq:lhbd},
\[
\EE |Y_1(t)-Y_2(t)|
\le
\varsigma \int_0^t
\EE\Big[
\big|w(Y_1(s)-m_1(s)) - w(Y_2(s)-m_2(s))\big|
\Big] ds,
\]
which completes the proof of \eqref{eq:1039n}.}
{\di \section{Verification of Properties of $\nu^*$.}
\label{sec:nustarprops}
In this section we verify that $\nu^*$ introduced in \eqref{eq:838} satisfies the properties \eqref{cM1} and \eqref{cM2}.
Let
\[
m_* \doteq \beta^{-1}\Psi(\gamma\beta^{-1}),
\;\;
\alpha \doteq \gamma\beta^{-1}.
\]
Then, using $\Gamma(1+\alpha)=\alpha \Gamma(\alpha)$, the probability measure $\nu^*$ has density
\[
\nu^*(dx)
=
\frac{\beta}{\Gamma(\alpha)}
\exp\!\left(
-\gamma(x-m_*)-e^{-\beta(x-m_*)}
\right)\,dx.
\]

We now argue that $\nu^*$ satisfies \eqref{cM1}.

Fix $a>0$. Consider the change of variables
\[
u=e^{-\beta(x-m_*)}, \;\;
 dx=-\frac{1}{\beta}\frac{du}{u}.
\]
Then
$
e^{-\gamma(x-m_*)}
=
u^\alpha$.
Hence
\begin{align*}
\nu^*(dx)
&=
\frac{\beta}{\Gamma(\alpha)}
u^\alpha e^{-u}\,\frac{1}{\beta}\frac{du}{u}
=
\frac{1}{\Gamma(\alpha)}u^{\alpha-1}e^{-u}\,du.
\end{align*}
So, if $X \sim \nu^*$ then,  the random variable
$
U \doteq e^{-\beta(X-m_*)}$
has the Gamma$(\alpha,1)$ distribution.

Now observe that, with such an $X$,
$
e^{-\beta(X-a)}
=
e^{-\beta(m_*-a)}U$.
Therefore
\begin{align*}
\int_{\mathbb R}\exp\{\delta \beta e^{-\beta(x-a)}\}\,\nu^*(dx)
&=
\EE\Bigl[\exp\{\delta \beta e^{-\beta(X-a)}\}\Bigr]
=
\EE\Bigl[\exp\{\delta \beta e^{-\beta(m_*-a)}U\}\Bigr].
\end{align*}
Since $U\sim\mathrm{Gamma}(\alpha,1)$,  the above integral is finite provided
\[
\delta \beta e^{-\beta(m_*-a)}<1.
\]
This verifies  \eqref{cM1} for $\nu^*$, for the given $a$, for any
$\delta \in (0,\beta^{-1}e^{\beta(m_*-a)})$.

We now show that $\nu^*$ satisfies \eqref{cM2} with $\beta =1$. With $X$ and $U$ as before, we have
\[
|X|^2
\le
2|m_*|^2+\frac{2}{\beta^2}|\log U|^2.
\]
Thus it suffices to show that
$
E|\log U|^2<\infty$.
However, this is clearly true, since
\[
E|\log U|^2
=
\frac{1}{\Gamma(\alpha)}
\int_0^\infty |\log u|^2 u^{\alpha-1}e^{-u}\,du.
\]

This proves that $\nu^*$ satisfies \eqref{cM2} with $\beta =1$.

}
{\footnotesize 
\bibliographystyle{is-abbrv}
\bibliography{references}
}

{\footnotesize 
%
%
%
}

\vspace{\baselineskip}

\noindent{\scriptsize {\textsc{\noindent S. Banerjee, A. Budhiraja, and D. Imon,\newline
Department of Statistics and Operations Research\newline
University of North Carolina\newline
Chapel Hill, NC 27599, USA\newline
email: sayan@email.unc.edu
\newline
email: budhiraj@email.unc.edu
\newline
email: idilshad@unc.edu
\vspace{\baselineskip} } }}

\end{document}